\documentclass[12pt]{amsart}

\usepackage{amsmath} 
\usepackage{amsthm} 
\usepackage{amssymb} 

\usepackage{microtype} 
\usepackage{pinlabel} 

\newcommand{\sqdiamond}{\tikz [x=1.2ex,y=1.85ex,line width=.1ex,line join=round, yshift=-0.285ex] \draw  (0,.5) -- (0.75,1) -- (1.5,.5) -- (0.75,0) -- (0,.5) -- cycle;}%
\renewcommand{\Diamond}{$\sqdiamond$} 

\usepackage[scaled=0.9]{sourcecodepro} 

\usepackage{enumitem}
\usepackage{color}
\usepackage{subfig, caption}

\usepackage{wrapfig}
\usepackage{pdflscape}
\usepackage{array}
\usepackage{tikz-cd}
\usepackage{longtable, booktabs}

\usepackage[hidelinks, pagebackref]{hyperref}

\makeatletter
\define@key{href}{font}{#1}
\makeatother
\usepackage{xpatch}
\newcommand\hrefdefaultfont{\ttfamily}
\xpatchcmd\href{\setkeys{href}{#1}}{\setkeys{href}{font=\hrefdefaultfont,#1}}{}{\fail}

\renewcommand*{\backref}[1]{}
\renewcommand*{\backrefalt}[4]{
  \ifcase #1 
  [No citations.]
  \or [#2]
  \else [#2]
  \fi }

\usepackage{cite}
\makeatletter
\newcommand{\citecomment}[2][]{\citenum{#2}#1\citevar}
\newcommand{\citeone}[1]{\citecomment{#1}}
\newcommand{\citetwo}[2][]{\citecomment[,~#1]{#2}}
\newcommand{\citevar}{\@ifnextchar\bgroup{;~\citeone}{\@ifnextchar[{;~\citetwo}{]}}}
\newcommand{\citefirst}{\@ifnextchar\bgroup{\citeone}{\@ifnextchar[{\citetwo}{]}}}
\newcommand{\cites}{[\citefirst}
\makeatother

\let\originalleft\left
\let\originalright\right
\renewcommand{\left}{\mathopen{}\mathclose\bgroup\originalleft}
\renewcommand{\right}{\aftergroup\egroup\originalright}

\newcommand{\thsup}{{\rm th}}

\newcommand{\calF}{\mathcal{F}}

\newcommand{\calS}{\mathcal{S}}
\newcommand{\calT}{\mathcal{T}}
\newcommand{\calU}{\mathcal{U}}
\newcommand{\calV}{\mathcal{V}}

\newcommand{\RR}{\mathbb{R}}

\newcommand{\ZZ}{\mathbb{Z}}

\newcommand{\from}{\colon} 
 
\newcommand{\cross}{\times}

\newcommand{\cover}[1]{{\widetilde{#1}}}

\newcommand{\closure}[1]{{\overline{#1}}}

\newcommand{\bdy}{\partial} 

\input{header_article.tex}

\theoremstyle{plain}
\newtheorem{XXXtheoremQED}[equation]{Theorem} 
  {\pushQED{\qed}\begin{XXXtheoremQED}}
  {\popQED\end{XXXtheoremQED}}

\newcommand{\fakeenv}{} 

\newenvironment{restate}[2]  
{ 
 \renewcommand{\fakeenv}{#2} 
 \theoremstyle{plain} 
 \newtheorem*{\fakeenv}{#1~\ref{#2}} 
 \begin{\fakeenv}
}
{
 \end{\fakeenv}
}

\newenvironment{restated}[2]  
{ 
 \renewcommand{\fakeenv}{#2} 
 \theoremstyle{definition} 
 \newtheorem*{\fakeenv}{#1~\ref{#2}} 
 \begin{\fakeenv}
}
{
 \end{\fakeenv}
}

\usepackage{stmaryrd}
\usepackage{stackengine}
\usepackage{lineno}

\newcommand{\bigon}{ \scalebox{0.6}{\rotatebox{90}{$\llparenthesis$} } }
\newcommand{\crimp}{{\succ\!\prec}}

\stackMath

\newcommand{\bigonx}[1]{\stackon[1pt]{#1}{\kern0.32em\bigon}}

\newsavebox{\FigEightSnappy}
\newsavebox{\FigEightVeer}
\newsavebox{\FigEightSisSnappy}
\newsavebox{\FigEightSisVeer}
\newsavebox{\FourTetExSnappy}
\newsavebox{\FourTetExVeer}
\newsavebox{\NonFiberedSnappy}
\newsavebox{\NonFiberedVeer}
\newsavebox{\BigExSnappy}
\newsavebox{\BigExVeer}

\begin{lrbox}{\FigEightSnappy}
\verb+m004+
\end{lrbox}
\begin{lrbox}{\FigEightVeer}
\verb+cPcbbbiht_12+
\end{lrbox}
\begin{lrbox}{\FigEightSisSnappy}
\verb+m003+
\end{lrbox}
\begin{lrbox}{\FigEightSisVeer}
\verb+cPcbbbdxm_10+
\end{lrbox}
\begin{lrbox}{\FourTetExSnappy}
\verb+m203+
\end{lrbox}
\begin{lrbox}{\FourTetExVeer}
\verb+eLMkbcddddedde_2100+
\end{lrbox}
\begin{lrbox}{\NonFiberedSnappy}
\verb+s227+
\end{lrbox}
\begin{lrbox}{\NonFiberedVeer}
\verb+gLLAQbecdfffhhnkqnc_120012+
\end{lrbox}
\begin{lrbox}{\BigExSnappy}
\verb+m115+
\end{lrbox}
\begin{lrbox}{\BigExVeer}
\verb+fLLQccecddehqrwjj_20102+
\end{lrbox}

\newcommand{\rlt}{\rotatebox{90}{$<$}}

\title[To dynamic pairs]{From veering triangulations \\ to dynamic pairs}
\author[Schleimer and Segerman]{Saul Schleimer and Henry Segerman}
\date{\today}

\begin{document}

\begin{abstract}
From a transverse veering triangulation (not necessarily finite) we produce a canonically associated dynamic pair of branched surfaces.
As a key idea in the proof, we introduce the shearing decomposition of a veering triangulation.
\end{abstract}




\maketitle

\section{Introduction}

Mosher, inspired by work of (and with) Christy~\cite[page~5]{Mosher96}
and Gabai~\cite[page~4]{Mosher96}, introduced the idea of a \emph{dynamic pair of branched surfaces}.
These give a combinatorial method for describing and working with pseudo-Anosov flows in three-manifolds.
Very briefly, suppose that $\Phi$ is such a flow.
Then $\Phi$ admits a transverse pair of foliations $F^\Phi$ and $F_\Phi$, called weak stable and weak unstable, respectively.
Carefully splitting both to obtain laminations, and then carefully collapsing, gives a \emph{dynamic pair} of branched surfaces $B^\Phi$ and $B_\Phi$.
These again intersect transversely and have other combinatorial properties that allow us to reconstruct $\Phi$ (up to orbit equivalence).

Agol, while investigating the combinatorial complexity of mapping tori, introduced the idea of a \emph{veering triangulation}~\cite[Main~construction]{Agol11}.
For any pseudo-Anosov monodromy $\phi$ he provides a canonical periodic splitting sequence of stable train-tracks $(\tau^\phi_i)$.
This gives a branched surface $B^\phi$ in the mapping torus $M(\phi)$.
Equally well, the splitting sequence of unstable tracks $(\tau^i_\phi)$ gives rise to the branched surface $B_\phi$.

More generally, even when not layered~\cite[Section~4]{HRST11}, a veering triangulation $\calV$ admits \emph{upper} and \emph{lower branched surfaces} $B^\calV$ and $B_\calV$, obtained by gluing together standard pieces within each tetrahedron (\refsec{BranchedSurfaces}).
Our main result is that these may be isotoped into \emph{draped position} and they then form a dynamic pair.

\begin{restate}{Theorem}{Thm:DynamicPair}
Suppose that $\calV$ is a transverse veering triangulation.
In draped position, the upper and lower branched surfaces $B^\calV$ and $B_\calV$ form a dynamic pair;
this position is canonical.
Furthermore, if $\calV$ is finite then draped position is produced algorithmically in polynomial time.
Finally, the dynamic train-track $B^\calV \cap B_\calV$ has at most a quadratic number of edges.
\end{restate}

Suppose that $S$ is a surface.
We say that train-tracks $\tau^*$ and $\tau_*$ on $S$ are \emph{dual} if they are transverse and no component of $S - (\tau^* \cup \tau_*)$ is a bigon.
Here is a consquence of \refthm{DynamicPair}.

\begin{restate}{Corollary}{Cor:Dual}
There is an algorithm that, given a surface $S$ and a pseudo-Anosov homeomorphism $f \from S \to S$, produces a (canonical) splitting/folding sequence of dual train tracks in $S$ that realise $f$.
\end{restate}

\begin{remark}
The resulting sequence of dual tracks is canonical and does not require \emph{bigon tracks}~\cite[page~191]{PennerHarer92}.
Thus \refcor{Dual} improves upon the analysis given
in~\cite[pages~207--208]{PennerHarer92}.
\end{remark}

Before giving an outline of the proof of \refthm{DynamicPair}, we highlight the main difficulty.

\begin{remark}
\label{Rem:Difficulty}
Suppose that $B^\calV$ and $B_\calV$ are in \emph{normal position} within each tetrahedron.
This is locally determined, and any other locally determined position can be obtained from normal position by local moves.
In normal position, the branched surfaces may coincide on large regions, spanning many tetrahedra; see \refsec{BranchedSurfaces}.
Such a region may contain a vertical M\"obius band.
If so, then any small isotopy making $B^\calV$ and $B_\calV$ transverse produces ``bad'' components of $M - (B^\calV \cup B_\calV)$.
We give more details in \refsec{PushOff} and an example in \reffig{Fail}.
\end{remark}

A more global procedure is thus required.
To guide this, in \refsec{NewCombinatorics} we define the \emph{shearing decomposition} associated to $\calV$.
This decomposes $M$ into solid tori (and possibly solid cylinders in the non-compact case).

\begin{restate}{Theorem}{Thm:ShearingDecomposition}
Suppose that $\calV$ is a veering triangulation (not necessarily transverse or finite).
Then there is a canonical shearing decomposition of $M$ associated to $\calV$.
\end{restate}

Here is a consequence.

\begin{restate}{Corollary}{Cor:PADecomposition}
Suppose that $\Phi$ is a pseudo-Anosov flow on $N$ without perfect fits.
Suppose that $M = N^\circ$ is the result of drilling out the singular orbits of $\Phi$.
Let $\Phi^\circ = \Phi|M$ be the restriction of $\Phi$ to $M$.
Let $\calV$ be the veering triangulation associated to $\Phi^\circ$.
Then the canonical shearing decomposition of $M$ (associated to $\calV$) factors $\Phi^\circ$ as a product of fractional Dehn twists.
\end{restate}

\begin{remark}
The shearing decomposition (\refthm{ShearingDecomposition}) is also used by Tsang in~\cite[Corollary~1.2]{Tsang22a}.
He proves that a transitive pseudo-Anosov flow on a closed three-manifold admits a Birkhoff section with at most two boundary components on orbits of the flow.
\end{remark}

With \refthm{ShearingDecomposition} in hand, we give a sequence of coordinatisations inside of the shearing regions.
In particular each shearing region is foliated by \emph{horizontal cross-sections};
see \refdef{CrossSection}.
In Sections~\ref{Sec:StraighteningShrinking}, \ref{Sec:Parting}, and~\ref{Sec:Draping} we give a sequence of pairs of isotopies to improve the positioning of $B^\calV$ and $B_\calV$ relative to each other and relative to the horizontal cross-sections.
In each cross-section these isotopies appear to be movements of a train-track.
We ``split'' track-cusps forward and then ``graphically'' isotope branches.
These happen both in space and in time.

\begin{remark}
\label{Rem:SemiLocal}
Our construction is ``semi-local'' in the following sense.
Suppose that $\calV$ and $\calV'$ are veering triangulations of manifolds $M$ and $M'$.
Suppose that $U$ and $U'$ are isomorphic red components (maximal connected unions of crimped red shearing regions).
Then the isomorphism carries the dynamic pair for $\calV$ to that of $\calV'$ (as intersected with $U$ and $U'$).
\end{remark}

Finally, in \refsec{DynamicPair} we verify that $B^\calV$ and $B_\calV$, in their final \emph{draped position} form a dynamic pair.

\subsection{Other work}

After Mosher's monograph~\cite{Mosher96}, other appearances of dynamic pairs in the literature include the following.
Fenley~\cite[Section~8]{Fenley99a} gives an exposition of various examples due to Mosher and proves that leaves of the resulting weak stable and unstable foliations have the continuous extension property.
Given a uniform one-cochain, Coskunuzer~\cite[Main~Theorem]{Coskunuzer06} follows Calegari~\cite[Theorem~6.2]{Calegari01} in producing various laminations, which are collapsed to give a dynamic pair.
Calegari~\cite[Sections~6.5 and~6.6]{Calegari07} gives a useful exposition of dynamic pairs and their relation to pseudo-Anosov flows.
In particular see his version of examples of Mosher~\cite[Example~6.49]{Calegari07}.

Closely related to our overall program is recent work of Agol and Tsang~\cite[Theorem~5.1]{AgolTsang22}.
Starting from a veering triangulation (with appropriate framing), they construct a pseudo-Anosov flow on the filled manifold.
They do not use dynamic pairs;
instead they apply a different construction of Mosher~\cite[Proposition~2.6.2]{Mosher96}.
They identify and remove \emph{infinitesimal cycles}, which are similar in spirit to the vertical M\"obius bands mentioned above.
Their construction relies on making certain choices, so it is not canonical.
Also, it is not clear if the resulting pseudo-Anosov flow recovers the original veering triangulation.

A very recent and very dramatic result concerning dynamic pairs appears in the work of Landry and Tsang~\cite{LandryTsang23}.
In addition to their other results, they carry out the base case of the construction promised by (but not given in) Mosher's monograph~\cite[Section~II]{Mosher96}: that is, they produce ``proper'' dynamic pairs inside of the compactified mapping torus of any given endperiodic map (if the mapping torus is atoroidal).
In fact, Landry and Tsang mainly work with just one (unstable) branched surface.
They then use it and the flow graph to produce the other (stable) branched surface.

\subsection{Future work}

This is the fourth paper in a series of five~\cite{SchleimerSegerman20, SchleimerSegerman21, FrankelSchleimerSegerman22} providing a dictionary between veering triangulations (framed with appropriate surgery coefficients) and pseudo-Anosov flows without perfect fits.
\refthm{DynamicPair} together with Mosher's work~\cite[Theorem~3.4.1]{Mosher96} gives one direction of the dictionary.
In the fifth paper we will prove that the two ``translation directions'' of the dictionary are in fact inverses.
To prepare for this, in \refapp{Rectangles} we use \refthm{DynamicPair} to show that the ``leaf space'' of the resulting pseudo-Anosov flow has maximal rectangles corresponding to (via the construction given in~\cite[Section~5.8]{SchleimerSegerman21}) the original veering tetrahedra.
This will imply that the map from a veering triangulation to a flow and back again, is the identity.

\subsection*{Acknowledgements}

We thank Lee Mosher for enlightening conversations regarding dynamic pairs.
We thank Chi Cheuk Tsang for his many helpful comments on several early drafts.
Henry Segerman was supported in part by National Science Foundation grants DMS-1708239 and DMS-2203993.

\section{Triangulations, train-tracks, and branched surfaces}
\label{Sec:Triangulations}

\subsection{Ideal triangulations}

Suppose that $M$ is a connected three-manifold without boundary.
Suppose that $\calT$ is a triangulation: a collection of oriented model tetrahedra and a collection of face pairings.
(We do not assume here that $\calT$ is finite, nor do we assume that the face pairings respect the orientations of the tetrahedra.)
We say that $\calT$ is an \emph{ideal triangulation} of $M$ if the quotient $|\calT|$, minus its zero-skeleton, is homeomorphic to $M$~\cite[Section~4.2]{Thurston78}.
In this case, the degree of each edge of $\calT$ is necessarily finite.
See \reffig{VeerFigEight} for an example.

\begin{figure}[htbp]
\includegraphics[width=0.7\textwidth]{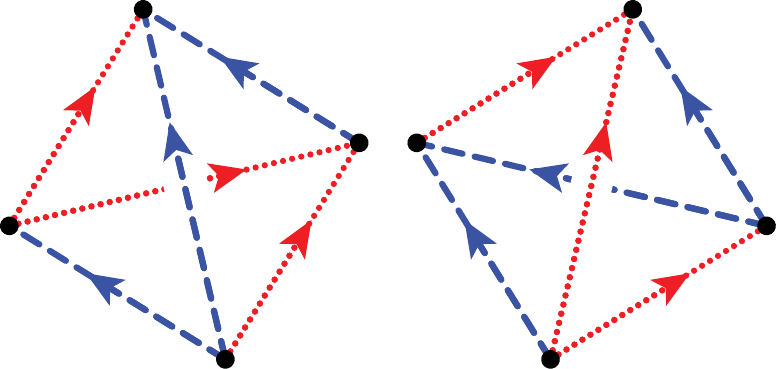}
\caption{An ideal triangulation of the complement of the figure-eight knot in the three-sphere.
Each edge is equipped with a colour -- red (dotted) or blue (dashed) -- and an orientation.
These determine the face pairings.
The flattening (into the plane) makes the triangulation taut and transverse.
Note that the taut structure and the orientation determine the veering structure and thus the colours.}
\label{Fig:VeerFigEight}
\end{figure}

A model tetrahedron $t$ is \emph{taut} if every model edge is equipped with a dihedral angle of zero or $\pi$, subject to the requirement that the sum of the three dihedral angles at any model vertex is $\pi$.
It follows that there are exactly two model edges in $t$ with angle $\pi$;
these do not share any vertex of $t$.
The remaining four model edges, with angle zero, are called \emph{equatorial}.
A taut tetrahedron can be flattened into the plane with its equatorial edges forming its boundary;
see \reffig{VeerFigEight}.
A taut tetrahedron $t$ contains an \emph{equatorial square}:
a disk properly embedded in $t$ whose boundary is the four equatorial edges.
An ideal triangulation $\calT$ of $M$ is a \emph{taut triangulation} if the model tetrahedra are taut and, for every edge $e$ in $|\calT|$, the sum of the dihedral angles of the models of $e$ is $2\pi$~\cite[Definition~1.1]{HRST11}.

A taut model tetrahedron $t$ is \emph{transverse} if every model face is equipped with a co-orientation (in or out of $t$), subject to the requirement that co-orientations agree across model edges of dihedral angle $\pi$ and disagree across model edges of dihedral angle zero.
See \reffig{TransverseTet}.
A taut triangulation $\calT$ of $M$ is a \emph{transverse taut triangulation} if every model tetrahedron is transverse taut and, for every face $f$ in $|\calT|$, the associated face pairing preserves the co-orientations of the two model faces~\cites[Definition~1.2]{HRST11}[page~370]{Lackenby00}.

\begin{figure}[htbp]
\centering
\subfloat[Co-orientations and angles in a transverse taut tetrahedron.]{
\labellist
\small\hair 2pt
\pinlabel $0$ at 20 130
\pinlabel $0$ at 240 120
\pinlabel $0$ at 135 27
\pinlabel $0$ at 135 217
\pinlabel $\pi$ at 125 140
\pinlabel $\pi$ at 125 87
\endlabellist
\includegraphics[width=0.35\textwidth]{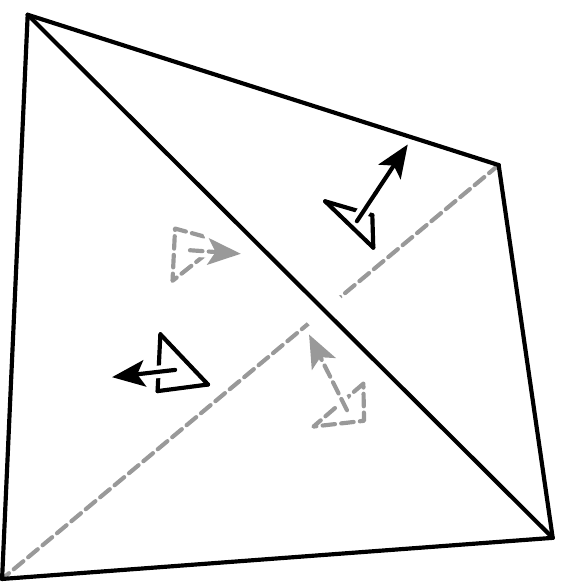}
\label{Fig:TransverseTet}
}
\qquad \qquad
\subfloat[Co-orientations around edges can be deduced from the co-orientations on the faces of the model tetrahedra.]{
\includegraphics[width=0.4\textwidth]{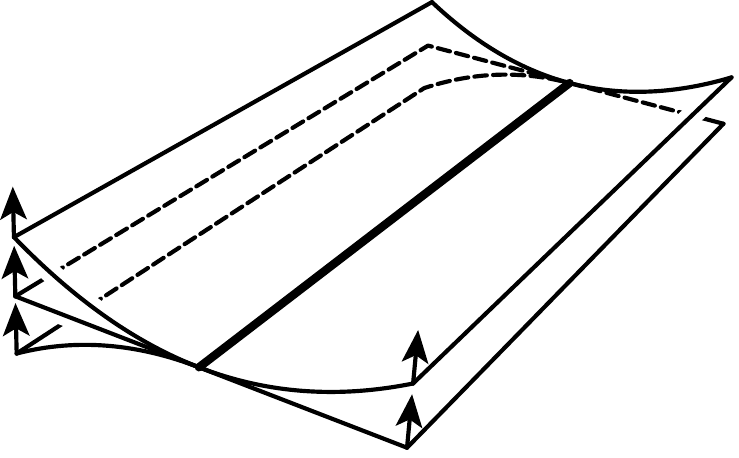}
\label{Fig:TransverseEdge}
}
\caption{}
\label{Fig:Transverse}
\end{figure}

Recall that the model tetrahedra are oriented.
A taut model tetrahedron $t$ is \emph{veering} if every model edge is equipped with a colour, red or blue, subject to the following.
\begin{itemize}
\item
The colours on the equatorial edges alternate between red and blue.
\item
Viewing any model face (from the outside of the tetrahedron) the non-equatorial edge is followed, in anticlockwise order, by a red equatorial edge.
\end{itemize}
Suppose that $t$ is a veering tetrahedron.
If the two non-equatorial edges of $t$ are both red (blue) then we call $t$ a red (blue) \emph{fan tetrahedron}.
If the two non-equatorial edges	of $t$ have different colours then we call $t$ a \emph{toggle tetrahedron}.
See \reffig{UpperGluingAutomaton} for all four of the possible veering model tetrahedra.
Note that the taut structure and the orientation of $t$ determine the colouring of its equatorial edges.

Suppose now that $\calT$ is a transverse taut triangulation of $M$.
Then $\calT$ is a \emph{transverse veering} triangulation if there is a colouring of the edges of $|\calT|$ making all of the model tetrahedra veering~\cites[Main~construction]{Agol11}[Definition~1.3]{HRST11}.
By the previous paragraph, when such a colouring exists it is unique.
Also, if the colouring exists then the orientations of the model tetrahedra of $\calT$ induce an orientation on $M$.
For an example of a transverse veering triangulation, see \reffig{VeerFigEight}.
The possible gluings between the various kinds of veering tetrahedra are recorded in \reffig{UpperGluingAutomaton}.

\begin{figure}[htb]
\centering
\subfloat[Upper tracks.]{
\includegraphics[width=0.47\textwidth]{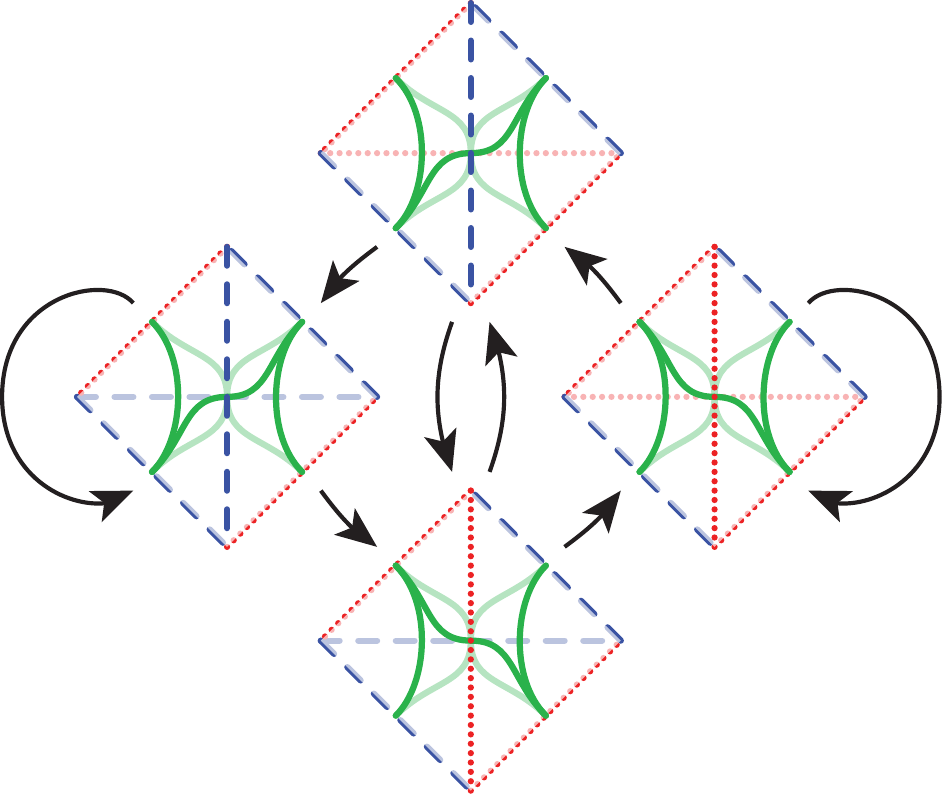}
\label{Fig:UpperGluingAutomaton}
}
\thinspace
\subfloat[Lower tracks.]{
\includegraphics[width=0.47\textwidth]{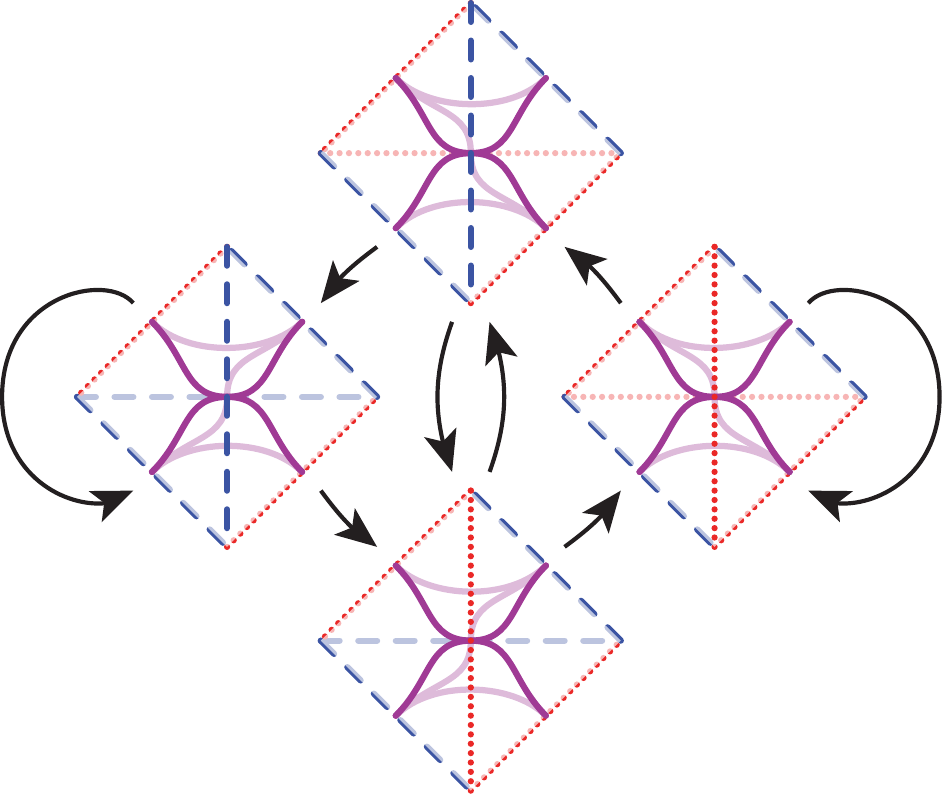}
\label{Fig:LowerGluingAutomaton}
}
\caption{
In both subfigures, above and below we have toggle tetrahedra while left and right we have, respectively, blue and red fan tetrahedra.
A black arrow indicates a possible gluing from an upper face of the initial tetrahedron to a lower face of the terminal.
Note that fan tetrahedra of different colours never share a face.
Finally, inside each tetrahedron $t$ on the left (right) we draw the branched surface $B^t$ ($B_t$).
}
\label{Fig:GluingAutomaton}
\end{figure}

\subsection{Train-tracks}
\label{Sec:TrainTracks}

For background on train-tracks generally we refer to~\cite{PennerHarer92} as well as~\cite[Chapter~8]{Thurston78}.
Suppose that $\calV$ is a transverse veering triangulation.
Suppose that $f$ is a face of $\calV$.
Let $t$ and $t'$ be the tetrahedra above and below $f$, respectively.
We now define the \emph{upper} and \emph{lower train-tracks} $\tau^f$ and $\tau_f$ in $f$.
The upper track $\tau^f$ consists of one switch at each edge midpoint and two branches perpendicular to the edges~\cite[Figure~11]{Agol11}.
The two branches meet only at the switch on the non-equatorial edge of $t$ (the tetrahedron \emph{above} $f$).
The lower track $\tau_f$ is defined similarly, except the two branches now meet at the switch on the non-equatorial edge
of $t'$ (the tetrahedron \emph{below} $f$).
We call the region immediately between the two branches, adjacent to the shared switch, a \emph{track-cusp}.
See \reffig{UpperLowerTracks}.
Starting in \refsec{StraighteningShrinking} we also discuss slightly more general train-tracks in slightly more general surfaces.

\begin{figure}[htbp]
\subfloat[The two taut tetrahedra adjacent to a face $f$ lie above and below $f$.]{
\labellist
\small
\hair 2pt
\pinlabel $f$ at 88 71
\endlabellist
\includegraphics[height=1in]{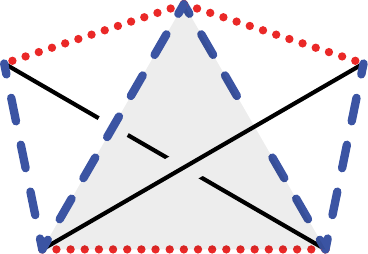}
\label{Fig:TwoTautTetrahedra}
}
\quad
\subfloat[The upper train-track $\tau^f$.]{
\labellist
\small
\hair 2pt
\pinlabel $\tau^f$ [tr] at 65 42
\endlabellist
\includegraphics[height=1in]{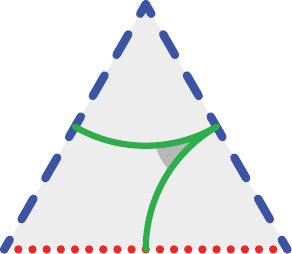}
\label{Fig:UpperTrack}
}
\quad
\subfloat[The lower train-track $\tau_f$.]{
\labellist
\small
\hair 2pt
\pinlabel $\tau_f$ [tl] at 77 37
\endlabellist
\includegraphics[height=1in]{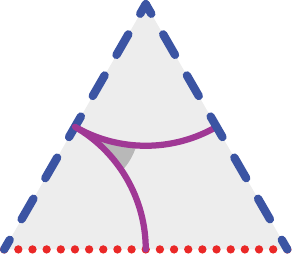}
\label{Fig:LowerTrack}
}
\caption{}
\label{Fig:UpperLowerTracks}
\end{figure}

\subsection{Branched surfaces}
\label{Sec:BranchedSurfaces}

For background on branched surfaces generally we refer to~\cite[Section~6.3]{Calegari07}.

Suppose that $M$ is an oriented three-manifold equipped with a transverse veering triangulation $\calV$.
Suppose that $t$ is a model tetrahedron of $\calV$.
The four faces $(f_i)$ of $t$ contain their upper tracks $\tau^i$.
These form a graph in $\bdy t$, transverse to the edges of $t$.
This graph bounds a normal quadrilateral and also a pair of normal triangles~\cite[page~4]{Gordon01}.
We arrange matters so that the three normal disks meet only along the lower faces of $t$, so that they are transverse to the equatorial square of $t$, and so that the union of the normal disks is a branched surface, denoted $B^t$.
We call $B^t$ the \emph{upper branched surface} in $t$.
We define $B_t$, the \emph{lower branched surface} in $t$ similarly, using the lower tracks $\tau_i$ instead of the upper.
We finally define $B^\calV = \cup_t B^t$ and $B_\calV = \cup_t B_t$ to be the \emph{upper} and \emph{lower branched surfaces} for $\calV$ in \emph{normal position}.
See \reffig{NormalUpperBranchedSurface}.

We define the \emph{horizontal branched surface} $B(\calV)$ to be the union of the faces of $\calV$.
Here we isotope the faces of $\calV$, near their boundaries, to meet the one-skeleton of $\calV$ as shown in \reffig{TransverseEdge}.
The horizontal branched surface $B(\calV)$ is \emph{taut}~\cite[page~374]{Lackenby00};
this explains the name \emph{taut ideal triangulation}.

The \emph{branch locus} $\Sigma = \Sigma(B)$ of a branched surface $B$ is the subset of non-manifold points.
Each component of $B - \Sigma$ is a \emph{sector} of $B$.
For $B^\calV$ (and $B_\calV$) a generic point of its branch locus is locally adjacent to exactly three sectors.
The \emph{vertices} of $B^\calV$ (and $B_\calV$) are the points of the branch locus locally meeting six sectors.
Note that, since we have removed the zero-skeleton from $|\calV|$, the horizontal branched surface $B(\calV)$ has no vertices~\cite[page~371]{Lackenby00}.

We may move $B^\calV$ into \emph{dual position} by applying a small upward isotopy of $B^\calV$.
See \reffig{DualUpperBranchedSurface}.
This done, every tetrahedron $t$ of $\calV$ contains exactly one vertex of $B^\calV$ and every face of $\calV$ contains exactly one point of the branch locus.
We arrange matters so that the vertex of $B^\calV$ in $t$ is halfway between the lower edge and the equatorial square of $t$.
Applying a small downward isotopy to $B_\calV$ produces its dual position.
We again arrange matters so that the vertex of $B_\calV$ in $t$ is halfway between the upper edge (of $t$) and the equatorial square.

\begin{remark}
\label{Rem:Dual}
In dual position, both $B^\calV$ and $B_\calV$ are isotopic to the dual two-skeleton of $\calV$.
See~\cite[Remark~6.4]{FrankelSchleimerSegerman22}.
\end{remark}

\begin{figure}[htb]
\centering
\subfloat[Normal position.]{
\label{Fig:NormalUpperBranchedSurface}
\includegraphics[width=0.47\textwidth]{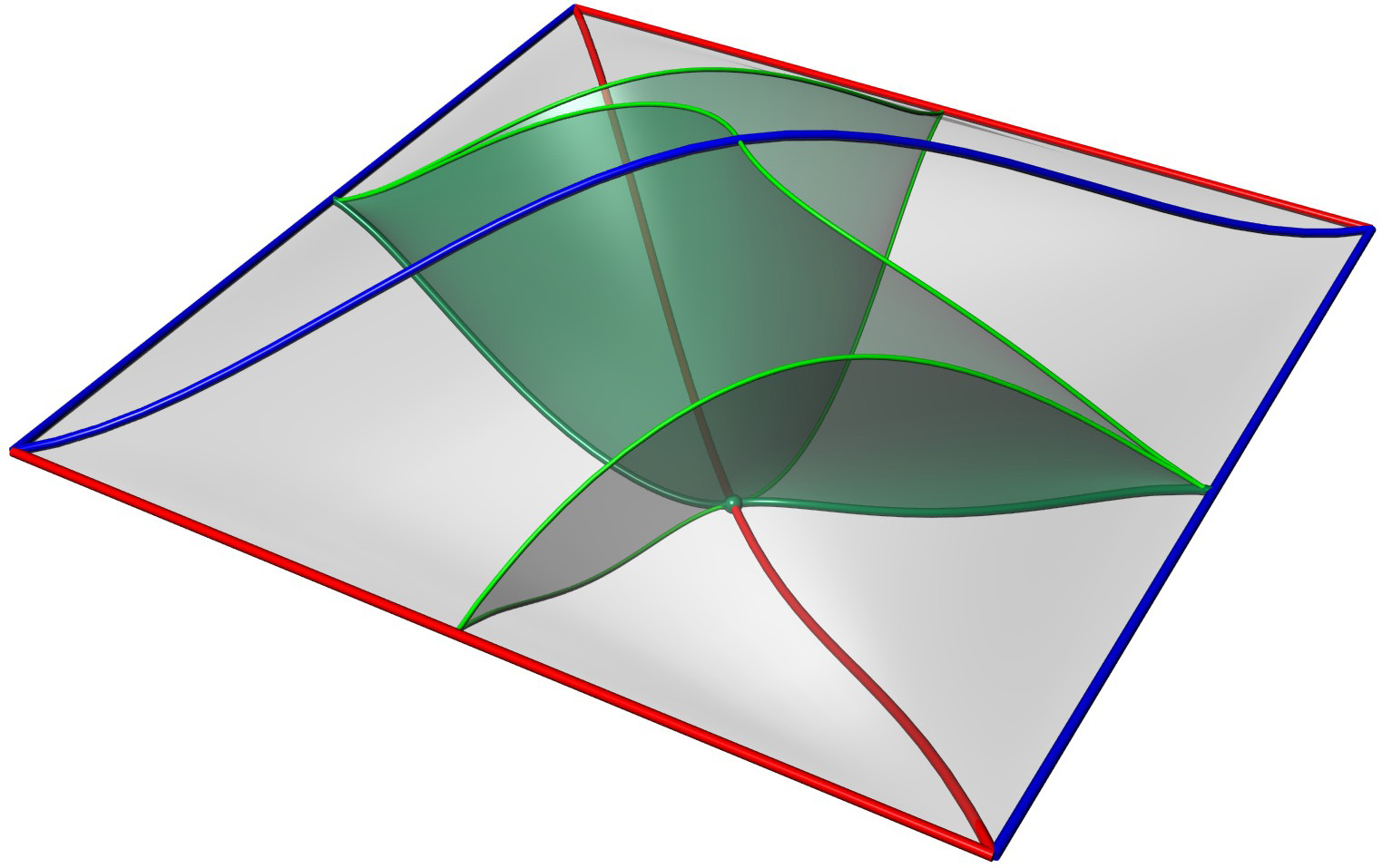}
}
\subfloat[Dual position.]{
\label{Fig:DualUpperBranchedSurface}
\includegraphics[width=0.47\textwidth]{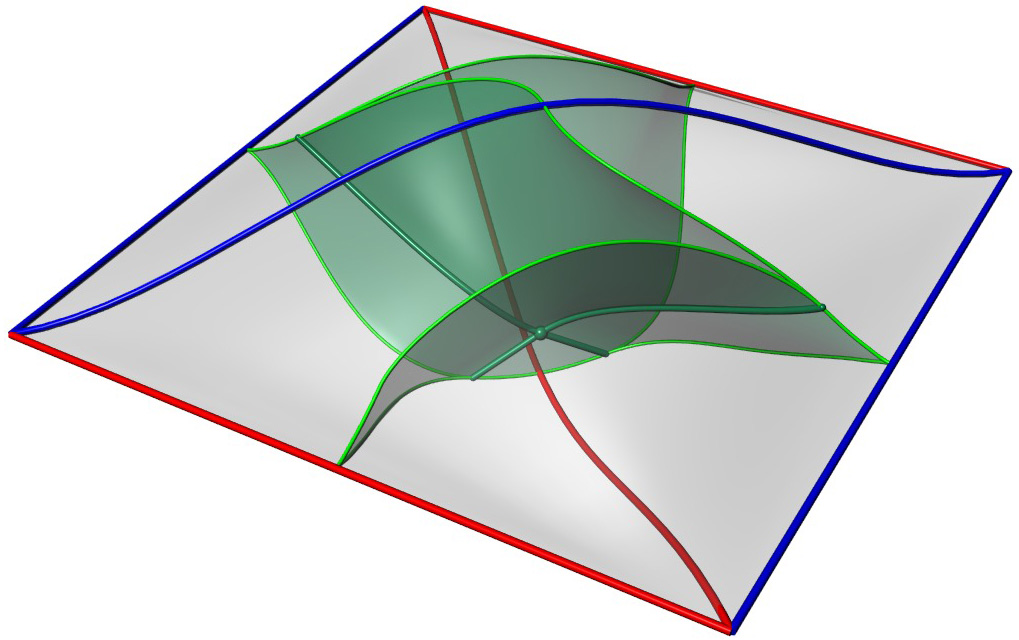}
}
\caption{Two positions of the upper branched surface in a tetrahedron.}
\label{Fig:UpperBranchedSurface}
\end{figure}

Suppose that $M$ be an oriented three-manifold equipped with a transverse veering triangulation $\calV$.
Suppose that $\cover{M}$ is the universal cover of $M$.
Suppose that $\cover{B}^\calV$ and $\cover{B}_\calV$ are the preimages of $B^\calV$ and $B_\calV$ in $\cover{M}$.
We now restate~\cite[Corollary~6.12]{FrankelSchleimerSegerman22}.

\begin{lemma}
\label{Lem:DualMeetsToggles}
In the universal cover $\cover{M}$, with $\cover{B}^\calV$ and $\cover{B}_\calV$ in dual position, every subray of every branch line of $\cover{B}^\calV$ and of $\cover{B}_\calV$ meets toggle tetrahedra.  \qed
\end{lemma}

\section{Dynamics}
\label{Sec:Dynamics}

Suppose that $M$ is a connected oriented three-manifold equipped with a riemannian metric.
We loosely follow Mosher~\cite[page~36]{Mosher96} for the next two definitions.
See also~\cite[Figure~1.3]{Christy93}.

\begin{definition}
\label{Def:DynamicVectorField}
A \emph{dynamic vector field} $X$ on $M$ is a smooth non-vanishing vector field.
If $M$ has boundary then we require $X$ to be tangent to the boundary of $M$.
\end{definition}

The dynamic vector field $X$ gives us a local notion of \emph{upwards} (the direction of $X$).
Note that in our setting $X$ is smooth while in Mosher's it is necessarily at best continuous.

\begin{definition}
\label{Def:DynamicBranchedSurface}
Suppose that $M$ is a three-manifold and $X$ is a dynamic vector field.
Suppose that $B^* \subset M$ is a properly embedded branched surface.
We say that $B^*$ is a \emph{stable dynamic branched surface} with respect to $X$ if it has the following properties.
\begin{itemize}
\item
For any point $p$ of any sector of $B^*$, there is a tangent to the sector, at $p$, which makes a positive dot product with $X$.
Choosing the largest such gives a vector field $X^*$ on $B^*$.
Integrating $X^*$ gives the \emph{upwards semi-flow}.
\item
The semi-flow $X^*$ is transverse to the branch locus of $B^*$ and points from the side with fewer sheets to the side with more.
\item
The semi-flow $X^*$ is never orthogonal to the branch locus.
\end{itemize}
The only change needed to define an \emph{unstable dynamic branched surface} $B_*$ is that $X_*$ points from the side with more sheets to the side with fewer.
\end{definition}

Note that Mosher requires his original vector field $X$ be tangent to $B^*$.
However, we wish to use just one vector field with respect to which both branched surfaces $B^\calV$ and $B_\calV$ are dynamic (but do not yet form a \emph{dynamic pair}).

\begin{remark}
The terms stable and unstable come from the fact that any pseudo-Anosov flow $\Phi$ leads to a pair of two-dimensional foliations~\cites[page~226]{Calegari07}[Section~3.1]{Mosher96}.
These are the \emph{weak stable} foliation $F^\Phi$ and the \emph{weak unstable} foliation $F_\Phi$.
If $L$ is a leaf of $F^\Phi$ then any two flow lines $\ell$ and $\ell'$ in $L$ are asymptotic in forward time.
Finally, the stable branched surface $B^\Phi$ carries $F^\Phi$.
\end{remark}

Suppose that $t$ is one of the four model transverse veering tetrahedra (shown in \reffig{GluingAutomaton}).
Let $X_t$ be a non-vanishing vector field in $t$ with the following properties.
\begin{itemize}
\item The vector field $X_t$ is orthogonal to each face of $t$.
\item Each orbit of $X_t$ connects a lower face of $t$ with an upper face.
\item The branched surfaces $B^t$ and $B_t$ (in dual position) are stable and unstable with respect to $X_t$.
\end{itemize}

Now suppose that $\calV$ is a transverse taut veering triangulation.
We define $X_\calV$ by gluing together the vector fields $X_t$.

\begin{corollary}
\label{Cor:DualDynamic}
The upper and lower branched surfaces $B^\calV$ and $B_\calV$ (in dual position) are, with respect to $X_\calV$, stable and unstable dynamic branched surfaces.  \qed
\end{corollary}

\section{Dynamic pairs}
\label{Sec:DynamicPairs}

In this section, loosely following Mosher~\cite[page~52]{Mosher96}, we give our definition of a \emph{dynamic pair} of branched surfaces.
Morally, these mimic the stable and unstable foliations of a pseudo-Anosov flow.
The transversality of the foliations implies that the branched surfaces should be transverse and should not have various kinds of ``bigon regions''.

We make this precise and then discuss the main difficulties in proving \refthm{DynamicPair}.

\subsection{Complementary components}

Suppose that $M$ is a connected oriented three-manifold equipped with a riemannian metric.
Suppose that $X$ is a dynamic vector field on $M$, as in \refdef{DynamicVectorField}.
Suppose that $B^*$ and $B_*$ are stable and unstable dynamic surfaces with respect to $X$.
Suppose further that $B^*$ and $B_*$ meet transversely.

\begin{definition}
\label{Def:PinchedTetrahedron}
Suppose that $C$ is a component of $M - (B^* \cup B_*)$.
We call $C$ a \emph{pinched tetrahedron} if $\closure{C}$ (the closure taken in the induced path metric) has the following properties.
\begin{itemize}
\item
$\closure{C}$ is a three-ball.
\item
$\bdy \closure{C}$ consists of four triangles, called the \emph{faces} of $C$.
\item
Each pair of faces meets in a simple arc;
these six arcs form the one-skeleton of a tetrahedron.
\item
When mapped to $M$, two faces are sent to $B^* - B_*$ and two are sent to $B_* - B^*$.
\item
The two faces sent to $B^* - B_*$ meet in a single arc of (the preimage of) the branch locus of $B^*$;
a similar property holds for the two faces sent to $B_* - B^*$. \qedhere
\end{itemize}
\end{definition}

See \reffig{PinchedTet} for a picture of an embedded pinched tetrahedron.

\begin{figure}[htbp]
\centering
\subfloat[A pinched tetrahedron. ]{
\includegraphics[height=5in]{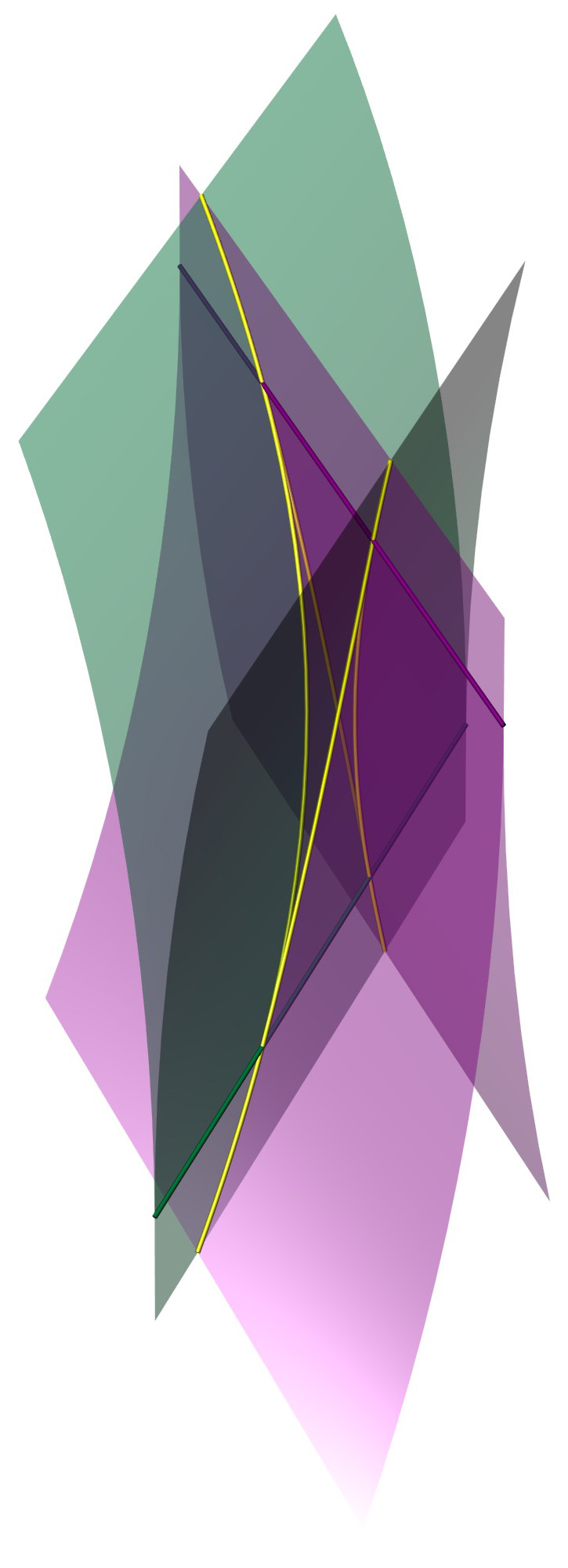}
\label{Fig:PinchedTet}
}
\qquad
\subfloat[Birth, life, and death. ]{
\includegraphics[height=5in]{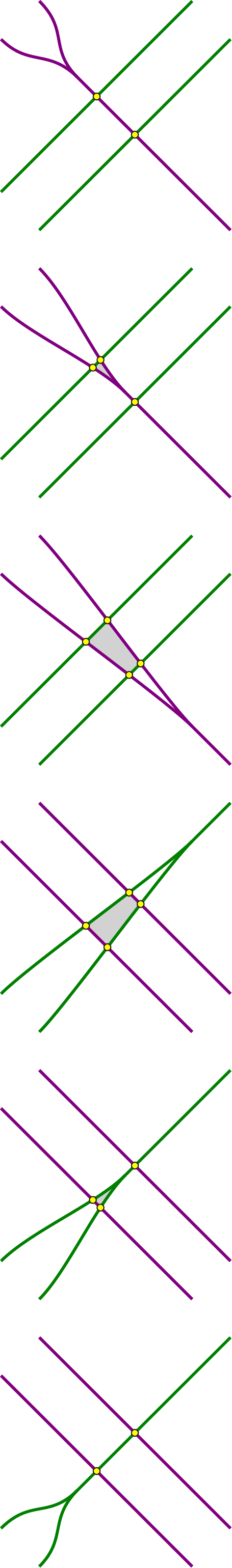}
\label{Fig:LifeAndDeath}
}
\caption{The right shows horizontal slices through the left. See also Figures~2.2,~2.3 and~2.6 of~\cite{Mosher96}.}
\label{Fig:PinchedTetBoth}
\end{figure}

\begin{definition}
We call a foliation of (a three-dimensional region of) $M$ \emph{horizontal} if it is everywhere transverse to $X$, to $B^*$, to $B_*$, and to $B^* \cap B_*$.
\end{definition}

The birth, life, and death of a pinched tetrahedron play out on the two-dimensional leaves of such a horizontal foliation.

\begin{definition}
\label{Def:LifeAndDeath}
Suppose that $C$ is a pinched tetrahedron for $B^*$ and $B_*$.
Since $C$ is simply connected for the purposes of this definition we may assume that $M$ is simply connected.
Suppose that  $(H_s)_{s\in\RR}$ is a horizontal foliation of a ball in $M$ containing $C$.
As $s$ increases, we move upwards, in the direction of $X$.
Let $\tau^s = H_s \cap B^*$ and $\tau_s = H_s \cap B_*$ be the \emph{upper} and \emph{lower tracks} in $H_s$ respectively.
Let $C_s = C \cap H_s$.
There are four special times $a < b < c < d$ as follows.
\begin{itemize}
\item
At time $a$, the pinched tetrahedron $C$ is born as a track-cusp of $\tau^a$ crosses an arc of $\tau_a$, moving forwards.
\item
For $s \in (a, b)$, the disk $C_s$ is a \emph{green trigon}.
It has two sides and a track-cusp in $\tau^s$.
The remaining side is in $\tau_s$.
\item
At time $b$, the track-cusp of $\tau^b$ (on the same branch line) crosses another arc of  $\tau_b$, still moving forward.
\item
For $s \in (b, c)$, the disk $C_s$ is a \emph{quadragon}.
Its four sides alternate between $\tau^s$ and $\tau_s$.
\item
At time $c$, a track-cusp of $\tau_c$ crosses an arc of $\tau^c$, moving backwards.
\item
For $s \in (c, d)$, the disk $C_s$ is a \emph{purple trigon}.
It has two sides and a track-cusp in $\tau_s$.
The remaining side is in $\tau^s$.
\item
At time $d$, the pinched tetrahedron $C$ dies as the track-cusp of $\tau_d$ (on the same branch line) crosses an arc of $\tau^d$, still moving backwards. \qedhere
\end{itemize}
\end{definition}
\reffig{LifeAndDeath} shows $\tau^s \cup \tau_s$ for six representative generic heights.

\begin{definition}
\label{Def:DynamicShell}
Suppose that $C$ is a component of $M - (B^* \cup B_*)$.
We call $C$ a \emph{dynamic torus shell} if it is homeomorphic to $T^2 \cross (0,1)$.
We require that for any $\epsilon$ the image of $T^2 \cross (0, \epsilon)$ in $C$ is an end of $M$.
The other end of $C$ must have closure (in the path metric) homeomorphic to $T^2 \cross (1/2, 1]$.
The boundary of this must meet, in alternating fashion, annuli from $B^* - B_*$ and from $B_* - B^*$.  The annuli from $B^* - B_*$ are the \emph{stable annuli} of $C$ while the annuli from $B_* - B^*$ are the \emph{unstable annuli} of $C$.  See \reffig{TorusShell}.

\begin{figure}[htb]
\centering
\includegraphics[width=0.5\textwidth]{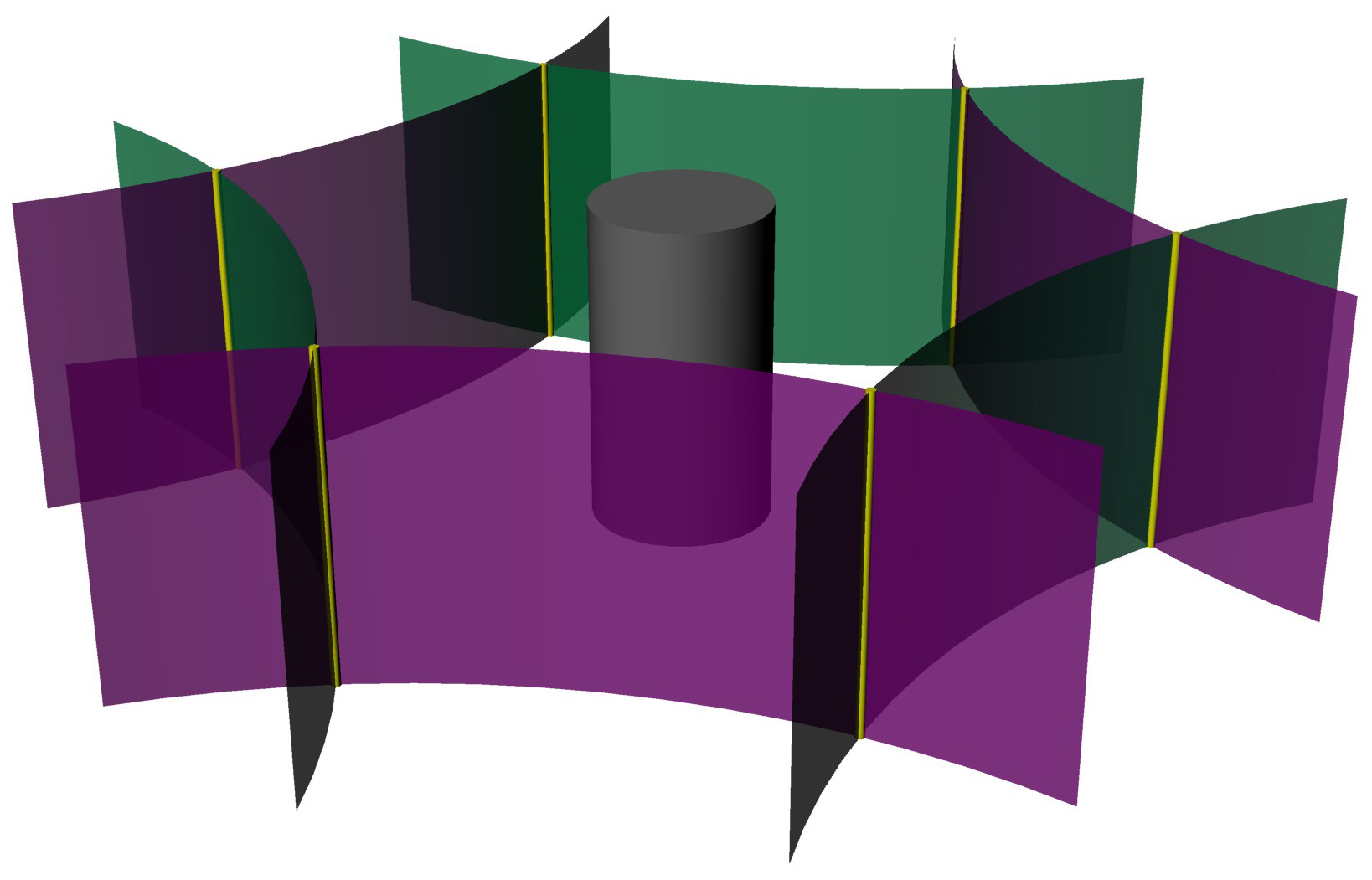}
\caption{A section of an annulus or torus shell.
The central grey cylinder represents an end of $M$.}
\label{Fig:TorusShell}
\end{figure}

Taking infinite degree covers of a dynamic torus shell yields (periodic)
\emph{dynamic annulus shells} and \emph{dynamic plane shells}.
More generally, such shells need not be periodic.
This occurs only when neither $B^*$ nor $B_*$ is compact.
There are two types of dynamic annulus shell.
In one, the frontier is a bi-infinite alternating union of stable and unstable annuli.
In the other, the frontier is a finite alternating union of stable and unstable \emph{strips} of the form $[0,1] \cross \RR$.
There is only one type of dynamic plane shell.
Here the frontier is a bi-infinite alternating union of stable and unstable strips.
Thus for any dynamic shell $C$, the components of the frontier (after cutting along $B^* \cap B_*$) are stable and unstable annuli or strips.
These annuli or strips are the \emph{faces} of the dynamic shell $C$.
\end{definition}

\begin{definition}
Suppose that $C$ is a complementary region.
Suppose that $F$ is an unstable face of $\closure{C}$.
The components of $F - B^{(1)}_*$ are called the \emph{subfaces} of $F$.
The subfaces of a stable face are defined similarly.
\end{definition}

\begin{definition}
\label{Def:Upwards}
A smooth path $\alpha$ in $B^*$ is \emph{upwards} if it always crosses the branch locus of $B^*$ from the side with fewer sheets to the side with more.
We make a similar definition for \emph{downwards} paths in $B_*$.
\end{definition}

We are now equipped to give our definition of a dynamic pair.

\begin{definition}
\label{Def:DynamicPair}
We say that $B^*$ and $B_*$ form a \emph{dynamic pair} if they satisfy the following.
\begin{enumerate}
\item
\label{Itm:Transversality}
(Transversality): The branched surfaces $B^*$ and $B_*$ intersect transversely.
\item
\label{Itm:Components}
(Components): Every component of $M - (B^* \cup B_*)$ is either a pinched tetrahedron or a dynamic shell.
\item
\label{Itm:Transience}
(Transience): For every component $F$ of $B_* - B^*$ there is an unstable face $F' \subset F$ of some dynamic shell so that $F'$ is a sink for all upwards rays in $F$.
The corresponding statement holds for downwards paths in $B^* - B_*$.
\item
\label{Itm:Separation}
(Separation):
No distinct pair of subfaces of dynamic shells are glued in $M$.  \qedhere
\end{enumerate}
\end{definition}

\begin{definition}
\label{Def:DynamicTrainTrack}
Suppose that $B^*$ and $B_*$ form a dynamic pair.
Then their \emph{dynamic train-track} is the intersection $B^\calV \cap B_\calV$.
\end{definition}

\begin{remark}
Dynamic shells (and pinched tetrahedra) may meet each other or themselves along intervals of the dynamic train-track.
For an example, see \reffig{FinalPositionFig8Sibling}.
\end{remark}

Our \refdef{DynamicTrainTrack} is taken directly from~\cite[page~54]{Mosher96}.
Note that our \refdef{DynamicPair} is more restrictive than Mosher's~\cite[page~52]{Mosher96}.
Mosher allows dynamic shells to meet along subfaces while we do not.
He also allows solid torus pieces.
We do not require (or allow) solid torus pieces in the cusped case.
In the closed case they are necessary;
we deal with this as follows.

\begin{remark}
Suppose that $\gamma$ is a curve in $T$, a torus boundary component of $M$.
Suppose that $C$ is a dynamic torus shell containing $T$.
Suppose that $\gamma$ meets the dynamic train-track (projected from $C$ to $T$) at least four times.
Then Dehn filling $M$ along $\gamma$ converts $C$ into a solid torus piece $C(\gamma)$.
So, after filling all dynamic torus shells we arrive at the closed case.
\end{remark}

\subsection{The naive push-off}
\label{Sec:PushOff}

As noted in \refrem{Difficulty}, in normal position the branched surfaces $B^\calV$ and $B_\calV$ coincide in (at least) all normal quadrilaterals in all fan tetrahedra.
To try and fix this, we choose orientations on the edges of $\calV^{(1)}$.
We then push $B_\calV$ slightly in the directions of the edge orientations and pull $B^\calV$ slightly against them.
We call this pair of isotopies the \emph{naive push-off}.
In Examples~\ref{Exa:Win} and \ref{Exa:Fail} we see that this sometimes works and sometimes does not.
The way in which the naive push-off fails is instructive;
as noted in \refrem{Difficulty} the obstructions are non-local.

\begin{example}
\label{Exa:Win}
In \reffig{Win} we draw an exploded view of the veering triangulation on the figure-eight knot complement, as previously introduced in \reffig{VeerFigEight}.
The upper and lower train-tracks are the result of intersecting $B^\calV$ and $B_\calV$ with the faces and equatorial squares of the veering tetrahedra.
The naive push-off keeps the dynamic branched surfaces dual to the horizontal branched surface $B = B(\calV)$ and makes them transverse to each other.
Note that no pair of train-tracks in any horizontal cross-section form a bigon.

In fact, the push-off makes $B^\calV$ and $B_\calV$ into a dynamic pair.
Parts~\refitm{Transversality} and~\refitm{Separation} of \refdef{DynamicPair} can be checked cross-section by cross-section.
For part~\refitm{Components}, we have labelled cross-sections through the four pinched tetrahedra~\textsc{a}$_i$ through~\textsc{d}$_i$, with subscripts indicating the vertical order.
One must check that as we move vertically through the manifold, the sections through the regions assemble to form pinched tetrahedra (see \reffig{LifeAndDeath}) and dynamic torus shells.
Note that in \reffig{Win}, as we move downwards from the middle section to the bottom of the two tetrahedra, regions~\textsc{c}$_1$ and~\textsc{d}$_1$ go from being quadragons to being green trigons (and then disappear), but the trigonal stage is not shown.
Part~\refitm{Transience} must be checked by hand.
\end{example}

\begin{example}
\label{Exa:Fail}
Consider the veering triangulation on the figure-eight knot sibling, shown in \reffig{Fail}.
Again we push $B_\calV$ in the direction of the orientations of the edges; this time bigons appear in several of the horizontal cross-sections.
In fact there is \emph{no} orientation of the edges that leads to a dynamic pair via the naive push-off.
This is because the \emph{mid-surface} (\refdef{MidSurface}) for the figure-eight knot sibling is not transversely orientable.
Further details are given in \refrem{Fail}.
\end{example}

\begin{landscape}
\begin{figure}[htbp]
\subfloat[The figure-eight knot complement with the veering triangulation \usebox{\FigEightVeer}.]{
\labellist
\footnotesize\hair 2pt
\pinlabel \textsc{b}$_4$ at 104 407
\pinlabel \textsc{a}$_6$ at 51 402
\pinlabel \textsc{d}$_2$ at 104 384
\pinlabel \textsc{a}$_2$ at 83 373
\pinlabel \textsc{b}$_4$ at 48 358
\pinlabel \textsc{a}$_6$ at 103 354

\pinlabel \textsc{b}$_3$ at 104 254
\pinlabel \textsc{a}$_5$ at 67 238
\pinlabel \textsc{b}$_7$ at 112 230
\pinlabel \textsc{d}$_1$ at 103 213
\pinlabel \textsc{b}$_3$ at 48 201
\pinlabel \textsc{a}$_5$ at 103 197
\pinlabel \textsc{a}$_1$ at 85 184
\pinlabel \textsc{c}$_3$ at 70 184

\pinlabel \textsc{b}$_2$ at 99 104
\pinlabel \textsc{a}$_4$ at 52 100
\pinlabel \textsc{b}$_6$ at 82 87
\pinlabel \textsc{b}$_2$ at 48 48
\pinlabel \textsc{a}$_4$ at 99 42
\pinlabel \textsc{c}$_2$ at 70 30

\pinlabel \textsc{a}$_4$ at 205 407
\pinlabel \textsc{b}$_6$ at 257 402
\pinlabel \textsc{c}$_2$ at 205 384
\pinlabel \textsc{b}$_2$ at 227 373
\pinlabel \textsc{a}$_4$ at 261 358
\pinlabel \textsc{b}$_6$ at 206 354

\pinlabel \textsc{a}$_3$ at 205 254
\pinlabel \textsc{b}$_5$ at 242 238
\pinlabel \textsc{a}$_7$ at 201 230
\pinlabel \textsc{c}$_1$ at 208 213
\pinlabel \textsc{a}$_3$ at 261 201
\pinlabel \textsc{b}$_5$ at 206 197
\pinlabel \textsc{b}$_1$ at 224 183
\pinlabel \textsc{d}$_3$ at 239 186

\pinlabel \textsc{a}$_2$ at 210 104
\pinlabel \textsc{b}$_4$ at 257 100
\pinlabel \textsc{a}$_6$ at 227 87
\pinlabel \textsc{a}$_2$ at 261 48
\pinlabel \textsc{b}$_4$ at 210 42
\pinlabel \textsc{d}$_2$ at 239 32
\endlabellist
\includegraphics[width = 0.65\textwidth]{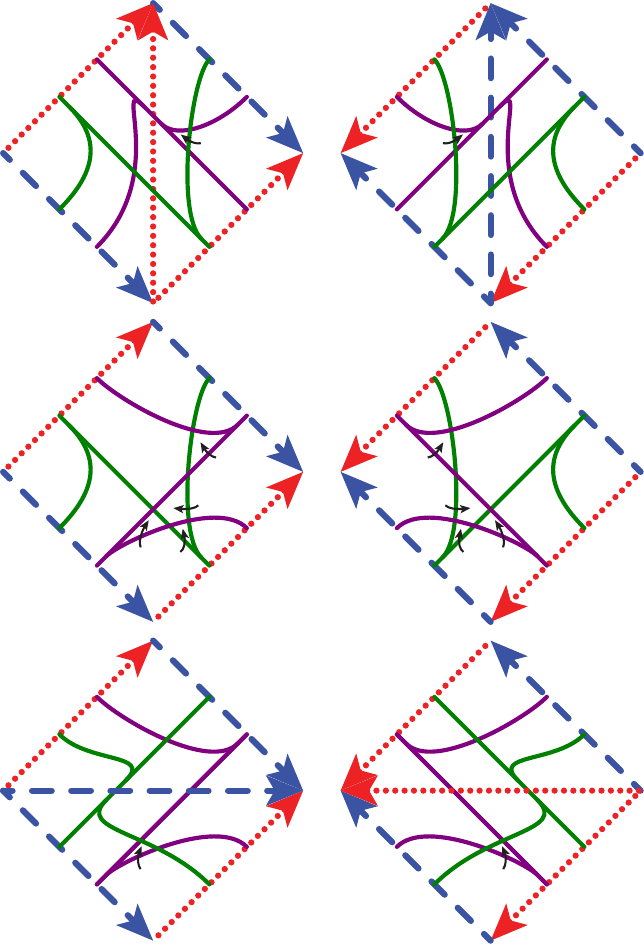}
\label{Fig:Win}
}
\qquad
\subfloat[The figure-eight knot sibling with the veering triangulation \usebox{\FigEightSisVeer}.]{
\includegraphics[width = 0.65\textwidth]{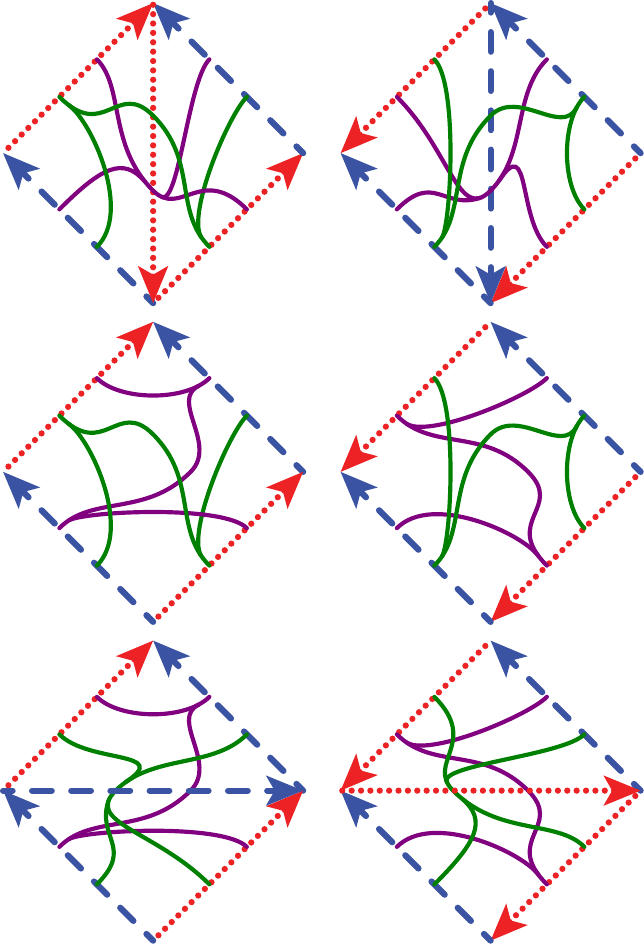}
\label{Fig:Fail}
}
\caption{Each column shows three slices:
the upper and lower faces of, and an equatorial square through, one of the tetrahedra.
In the figure-eight knot complement, $B^\calV$ (green) and $B_\calV$ (purple) have been naively pushed off each other to produce a dynamic pair.
In the sibling, this does not work.}
\label{Fig:WinFail}
\end{figure}
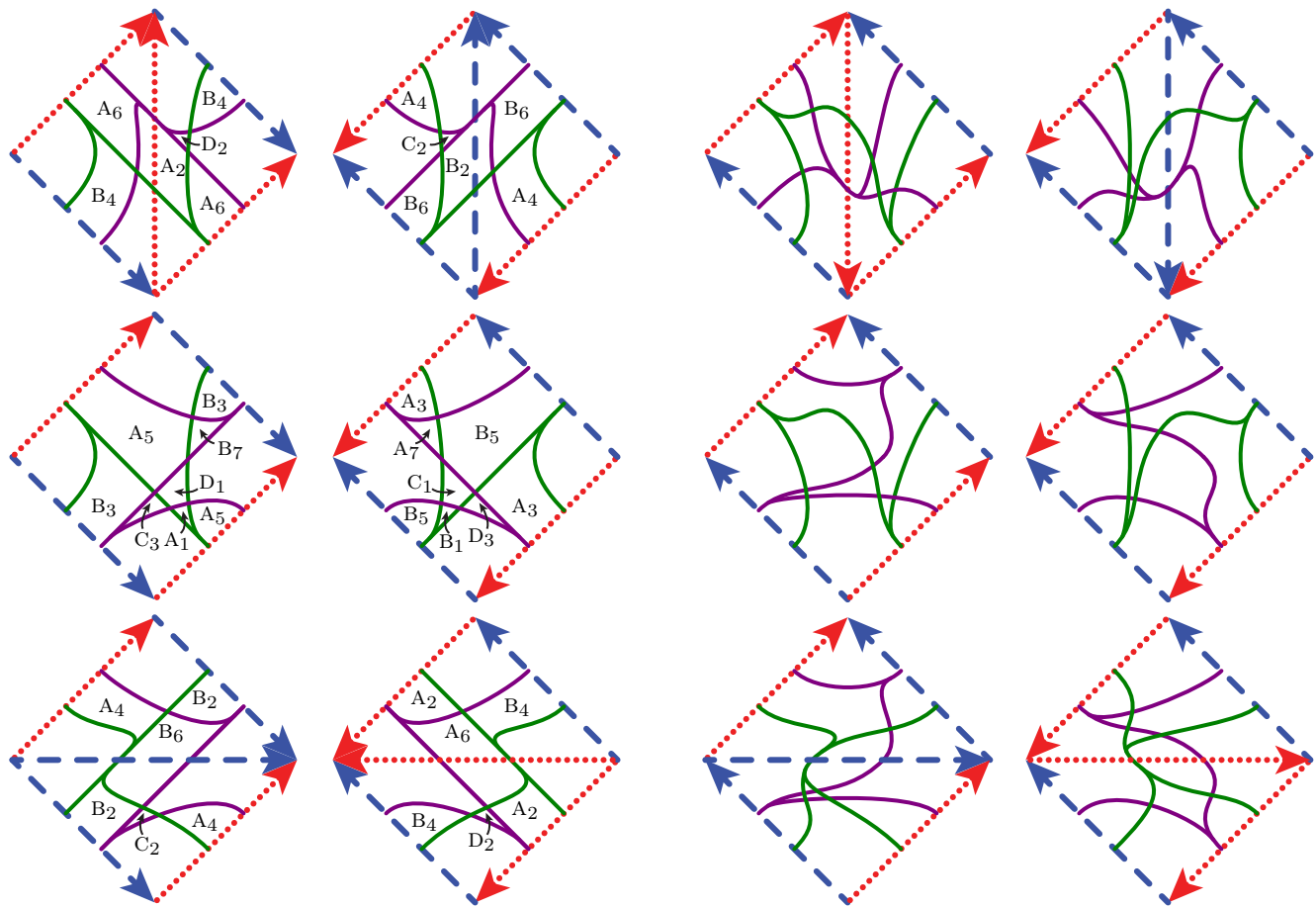
\end{landscape}

Even when it works, the naive push-off requires making a choice.
Thus the resulting dynamic pair is not canonically associated to the initial veering triangulation.

Instead of isotoping the branched surfaces horizontally, we will ``split'' them closer to the stable and unstable foliations of the hypothesised pseudo-Anosov flow.
To define these isotopies, we define various decompositions of $M$ (in Sections~\ref{Sec:NewCombinatorics} and~\ref{Sec:BigonCoords}).
We then describe a sequence of isotopies, of each of $B^\calV$ and $B_\calV$, through the new decompositions (in Sections~\ref{Sec:StraighteningShrinking},~\ref{Sec:Parting}, and~\ref{Sec:Draping}).

\section{Shearing regions, mid-bands, and the mid-surface}
\label{Sec:NewCombinatorics}

Here we give a decomposition of a veering triangulation into a canonical collection of \emph{shearing regions}.
Each of these is either a solid torus or a solid cylinder.
We use these to define the \emph{mid-bands} and the \emph{mid-surface}.

\subsection{Shearing regions}
\label{Sec:ShearingDecomposition}

\begin{definition}
\label{Def:IdealSolid}
An \emph{ideal solid torus} $U$ is a solid torus $D^2 \cross S^1$, together with a non-empty discrete subset of $(\bdy D^2) \cross S^1$, called the \emph{ideal points} of $U$.
We define an \emph{ideal solid cylinder} in similar fashion, replacing $S^1$ by $\RR$.
\end{definition}

\begin{definition}
\label{Def:TautSolid}
A \emph{taut solid torus (cylinder)} $U$ is an ideal solid torus (cylinder) decorated with a \emph{paring locus} $\gamma$ containing all of the ideal points of $U$.
The paring locus is a multi-curve $\gamma = \gamma(U)$ meeting every meridional disk exactly twice.
There is at least one ideal point on every component of $\gamma$.
A taut solid torus $U$ has a \emph{mid-band} $B$; this is either an annulus or a M\"obius band, properly embedded in $U$ and disjoint from $\gamma$.
The mid-band of a taut solid cylinder is instead a strip, $[0,1] \cross \RR$.
In all cases, the mid-band intersects every meridional disk in a single arc and every boundary compression of the mid-band intersects the pairing locus.
\end{definition}

\begin{definition}
\label{Def:TransverseSolid}
A \emph{transverse taut solid torus (cylinder)} $U$ is a taut solid torus (cylinder) where $\bdy U - \gamma$ has two components, called the \emph{upper} and \emph{lower boundaries} $\bdy^+ U$ and $\bdy^- U$.
These are equipped with transverse orientations that point out of and into $U$, respectively.
Note that all taut solid cylinders can be equipped with such an orientation.
\end{definition}

In a transverse taut solid torus the mid-band is necessarily an annulus.
In a taut solid cylinder it is necessarily a strip.

\begin{definition}
\label{Def:ShearingRegion}
A \emph{shearing region} $U$ is a taut solid torus or cylinder, together with a \emph{colour} (red or blue) and a squaring of $\bdy U - \gamma$, with vertices at the ideal points.
All edges contained in the paring locus $\gamma$ are the opposite colour to $U$ and are called \emph{longitudinal}.
All edges not in $\gamma$ are the same colour as $U$ and are called \emph{helical}.
The helical edges form a helix that spirals right or left (as $U$ is red or blue);
the helix meets every meridional disk exactly once, transversely.
We give the mid-band $B \subset U$ the same colour as $U$ itself.
\end{definition}

See \reffig{Continuing} for the local model of a red shearing region.

\begin{figure}[htbp]
\subfloat[Start with veering tetrahedra.]{
\includegraphics[width=0.48\textwidth]{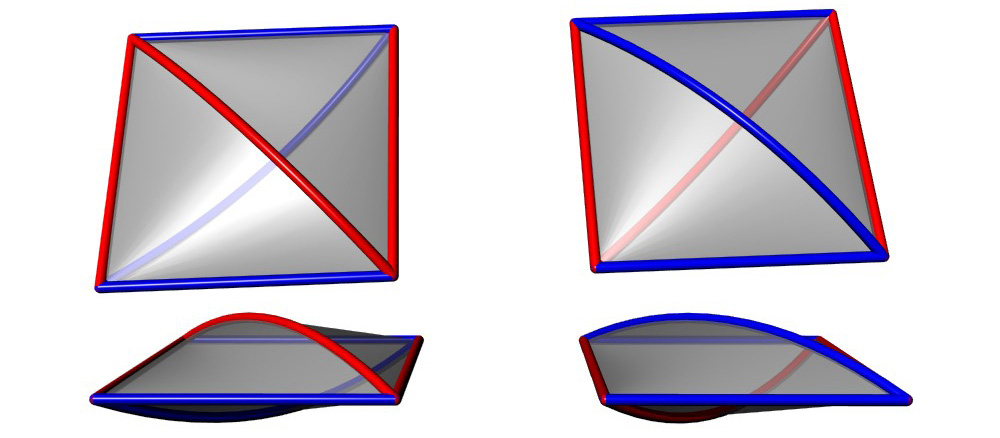}
}
\subfloat[Cut into half-tetrahedra, select red half-tetrahedra.]{
\includegraphics[width=0.48\textwidth]{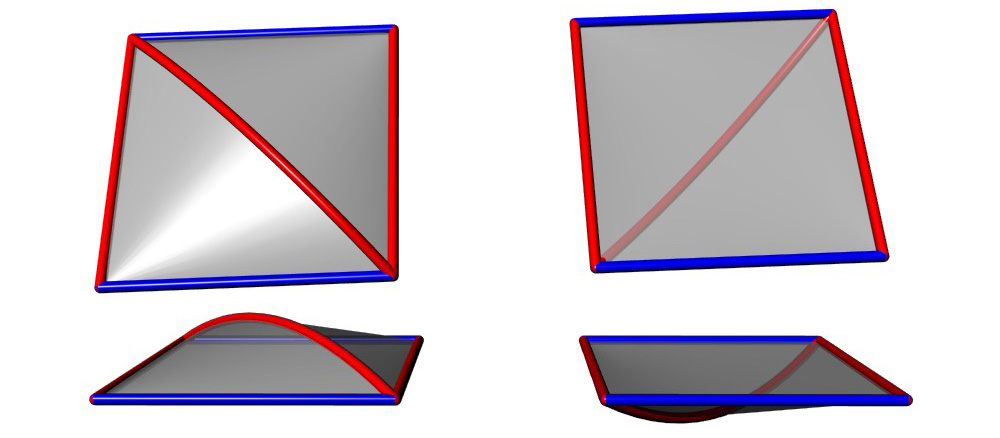}
\label{Fig:CutVeeringTet}
}

\subfloat[Shear.]{
\includegraphics[width=0.48\textwidth]{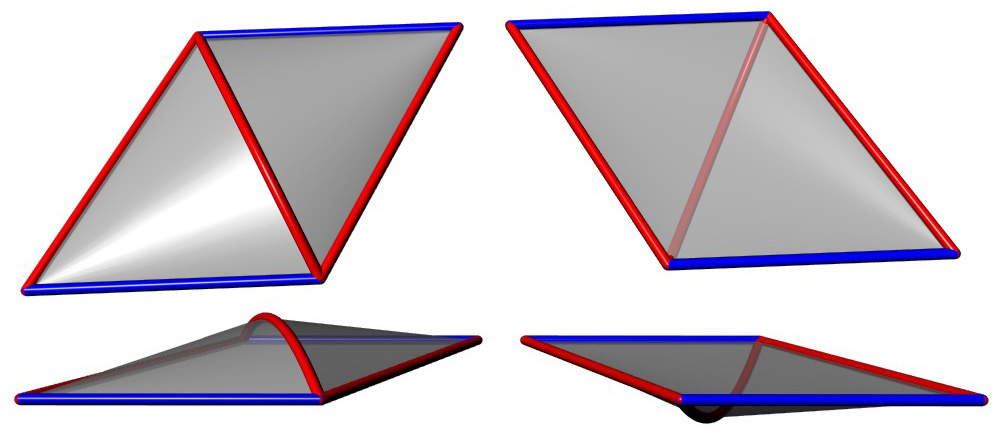}
\label{Fig:ShearHalfTets}
}
\subfloat[Bend.]{
\includegraphics[width=0.48\textwidth]{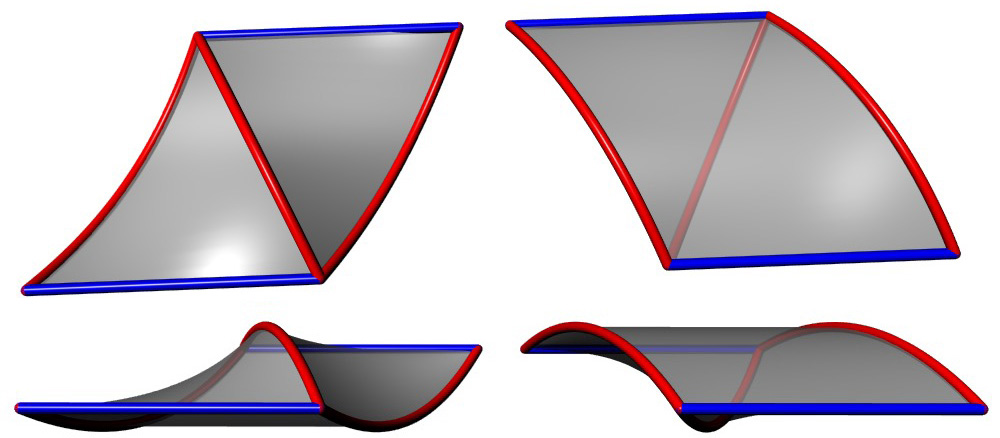}
\label{Fig:PartHalfTets}
}

\subfloat[Glue half-tetrahedra together.]{
\includegraphics[width=0.48\textwidth]{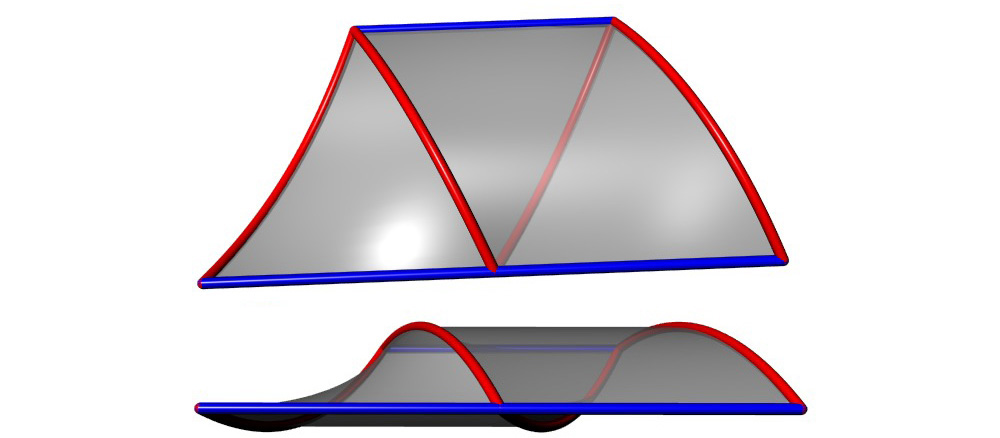}
\label{Fig:GlueHalfTets}
}
\subfloat[Continue gluing.]{
\includegraphics[width=0.48\textwidth]{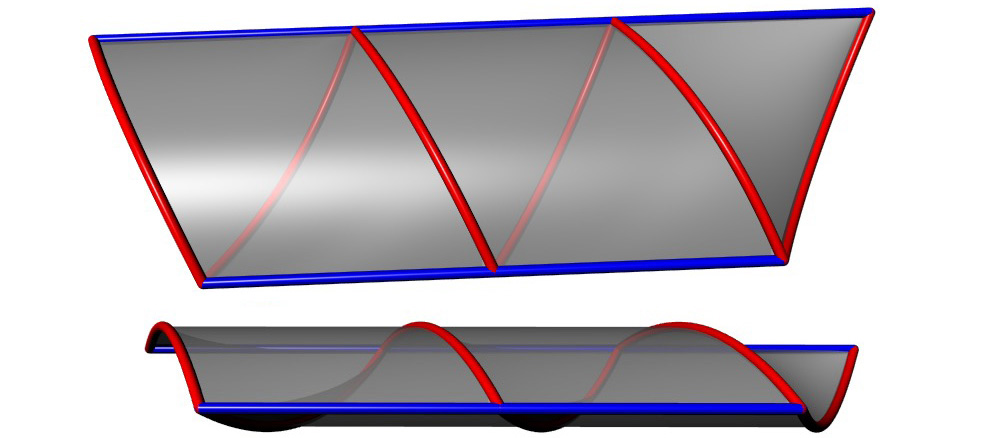}
\label{Fig:Continuing}
}
\caption{Top and side views of the construction of a red shearing region. }
\label{Fig:SolidTorusConstruction}
\end{figure}

\begin{definition}
\label{Def:ShearingDecomposition}
Suppose that $\calU$ is a collection of model shearing regions.
Let $\calU^{(0)}$ be the set of ideal points.
Suppose furthermore that the shearing regions are glued along all of their squares, respecting the colours of edges and so that every edge has exactly two helical models.
We call $\calU$ a \emph{shearing decomposition} of $|\,\calU - \calU^{(0)}|$.
The decomposition is called \emph{transverse} if all of the shearing regions in $\calU$ are transverse and the gluings respect the transverse orientations on the squares.
\end{definition}

Suppose that $\calV$ is a veering triangulation (not necessarily transverse or finite).
Recall from \refsec{Triangulations} that there are blue and red fan tetrahedra as well as toggle tetrahedra.
Cutting a veering tetrahedron along its equatorial square results in a pair of \emph{half-tetrahedra};
see \reffig{CutVeeringTet}.
In every half-tetrahedra there is a unique (up to isotopy) \emph{half-diamond}:
this is a triangle, properly embedded in the half-tetrahedron, meeting only the edges of the colour of the $\pi$--edge, and those only exactly once at each midpoint.
We give a half-diamond the colour of the edges it meets.
See \reffig{HalfDiamonds}.
We arrange matters so that the two half-diamonds in a fan tetrahedron meet along their bases, and so form a full diamond.
The two half-diamonds in a toggle tetrahedron $t$ meet in exactly one point:
the centre of the equatorial square of $t$.
For each half-diamond in a toggle tetrahedron, the central half
of its intersection with the equatorial square is the \emph{boundary arc} of the half-diamond.
(In \refdef{MidSurface}, the union of the boundary arcs will give the boundary of the \emph{mid-surface}.)
Again, see \reffig{HalfDiamonds}.

\begin{figure}[htbp]
\includegraphics[width=0.97\textwidth]{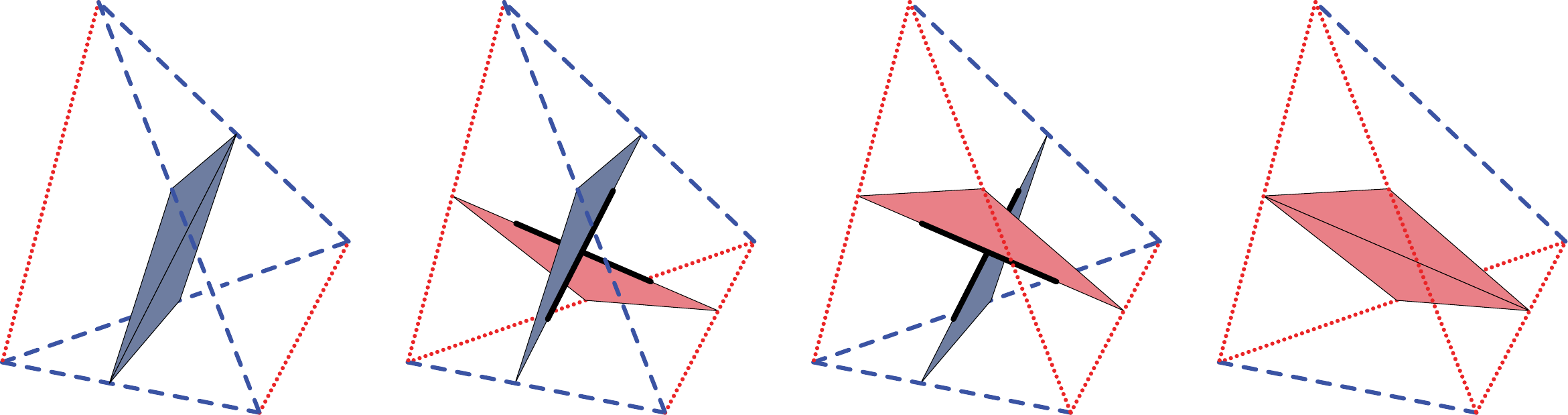}
\caption{Diamonds and half-diamonds.
Each half-diamond in a toggle tetrahedron has a boundary arc, shown here in black.}
\label{Fig:HalfDiamonds}
\end{figure}

\begin{figure}[htbp]
\subfloat[]
{
\includegraphics[width = 6 cm]{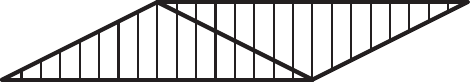}
\label{Fig:LineFieldGood}
}
\qquad
\subfloat[]
{
\includegraphics[width = 4.2 cm]{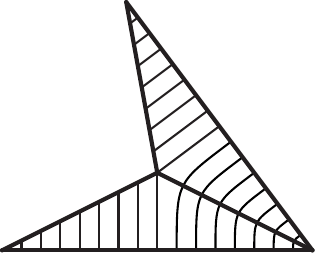}
\label{Fig:LineFieldBad}
}
\caption{In \reffig{LineFieldGood} we see adjacent half-diamonds in a veering triangulation.
In \reffig{LineFieldBad} we see an unpleasant possibility for adjacent half-diamonds in a taut triangulation.}
\label{Fig:LineField}
\end{figure}

\begin{theorem}
\label{Thm:ShearingDecomposition}
Suppose that $\calV$ is a veering triangulation (not necessarily transverse or finite).
Then there is a canonical shearing decomposition of $M$ associated to $\calV$.
\end{theorem}

\begin{proof}
Suppose that $t$ is a half-tetrahedron and $d$ is its half-diamond.
Fix a vertical line field on $d$ as shown in the left-most half-diamond of \reffig{LineFieldGood}.
Let $f$ and $f'$ be the triangular faces of $t$.
The colour of $d$ is the majority colour of the edges of $t$.
Thus the colour of $t$ and $d$ matches the majority colour of both $f$ and $f'$.
Suppose that $t$ is glued to another half-tetrahedron, $t'$, across $f'$.
Let $d'$ be the half-diamond of $t'$.
Thus $d'$ and $d$ have the same colour.


Note that the $\pi$--edges of $t$ and $t'$ are distinct edges of the model face $f'$.
(This follows from the definition of a veering triangulation: see \reffig{TwoTautTetrahedra}.)
Thus, as shown in \reffig{LineFieldGood}, we can locally extend the vertical line field on $d$, through $f'$, to $d'$.
See \reffig{GlueHalfTets}.
Let $f''$ be the other triangular face of $t'$.
Continuing in this fashion in both directions, we obtain a shearing region.
The union of the half-diamonds is the mid-band.
See \reffig{SolidTorus}.
\end{proof}

\begin{figure}[htbp]
\subfloat[Three-quarters view.]{
\includegraphics[width=0.97\textwidth]{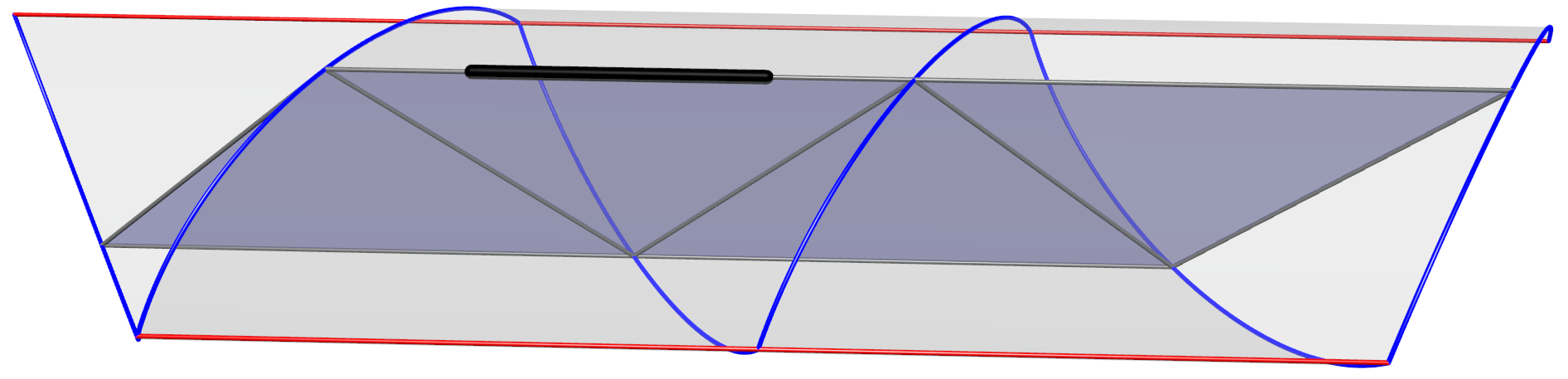}
}

\subfloat[View from above.]{
\includegraphics[width=0.97\textwidth]{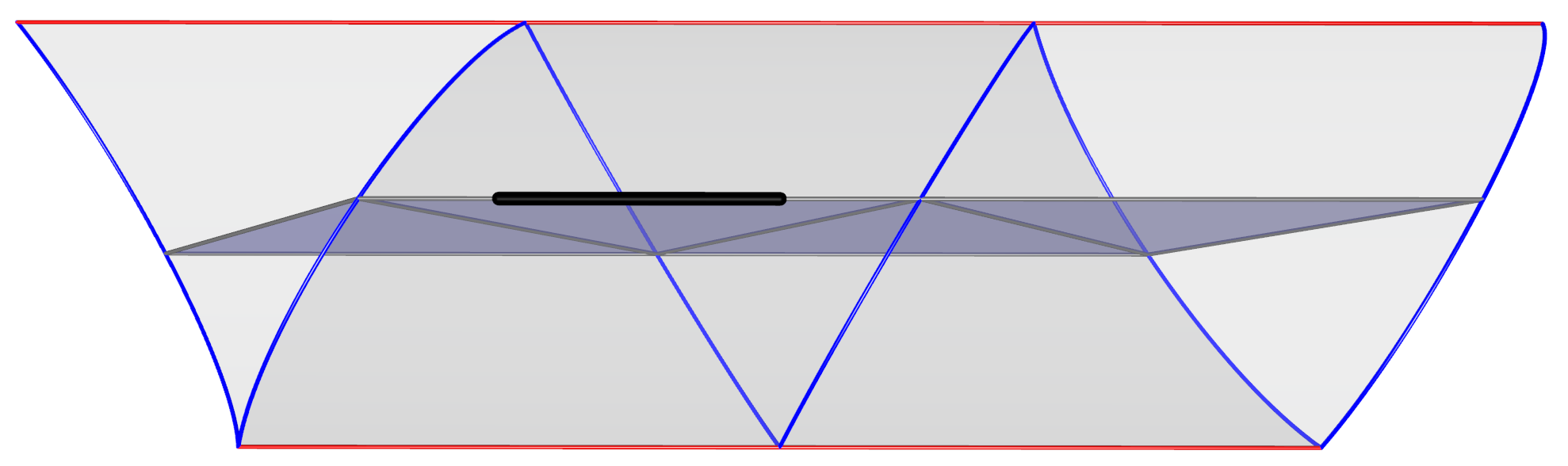}
}
\caption{A red shearing region, with embedded mid-band.
The boundary arc of the half-diamond (contained in a half-tetrahedron contained in a toggle tetrahedron) is drawn in black.}
\label{Fig:SolidTorus}
\end{figure}

We give examples of mid-bands in Figures~\ref{Fig:FigEightPictures}, \ref{Fig:m203Pictures}, \ref{Fig:PicturesNonFibered}, and~\ref{Fig:fLLQccecddehqrwjj}.
These are taken from the veering census~\cite{GSS19}.
For each example we draw, in one column per tetrahedron, its upper and lower faces.
On the faces we indicate their intersections with $B^\calV$ and $B^\calV$ after the straightening isotopy.
We also draw the mid-annuli.
See Figures~\ref{Fig:UpperHalfTet},~\ref{Fig:LowerHalfFan}, and~\ref{Fig:LowerHalfToggle}.

\begin{figure}[htbp]
\centering
\subfloat[Tetrahedra.]{
\includegraphics[height=5cm]{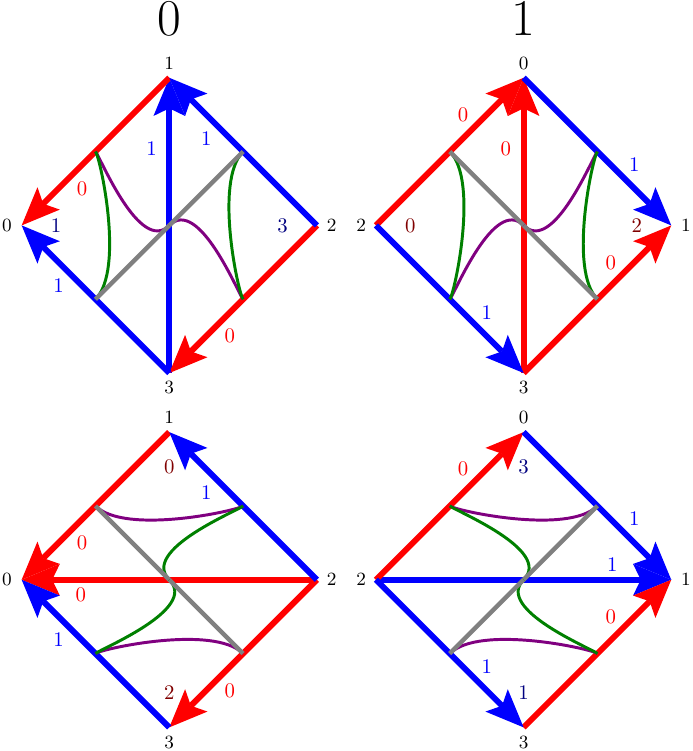}
\label{Fig:cPcbbbiht_tetrahedra}
}
\subfloat[Mid-annuli.]{
\includegraphics[height=5cm]{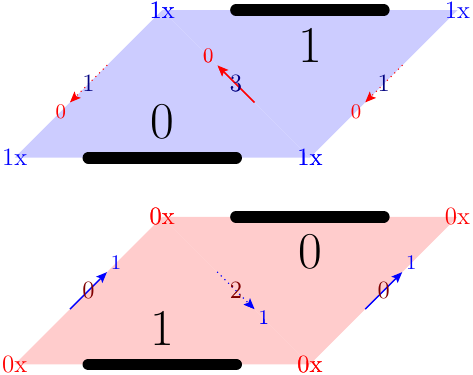}
\label{Fig:cPcbbbiht_mid-annuli}
}
\caption{A veering triangulation for \usebox{\FigEightSnappy} from the SnapPea census~\cite{snappy}.
This is \usebox{\FigEightVeer} in the census of transverse veering triangulations~\cite{GSS19}.
}
\label{Fig:FigEightPictures}
\end{figure}

\begin{figure}[htbp]
\centering
\subfloat[Tetrahedra.]{
\includegraphics[height=5cm]{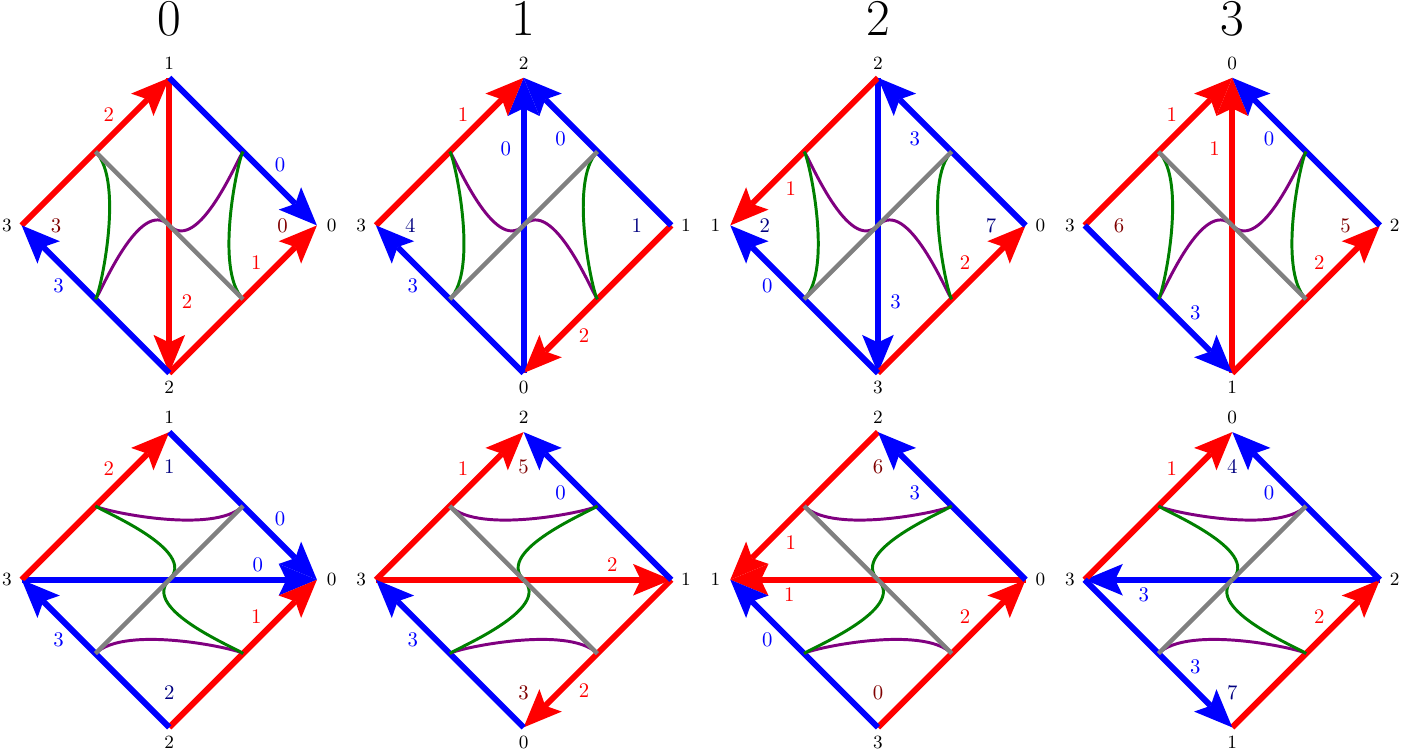}
\label{Fig:eLMkbcddddedde_tetrahedra}
}

\subfloat[Mid-annuli.]{
\includegraphics[height=5cm]{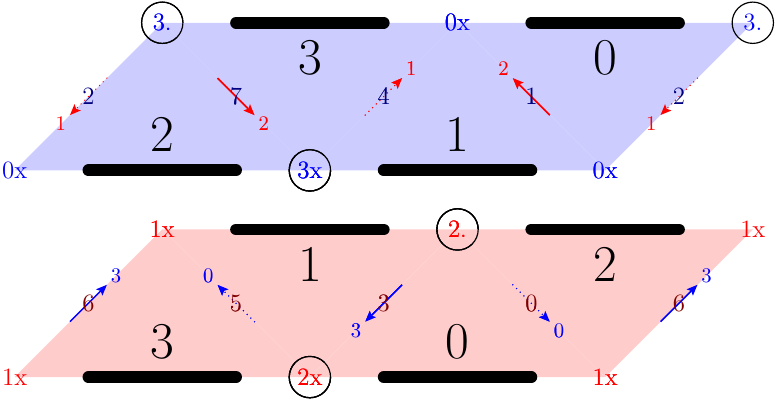}
\label{Fig:eLMkbcddddedde_mid-annuli}
}
\caption{ \usebox{\FourTetExSnappy}, \usebox{\FourTetExVeer}.}
\label{Fig:m203Pictures}
\end{figure}

\begin{figure}[htbp]
\centering
\subfloat[Tetrahedra.]{
\includegraphics[width = 0.97\textwidth]{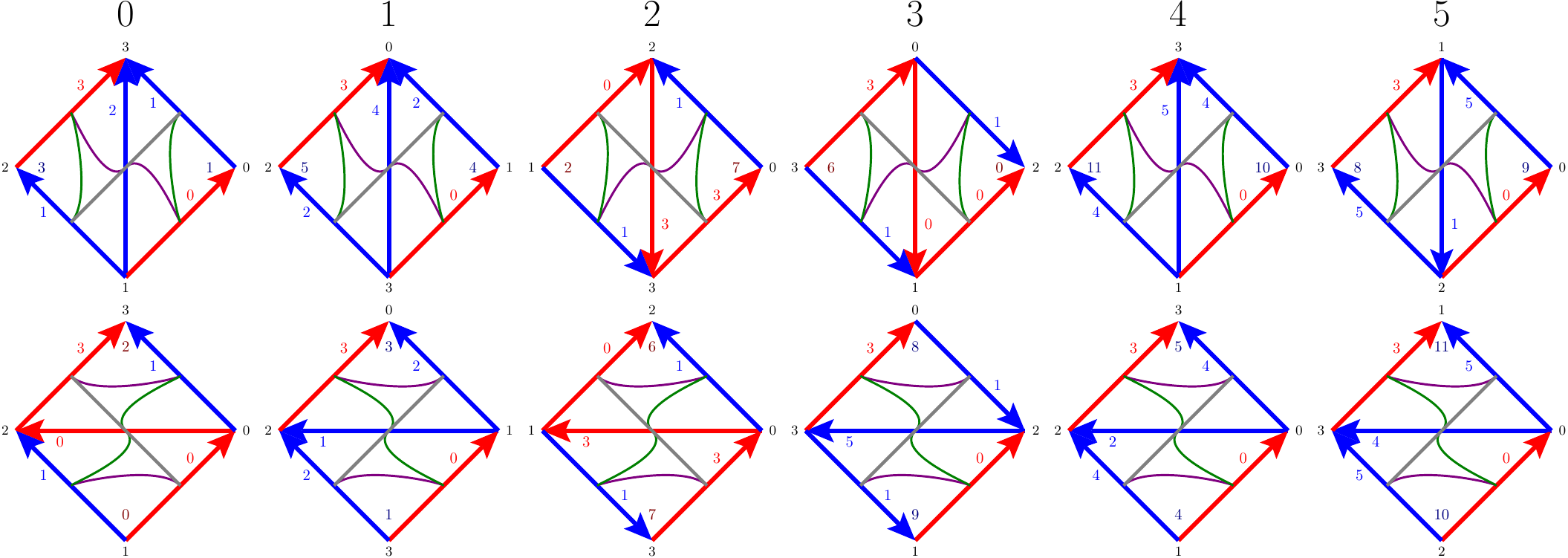}
\label{Fig:gLLAQbecdfffhhnkqnc_tetrahedra}
}

\subfloat[Mid-annuli.]{
\includegraphics[height = 8.4cm]{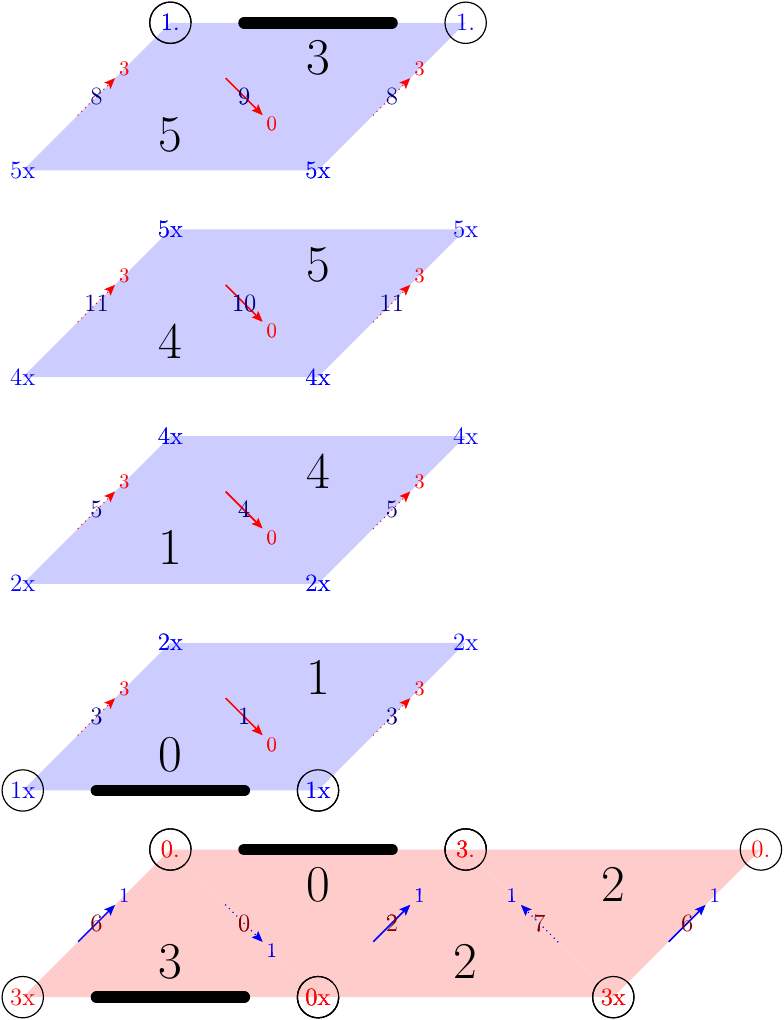}
\label{Fig:gLLAQbecdfffhhnkqnc_mid-annuli}
}
\caption{\usebox{\NonFiberedSnappy}, \usebox{\NonFiberedVeer}.}
\label{Fig:PicturesNonFibered}
\end{figure}

\begin{figure}[htbp]
\centering
\subfloat[Tetrahedra.]{
\includegraphics[width=0.97\textwidth]{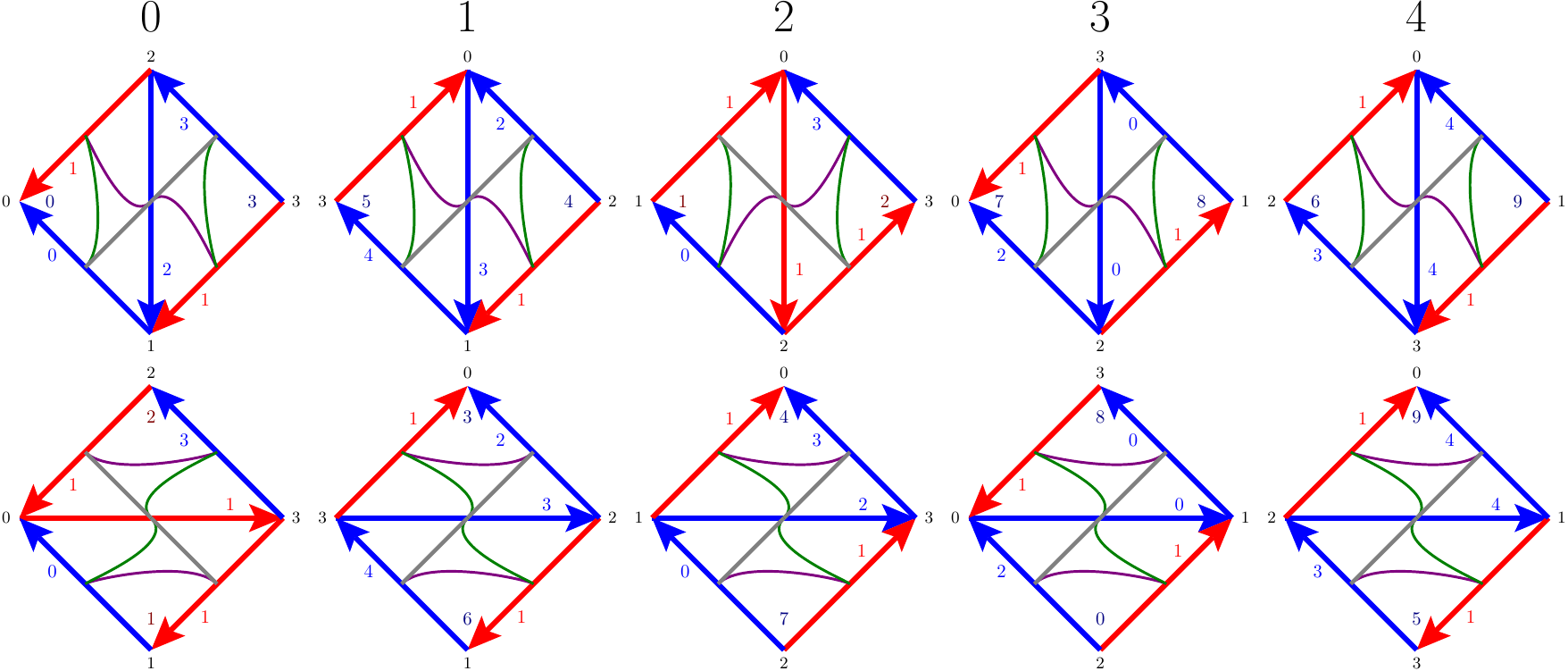}
\label{Fig:fLLQccecddehqrwjj_tetrahedra}
}

\subfloat[Mid-annuli]{
\includegraphics[width=0.97\textwidth]{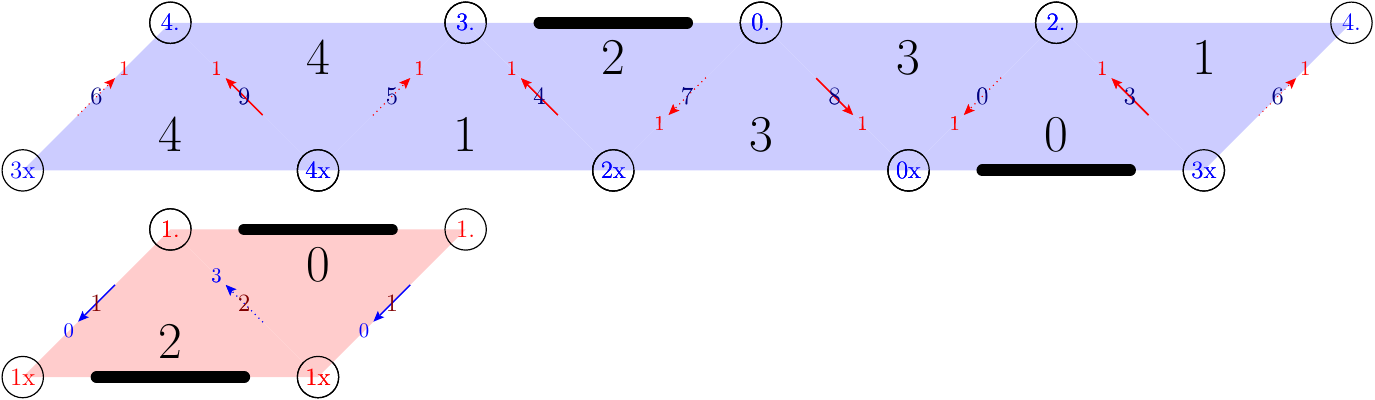}
\label{Fig:fLLQccecddehqrwjj_mid-annuli}
}
\caption{\usebox{\BigExSnappy}, \usebox{\BigExVeer}.
}
\label{Fig:fLLQccecddehqrwjj}
\end{figure}

\begin{remark}
\label{Rem:Alternate}
If the veering triangulation $\calV$ is transverse then the half-tetrahedra in each shearing region alternate between being the upper and lower halves of tetrahedra.
Thus the transverse structure on $\calV$ induces a transverse structure on the associated shearing decomposition.
\end{remark}




We now give a consequence of \refthm{ShearingDecomposition}.

\begin{corollary}
\label{Cor:PADecomposition}
Suppose that $\Phi$ is a pseudo-Anosov flow on $M$ without perfect fits.
Suppose that $M^\circ$ is the result of drilling out the singular orbits of $\Phi$.
Let $\Phi^\circ = \Phi|M^\circ$ be the restriction of $\Phi$ to $M^\circ$.
Let $\calV$ be the veering triangulation associated to $\Phi^\circ$.
Then the canonical shearing decomposition of $M^\circ$ (associated to $\calV$) factors $\Phi^\circ$ as a product of fractional Dehn twists.
\end{corollary}

\begin{proof}
Consulting \reffig{SolidTorusConstruction} we find that each shearing solid torus (of the shearing decomposition) gives a fractional Dehn twist. 
This gives the desired (canonical) factorisation of $\Phi^\circ$.
\end{proof}

\begin{remark}
\label{Rem:NotDecomposition}
Despite this (and despite \refcor{Dual}), we cannot conclude that every pseudo-Anosov \emph{homeomorphism} decomposes as a product of fractional Dehn twists.
For example, consider the manifold \usebox{\BigExSnappy} from the SnapPy census (see \reffig{fLLQccecddehqrwjj}).
This admits a layered veering triangulation (with veering isomorphism signature \usebox{\BigExVeer}~\cite{GSS19}).
Since \usebox{\BigExSnappy} has first homology of rank one, the resulting fibration over the circle is the unique such on this manifold.
A bit of linear algebra then shows that the fibres have non-zero algebraic intersection number with the cores of the ideal solid tori of the shearing decomposition.
\end{remark}

\begin{question}
\label{Que:CoreCurves}
Let $\gamma(U)$ be a core curve for the shearing region $U$.
Performing certain Dehn fillings along $\gamma(U)$ produces new veering triangulations; see~\cite{veering_dehn_surgery} and~\cite[Definition~4.1]{Tsang22b}.
Let $\gamma(\calV)$ be the union of the curves $\gamma(U)$.

Suppose that $U$ and $V$ are a pair of regions.
Suppose that the upper boundary of $U$ equals the lower boundary of $V$.
That is, suppose that $\bdy^+ U = \bdy^- V$.
Then $\gamma(V)$ is parallel to $\gamma(U)$;
accordingly we delete $\gamma(V)$ from $\gamma(\calV)$.

Now $\gamma(\calV)$ is a link canonically associated to $M$ and $\calV$.
What are the geometric properties of $M - \gamma(\calV)$?
\end{question}

\subsection{Crimping}
\label{Sec:Crimping}




Here we define the \emph{crimped shearing decomposition} of $M$.
This ensures that the union of the shearing regions of a fixed colour is a manifold (with various inward and outward paring loci) containing all of the veering edges of that colour.
Crimping also improves the way that the mid-bands meet.
Their union becomes the \emph{mid-surface}.
Crimping is similar to folding, in a train-track, all switches with both in- and out-degree bigger than one.

Suppose that $\calV$ is a veering triangulation.
The associated crimped shearing decomposition is obtained as follows.

\begin{definition}
The \emph{equatorial} branched surface $E(\calV)$ is the union of the equatorial squares of all veering tetrahedra.
\end{definition}

\begin{figure}[htbp]
\includegraphics[height = 3.5 cm]{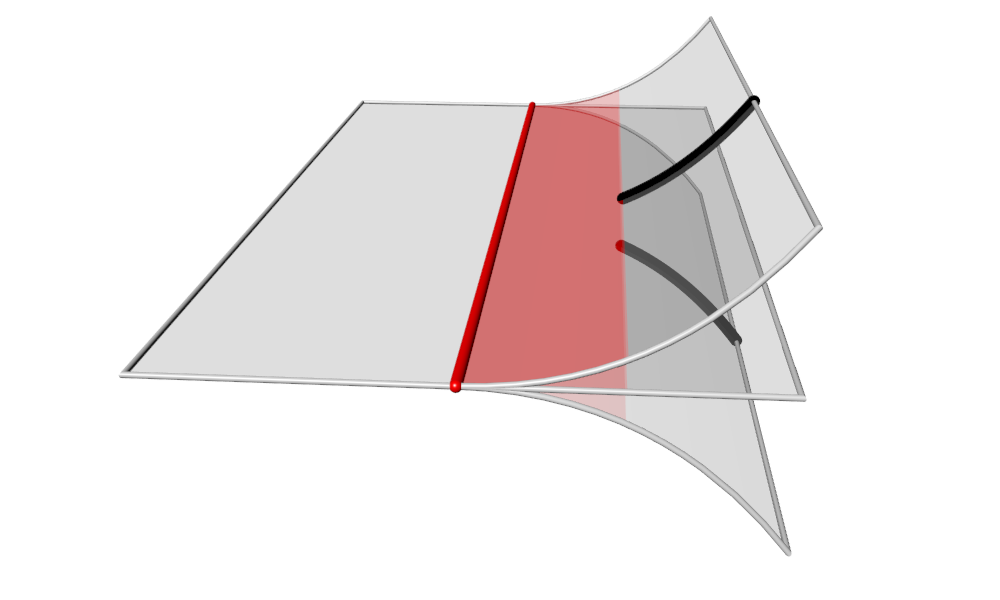}
\quad
\includegraphics[height = 3.5 cm]{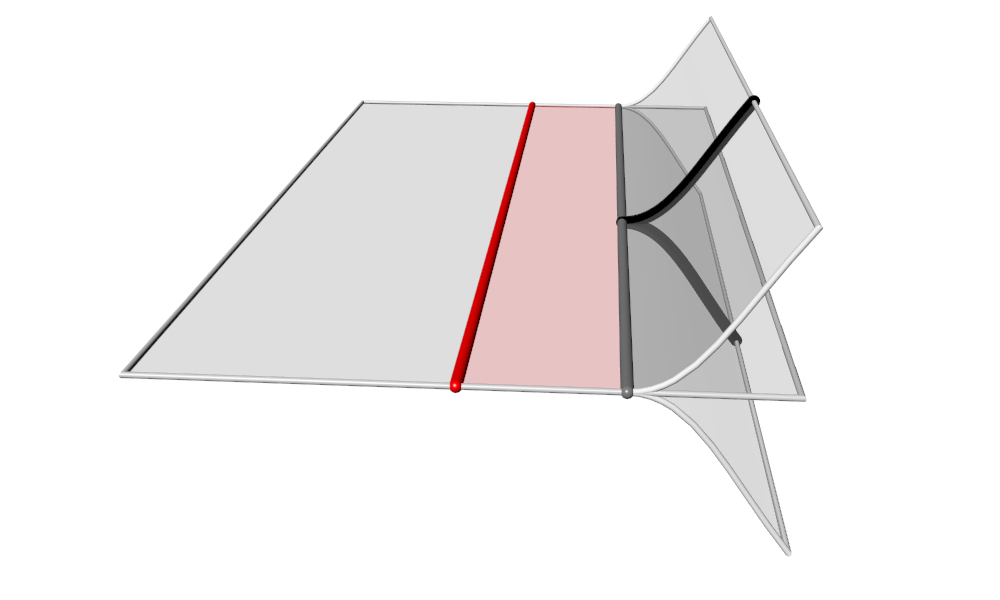} \\
\includegraphics[height = 3.5 cm]{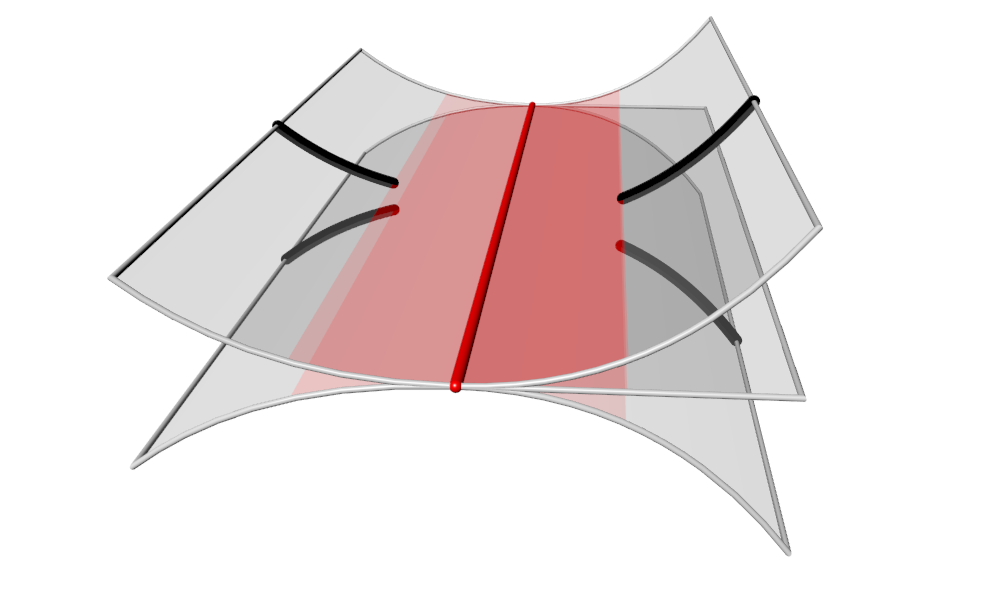}
\quad
\includegraphics[height = 3.5 cm]{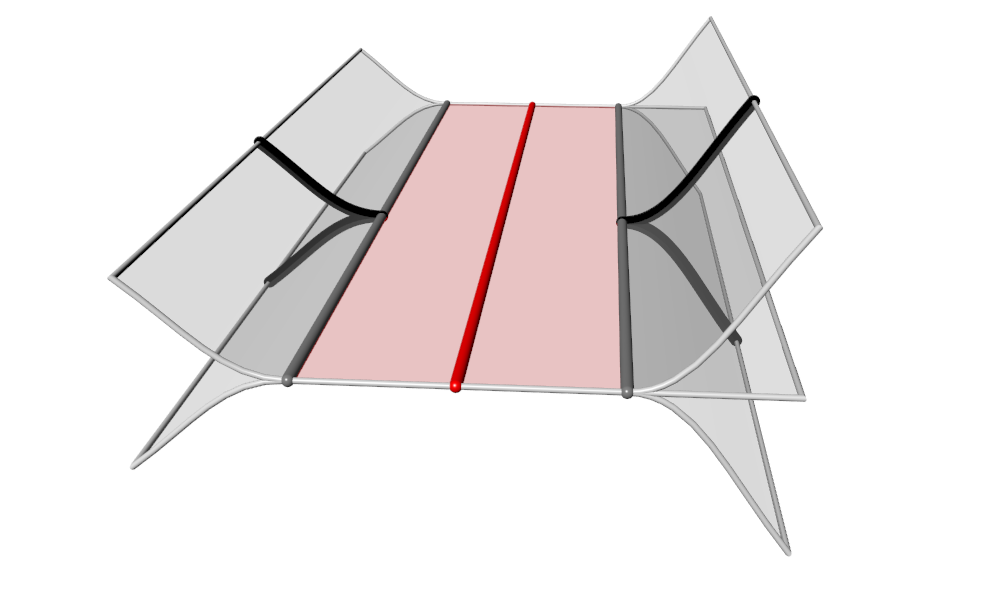}
\caption{Top row:
An edge $e \in \calV^{(1)}$ before and after crimping on the right.
No crimping is required on the left.
Bottom row:
Both sides are crimped.
The veering edges are drawn in red, the crimped edges are drawn in grey, and the boundary arcs are drawn in black.
The neighbourhoods $N_\prec(e)$ and $N_\succ(e)$, and the crimped rectangles, are shaded red.
}
\label{Fig:Crimping}
\end{figure}

Note that an edge $e \in \calV^{(1)}$ lies in the branch locus of $E(\calV)$ if and only if the degree of $e$ (in $E(\calV)$) is at least three.
Suppose that there are at least two squares to the right of $e$.
Let $N_\prec(e)$ be a collar neighbourhood to the right side of $e$, taken inside of $E(\calV)$.
(We choose the size of the collar neighbourhood so that it meets each boundary arc of each adjacent half-diamond in a single point.)

So $N_\prec(e)$ contains $e$ and a rectangle for every equatorial square to its right.
See \reffig{Crimping} (upper left) for a picture of one possible $N_\prec(e)$.
We define $N_\succ(e)$ similarly, again when there are at least two squares to the left of $e$.
See \reffig{Crimping} (lower left).

\begin{definition}
\label{Def:CrimpedEquatorial}
We obtain the \emph{crimped equatorial} branched surface $E_\crimp(\calV)$ from the equatorial branched surface $E(\calV)$ by \emph{crimping} edges, as follows.
For every veering edge $e$, fold together all rectangles in $N_\prec(e)$ to obtain a single rectangle;
do the same to the left collar $N_\succ(e)$.
\end{definition}

After crimping the sides of all veering edges (having at least two squares), the veering edges are disjoint from the branch locus of $E_\crimp(\calV)$.
Also, there are no vertices in $E_\crimp(\calV)$.
Thus we call the components of $E_\crimp^{(1)}(\calV)$ \emph{crimped edges}.
The midpoint of each crimped edge equals the endpoints of two boundary arcs.
See \reffig{Crimping} (right) for pictures of possibilities for $E_\crimp(\calV)$.

Suppose that we had to crimp the right side of $e$.
So, before crimping, $N_\prec(e)$ contained two or more rectangles.
Then, after crimping, there is a single \emph{crimped rectangle} between $e$ and the crimped edge immediately to the right of $e$.
Note that two closed crimped rectangles are either disjoint or meet along their common veering edge.
In our figures we colour the crimped edges as a dashed grey.

\begin{definition}
\label{Def:CrimpedBigon}
Since we draw pictures in the cusped manifold, we will refer to the crimped rectangle as a \emph{crimped bigon}.
\end{definition}

In Figures~\ref{Fig:Crimping},~\ref{Fig:CrimpedShearingRegion}, and~\ref{Fig:UpperHalfTet}
through~\ref{Fig:StraightenedCrimpedShearingRegion} we shade crimped bigons the colour of their veering edge.

Crimping moves the equatorial square of a toggle tetrahedron into $E_\crimp(\calV)$.
There it is subdivided, by the crimped edges, into four crimped bigons and one \emph{toggle square}.
In our figures we shade the toggle squares in grey.
Since crimped bigons are disjoint, every toggle square has four cusps that reach out to the ideal points of the three-manifold.
See \reffig{ToggleSquare}.
(The widths of these cusps are set in \refsec{Junctions}.)

The two boundary arcs (of the mid-surface) in the toggle tetrahedron lie inside of the toggle square.
They end at the midpoints of the crimped edges and divide the toggle square into four symmetric regions.
See \reffig{ToggleSquare}.
The veering hypothesis implies that a crimped bigon meets, along its crimped edge, exactly two toggle squares:
one at the top and one at the bottom of a stack of fan tetrahedra.
Similarly, the equatorial square of a fan tetrahedron is subdivided into two crimped bigons and one \emph{fan square}.
See \reffig{FanSquare}.

\begin{definition}
\label{Def:Station}
For every cusp $c$ of every toggle square $S$ we choose a short arc $\alpha_c$ properly embedded in $S$ which separates the cusp $c$ from the body of $S$.
The arc $\alpha_c$ meets exactly two crimped edges on the boundary of $S$.

Suppose that $e$ is a crimped edge on the boundary of $S$.
Note that $e$ is adjacent to exactly two toggle squares: $S$ is one and suppose that $S'$ is the other.
Let $\alpha'_c$ be the chosen short arc cutting the cusp $c$ off of $S'$.
We arrange matters so that the end points of $\alpha_c$ and $\alpha'_c$ on $e$ coincide.

Thus the union of all of the chosen arcs gives an embedded collection of loops and lines in the three-manifold;
the components of the union are isotopic into the ideal points of the three-manifold.
We take a small tubular neighbourhood of this union.
(The radius of the tubular neighbourhood varies;
the details are given in \refsec{Junctions}.)
We call a connected component of the result a \emph{station}.

Suppose that $\Sigma$ is a station.
Suppose that $S$ is a toggle square meeting $\Sigma$.
Let $E$ be the equatorial square containing $S$.
Looking from above, the cusp of $S$ meeting $\Sigma$ lies between two veering edges of $E$, one red and one blue.
If these are ordered red then blue as we walk anticlockwise around $\bdy E$ then we say that $\Sigma$ is an \emph{upper station}.
Otherwise $\Sigma$ is a \emph{lower station}.
\end{definition}

The naming scheme for stations is explained in \refsec{Parting}.
There upper track-cusps will pass through upper stations, and similarly for lower track-cusps and lower stations.

In \reffig{FanToggleSquares}, the intersection of the stations with the squares is shown with dots coloured green (for upper stations) or purple (for lower stations).

\begin{figure}[htbp]
\subfloat[Toggle square.]{
\includegraphics[height = 0.31\textwidth]{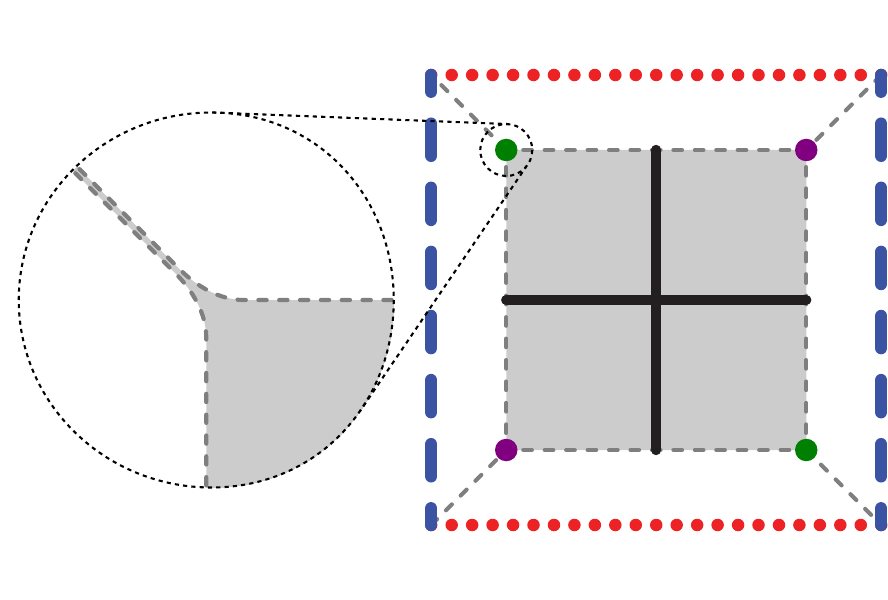}
\label{Fig:ToggleSquare}
}
\qquad
\subfloat[Fan square.]{
\includegraphics[height = 0.31\textwidth]{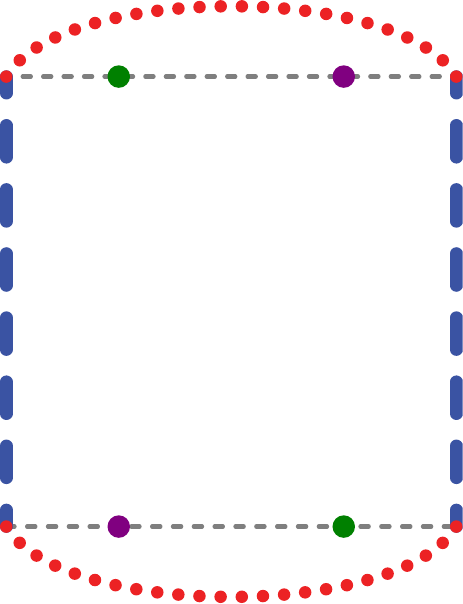}
\label{Fig:FanSquare}
}
\caption{The toggle square has four adjacent crimped bigons, the fan square has two.
Here we draw the boundary arcs (of the half-diamonds immediately above and below) on the toggle square in black.
The crimped edges are drawn in dashed grey.
The upper and lower stations are here represented as small green and purple disks following the convention given in \refdef{Station}.
}
\label{Fig:FanToggleSquares}
\end{figure}

\begin{figure}[htbp]
\subfloat[View from the side.]{
\includegraphics[width =0.97\textwidth]{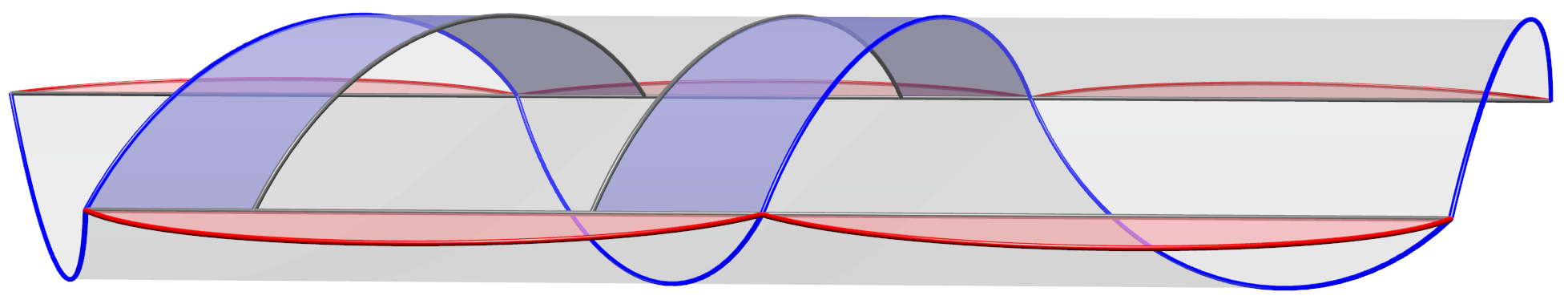}
\label{Fig:CrimpedShearingRegionSide}
}

\subfloat[View from above, with the mid-band and station (after intersecting with the crimped shearing region).]{
\includegraphics[width =0.97\textwidth]{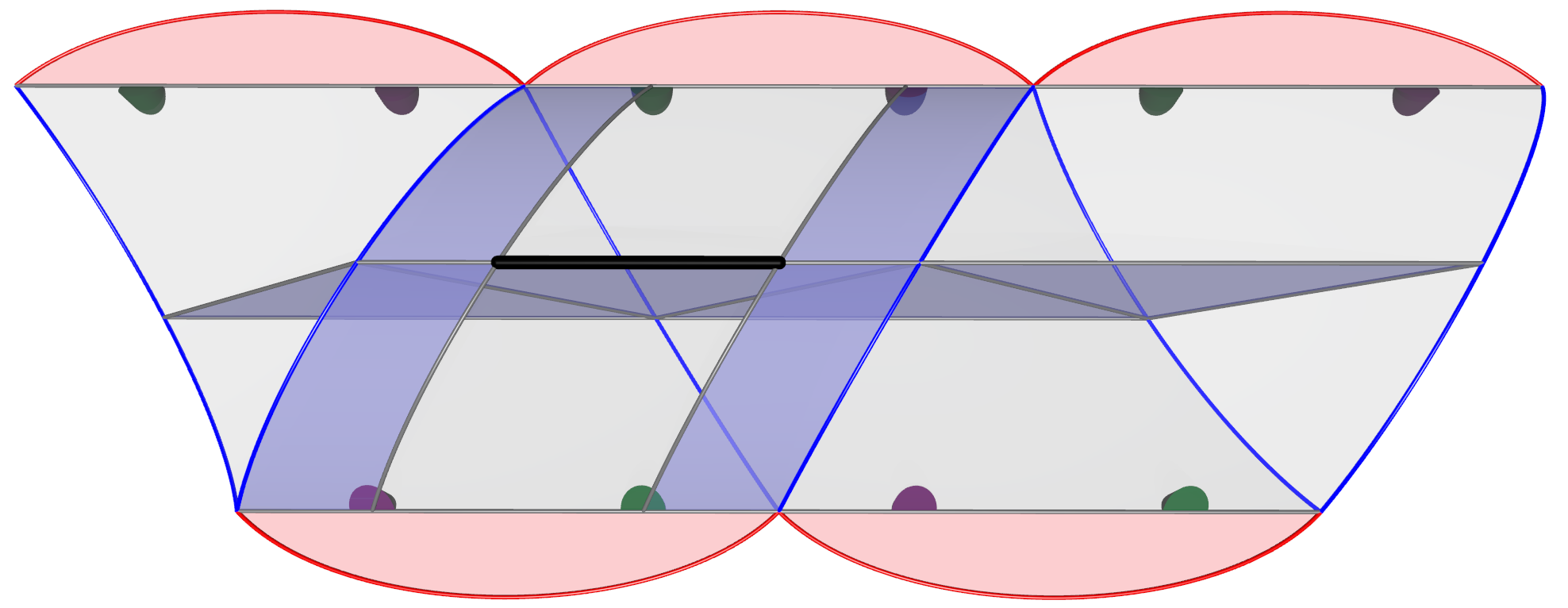}
\label{Fig:CrimpedShearingRegionAboveMidAnnulus}
}
\caption{A crimped blue solid torus, and incident red crimped bigons.
The crimped edges are drawn in grey and meet the boundary arc in its endpoints.
Wherever the red and blue crimped bigons appear to meet, they are in fact separated by a cusp of the adjacent toggle square.
See \reffig{ToggleSquare}.}
\label{Fig:CrimpedShearingRegion}
\end{figure}

\begin{definition}
\label{Def:CrimpedShearingRegion}
We define the (closures taken in the path metric of) components of $M - E_\crimp(\calV)$ as \emph{crimped shearing regions}.
See \reffig{CrimpedShearingRegion}.
Let $U$ be a model crimped shearing region.
As before, we write $\bdy^+ U$ and $\bdy^- U$ for the upper and lower boundaries of $U$.
Suppose that $e$ and $e'$ bound a crimped bigon $B$ with $e \in \calV^{(1)}$ and $e'$ a crimped edge.
If $B$ lies in either $\bdy^+ U$ or $\bdy^- U$ then we say that $e$ and $e'$ are \emph{helical} for $U$.
If $B \cap U = e'$ then we say that $e$ and $e'$ are \emph{longitudinal} for $U$.
\end{definition}

\noindent
Note that $\bdy^+ U \cap \bdy^- U$ is the collection of longitudinal crimped edges for $U$.

Note that the stations meet $U$ in a collection of three-balls.
Each such three-ball meets $\bdy U$ in a disk.
This disk is cut into exactly two pieces by the longitudinal crimped edge which it meets.

As before, we assign $U$ the colour of its helical edges.
This colour is opposite to that of each edge of $\calV^{(1)}$ that is parallel, across a crimped bigon, to the longitudinal crimped edges of $U$.

\begin{definition}
\label{Def:CrimpedTriangle}
Within $U$, we replace each triangle of the original triangulation with a corresponding \emph{crimped triangle}.
The sides of each crimped triangle consist of two helical edges, one on $\bdy^+ U$ and one on $\bdy^- U$, and a single longitudinal crimped edge.
\end{definition}

\begin{definition}
\label{Def:CrimpedShearingDecomposition}
The union of the crimped shearing regions is again homeomorphic to $M$;
together they form the \emph{crimped shearing decomposition} of $M$.
\end{definition}

\begin{definition}
\label{Def:Monochromatic}
The union of the red crimped shearing regions is the \emph{red submanifold} of the crimped shearing decomposition.
A connected component of the red submanifold is a \emph{red component}.
We define the \emph{blue submanifold} and \emph{blue components} similarly.
These form the components of the \emph{monochromatic decomposition}.
\end{definition}

Each red component is a handlebody with inward and outward paring loci.
The red submanifold of the monochromatic decomposition contains all of the red edges of $\calV^{(1)}$.
Furthermore, its material boundary is the union of the toggle squares.
Analogous statements are true for blue components and the blue submanifold.

\subsection{The mid-surface}
\label{Sec:MidSurface}

The mid-bands sit within the crimped shearing regions just as they sat within the original shearing regions.
See \reffig{CrimpedShearingRegionAboveMidAnnulus}.
We may now glue the mid-bands to each other along their boundaries obtain a surface.

\begin{definition}
\label{Def:MidSurface}
The union of the red mid-bands in the red submanifold gives the \emph{red mid-surface} $\calS_R$.
We build the blue mid-surface $\calS_B$ in a similar fashion.
We define the \emph{mid-surface} to be $\calS = \calS_R \cup \calS_B$.
\end{definition}

Note that each component of $\calS_R$ sits inside, and is a deformation retract of, a red component of the crimped shearing decomposition.
Thus $\calS_R$ meets all red edges but no blue edges.
A similar statement holds for $\calS_B$.
Each boundary arc of $\calS_R$ meets precisely one boundary arc of $\calS_B$;
these intersect in a single point at the centre of the corresponding toggle square.

Note that $\calS_R$ and $\calS_B$ receive cell-structures from the (images under crimping of the) half-diamonds.
Taking a horizontal union of half-diamonds yields a mid-band.
Taking a diagonal union of half-diamonds (stopping at toggle squares, if any) yields a \emph{diagonal strip}.
\reflem{DualMeetsToggles} implies the following.

\begin{corollary}
Every diagonal strip starts and ends at toggle squares.
Thus every component of $\calS_R$ and of $\calS_B$ has at least one boundary component. \qed
\end{corollary}

\begin{example}
In \reffig{fLLQccecddehqrwjj} the red mid-surface has two diagonal paths, both traversing two half-diamonds.
The blue mid-surface also has two diagonal paths, one traversing six half-diamonds and the other traversing ten.
\end{example}

Every boundary component of the mid-surface runs alternatingly along boundary arcs contained in the upper and lower boundaries of crimped shearing regions.
In Figures~\ref{Fig:cPcbbbiht_mid-annuli},
\ref{Fig:eLMkbcddddedde_mid-annuli},
\ref{Fig:gLLAQbecdfffhhnkqnc_mid-annuli},
and~\ref{Fig:fLLQccecddehqrwjj_mid-annuli}
we give several examples;
the boundary arcs are indicated by thick black lines.
In \reffig{cPcbbbiht_mid-annuli} both mid-surfaces are once-holed tori;
each boundary component of each mid-surface consists of two boundary arcs.
In \reffig{eLMkbcddddedde_mid-annuli} both mid-surfaces are copies of $N_{3,1}$:
the non-orientable surface with one boundary component and three cross-caps.
In \reffig{gLLAQbecdfffhhnkqnc_mid-annuli} both mid-surfaces are copies of $N_{2, 1}$:
the once-holed Klein bottle.
(This last was the first example of a non-fibered veering triangulation;
see~\cite[Section~4]{HRST11}.)
Finally, in \reffig{fLLQccecddehqrwjj_mid-annuli} the mid-surfaces are a pair of once-holed Klein bottles, with one having greater area than the other.

\begin{remark}
Mid-surfaces also allow one to see the \emph{walls} of a veering decomposition, as defined by Agol and Tsang~\cite[Definition~3.3]{AgolTsang22}.
For example, in \reffig{fLLQccecddehqrwjj} there is a wall of width three consisting of the tetrahedra 4 and 1.
\end{remark}

\subsection{Labelling the mid-surface}
\label{Sec:Labelling}

We now describe the labelling scheme for the mid-surfaces used in the census~\cite{GSS19}.
This is useful when drawing pictures and discussing examples.
Suppose that $\calV$ is a finite transverse veering triangulation.
We number the tetrahedra, the faces, the edges, and the vertices of the tetrahedra using the conventions from Regina~\cite{regina}.
Regina also provides us with orientations for the edges of $\calV^{(1)}$;
we will alter these to make them agree, as much as possible, with transverse orientations of mid-annuli.

We give four examples in Figures~\ref{Fig:FigEightPictures}, \ref{Fig:m203Pictures}, \ref{Fig:PicturesNonFibered}, and~\ref{Fig:fLLQccecddehqrwjj}.
For each example, we draw its mid-annuli and, in one column per tetrahedron, the upper and lower faces for each tetrahedron (viewed from above).
On each face we draw the upper (green) and lower (purple) train-tracks.
(Where these intersect, the intersection is coloured grey.)

In order to draw a mid-band $A = A(U)$ we choose a transverse orientation for it;
this then induces a transverse orientation on each half-diamond $d$ of $A$.

In the examples of Figures~\ref{Fig:FigEightPictures}, \ref{Fig:m203Pictures}, \ref{Fig:PicturesNonFibered}, and~\ref{Fig:fLLQccecddehqrwjj} the mid-bands are all annuli and the transverse orientation points into the page.

We label the vertices, edges, and face of the half-diamond $d$ as follows.
\begin{itemize}
\item
Suppose that $v$ is a vertex of $d$.
We label $v$ with the number of the edge $e$ in $\calV^{(1)}$ which contains $v$.
Note that $e$ is helical for $U$.
We append this number with one of the symbols from $\{\cdot, \text{x}\}$.
The x means that the orientation of $e$ agrees with the transverse orientation on $d$; the dot means the opposite.
(The x represents the fletching of an arrow, while the dot represents the arrowhead.)
\item
Suppose that $\epsilon$ is a diagonal edge of $d$.
We label $\epsilon$ with the number of the face $f$ in $\calV^{(2)}$ which contains $\epsilon$;
we place the label at the midpoint of $\epsilon$.
The vertices of $\epsilon$ are already labelled with the numbers of two of the three edges of $f$.
Let $e$ be the third edge of $f$.
Note that $e$ is longitudinal for $U$.
We draw a small copy of $e$ on top of $\epsilon$ and label the copy with the number of $e$ (in the other colour and using a smaller font).
Note that $\epsilon$ and $e$ cobound a rectangle in $f$;
we use this rectangle to transport the orientation of $e$ to $\epsilon$.
Finally, we draw the arrow dotted or solid as the transverse orientation on $d$ points towards or away from $e$.
(That is, as drawn in Figures~\ref{Fig:FigEightPictures}, \ref{Fig:m203Pictures}, \ref{Fig:PicturesNonFibered}, and~\ref{Fig:fLLQccecddehqrwjj}, the edge $e$ is behind or in front of $A$.)
\item
Suppose that $\epsilon$ is the base of a half-diamond $d$.
If $d$ lies in a toggle tetrahedron then we draw a thick black line on $\epsilon$, to indicate the boundary arc on $d$.
\item
Finally, we label $d$ itself with the number of the tetrahedron that contains $d$.
\end{itemize}

Suppose that $A$ and $B$ are mid-annuli.
Let $\bdy^- A$ be the lower boundary of $A$, \emph{minus} the open boundary arcs.
Thus $\bdy^- A$ is either a single line, a single circle, or a collection of intervals and at most two rays.
We define $\bdy^+ B$ similarly.
Suppose that $A$ and $B$ are glued to each other, say with a component $\gamma$ of $\bdy^- A$ meeting a component of $\bdy^+ B$.
(It is also possible for $A$, say, to be glued to itself.)
We call the gluing  $\gamma$ \emph{untwisted} or \emph{twisted} exactly as it does or does not faithfully transport the chosen transverse orientation on $A$ to the one on $B$.

In Figures~\ref{Fig:FigEightPictures}, \ref{Fig:m203Pictures}, \ref{Fig:PicturesNonFibered}, and~\ref{Fig:fLLQccecddehqrwjj} we indicate a twisted gluing by drawing a small black circle about all vertices of the affected boundary circle or sub-arc.
We have chosen the transverse orientations of the mid-annuli to minimise the number of half-twists required.

\begin{remark}
\label{Rem:Fail}
If all gluings are untwisted then the mid-surface is transversely orientable and thus orientable.
Conversely, if the mid-surface is orientable then there is a choice of transverse orientations for the mid-bands that ensures that all gluings are untwisted.
The naive push-off discussed in \refsec{PushOff} should produce a dynamic pair when and only when the mid-surface is orientable.

Thus, if one is willing to pass to a double cover, then there should be edge orientations making the naive push-off work.
However this push-off will not be invariant under the deck transformation.
\end{remark}

\section{Bigon coordinates}
\label{Sec:BigonCoords}

In this section we place a coordinate system on the crimped shearing regions (introduced in \refsec{Crimping}).
We also give a refinement of the crimped shearing decomposition of $M$ and introduce the horizontal cross-sections.

Let $B$ be a \emph{coordinate bigon}:
a oriented disk with two marked points $x$ and $y$ in its boundary.
The points $x$ and $y$ are the \emph{corners} of $B$.
We equip $\bdy B$ with the induced orientation.
The two arcs of $\bdy B - \{x, y\}$ are denoted by $\bdy^+ B$ and $\bdy^- B$ respectively.
We arrange matters so that $\bdy^+ B$ is the arc running from $y$ to $x$.

We equip $B$ with a pair of transverse foliations:
the \emph{horizontal arcs} all meet both corners while the \emph{vertical arcs} all meet $\bdy^+ B$ and $\bdy^- B$.
We orient the former from $x$ to $y$ and the latter from $\bdy^- B$ to $\bdy^+ B$.
See \reffig{BiFoliatedBigon}.

We subdivide $B$ into a pair of sub-bigons called $\theta^B$ (upper) and $\theta_B$ (lower).
These are shown in \reffig{BigonRegions}.

\begin{figure}[htbp]
\subfloat[Model bi-foliated coordinate bigon.]{
\labellist
\small\hair 2pt
\pinlabel {$x$} [r] at 0 175
\pinlabel {$y$} [l] at 575 175
\pinlabel {$\bdy^+ B$} [br] at 120 300
\pinlabel {$\bdy^- B$} [tr] at 120 50
\endlabellist
\includegraphics[width = 0.45\textwidth]{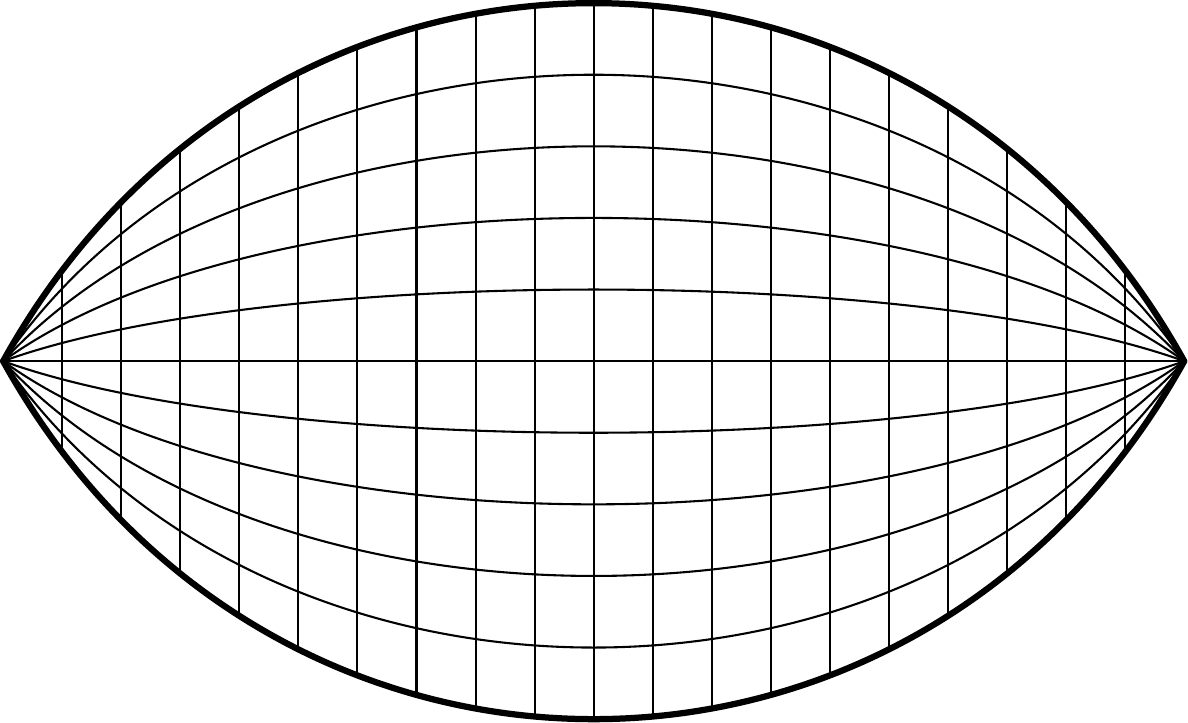}
\label{Fig:BiFoliatedBigon}
}
\quad
\subfloat[Bigon regions.]{
\labellist
\scriptsize\hair 2pt
\pinlabel {$\theta^B$} at 290 250
\pinlabel {$\theta_B$} at 290 100
\endlabellist
\includegraphics[width = 0.45\textwidth]{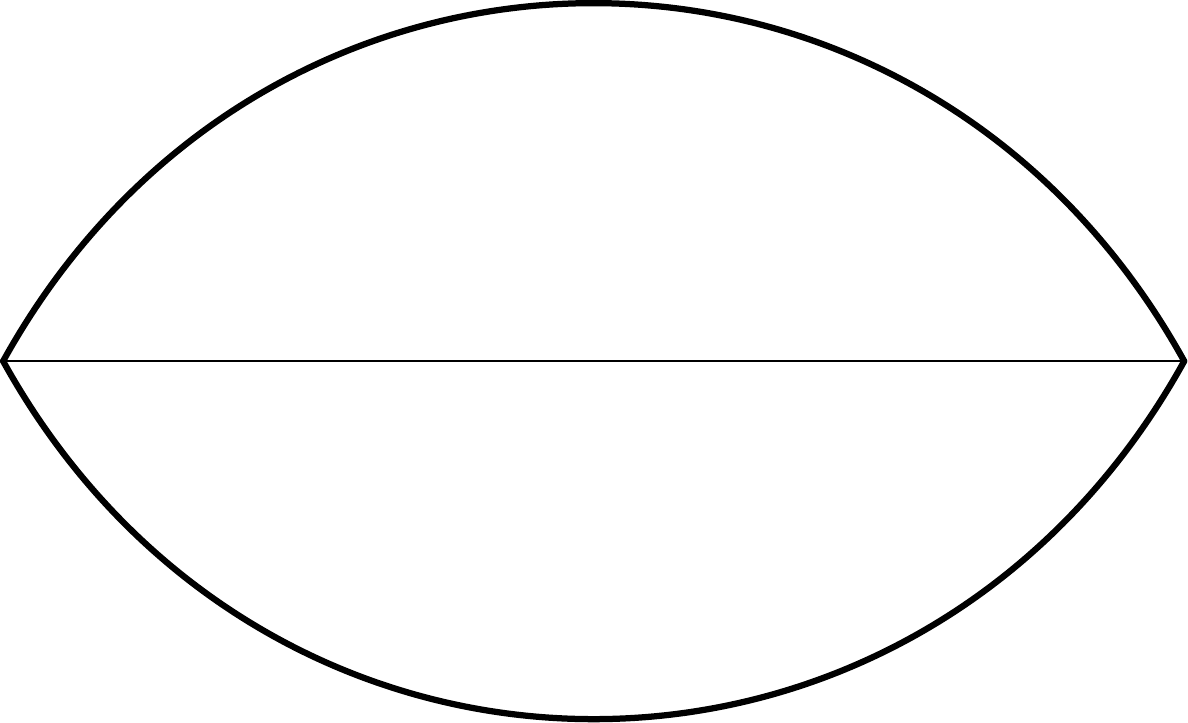}
\label{Fig:BigonRegions}
}

\caption{}
\label{Fig:Bigon}
\end{figure}

Recall that $M$ is oriented and $\calV$ is transverse veering.
Suppose that $U$ is a model crimped shearing region.
Thus $U$ inherits an orientation and, by \refrem{Alternate},
a notion of ``upwards''.
We now choose a homeomorphism $h$ between $U$ and $B \cross S^1$ or $B \cross \RR$, as $U$ is a solid torus or cylinder.
We require that $h$ preserve the various orientations.
In particular, the upper boundary of $B$ is sent to the upper boundary of $U$ by $h$.
We call $h$ the \emph{bigon coordinates} for $U$.

Let $\Theta^U$ be the image of $\theta^B \cross S^1$ (or $ \theta^B \cross \RR$) in $U$.
We define $\Theta_U$ similarly.
Note that the upper boundaries of $U$ and $\Theta^U$ agree, as do the lower boundaries of $U$ and $\Theta_U$.
That is, $\bdy^+ U = \bdy^+ \Theta^U$ and $\bdy^- U = \bdy^- \Theta_U$.
Also, we have $\bdy^- \Theta^U = \bdy^+ \Theta_U$.
We take $\Theta^\calV \subset M$ to be the union of the $\Theta^U$, taken over all model crimped shearing regions and then projected to $M$.
We define $\Theta_\calV$ similarly.
The interiors of $\Theta^\calV$ and $\Theta_\calV$ are disjoint and their union is $M$;
this is the \emph{$\Theta$--decomposition}.

\begin{remark}
\label{Rem:NiceBigonCoords}
Suppose that $U$ is a blue shearing region.
We arrange the metric in $U$ (coming from bigon coordinates) to ensure the following.
\begin{enumerate}
\item
In the induced coordinates on $\bdy^+ U$ the (pullbacks of the) blue edges of $\calV^{(1)}$ are straight and, when viewed from above, have slope $\sqrt{3}$.
Similarly, the blue edges in $\bdy^- U$ are straight and, when viewed from above, have slope $-\sqrt{3}$.
\item
For $p \in U$ we take $B(p, U)$ to be the coordinate bigon in $U$ containing $p$.
Then the two notions of vertical (coming from the coordinate bigons $B(p, U)$ and the transverse veering structure) agree.
Furthermore, the intersection of the mid-band $A(U)$ with any $B(p, U)$ is the central vertical arc of the latter.
\item
\label{Itm:Station}
As noted in \refdef{Station}, each longitudinal crimped edge intersects the stations in two short intervals.
These intervals appear slightly more than one-quarter of the length of the edge in from the ideal points of the three-manifold.
\end{enumerate}
See \reffig{CrimpedShearingRegion}.
We similarly give bigon coordinates to red model crimped shearing regions.
\end{remark}

We use the following notations for the various coordinate arcs and surfaces in bigon coordinates.

\begin{definition}
\label{Def:CrossSection}
Suppose that $U$ is a model crimped shearing region.
Fix $p \in U$.
\begin{itemize}
\item
As above, $B(p, U)$ is the coordinate bigon containing $p$.
\item
Let $x(p, U) = p \cross S^1$ ($p \cross \RR$) be the horizontal circle (line) in $U$ through $p$.
\item
Let $y(p, U)$ be the leaf of the horizontal foliation of $B(p, U)$, through $p$.
\item
Let $z(p, U)$ be the leaf of the vertical foliation of $B(p, U)$, through $p$.
\item
Let $Y(p, U)$ be the union of the leaves $z(q, U)$ as $q$ ranges over $x(p, U)$.
We call $Y(p, U)$ the \emph{vertical band} in $U$ through $p$.
\item
Let $Z(p, U)$ be the union of the leaves $x(q, U)$ as $q$ ranges over $y(p, U)$.
We call $Z(p, U)$ the \emph{(horizontal) cross-section} in $U$ through $p$.
\item
Finally, we define $X(p, U) = B(p, U)$.  \qedhere
\end{itemize}
\end{definition}

Note that the upper and lower boundaries of $\Theta^U$ and $\Theta_U$ are horizontal cross-sections.

\section{Straightening and shrinking}
\label{Sec:StraighteningShrinking}

From now on, instead of working in $M$, we work in the universal cover $\cover{M}$.
Thus we take care to ensure that our constructions are invariant under the action of the deck group.
In a slight abuse of notation we continue to write $B^\calV$ instead of the more correct $\cover{B}^\calV$.

Here we define the \emph{straightening} and \emph{shrinking isotopies}.
These are applied to the upper and lower branched surfaces $B^\calV$ and $B_\calV$.
The isotopies are \emph{local}:
in each tetrahedron they (and the resulting shrunken position) depend only on the combinatorics of that tetrahedron and its immediate neighbours.

The branched surfaces begin in dual position (shown in \reffig{DualUpperBranchedSurface}).
We \emph{straighten} the branched surfaces to move as much of each as possible into the mid-surface $\calS = \calS_R \cup \calS_B$.
We \emph{shrink} the branched surfaces to move vertices of $B^\calV$ down into $\Theta_\calV$ and those of $B_\calV$ up into $\Theta^\calV$.

We now describe in detail the upper straightening and shrinking isotopies of $B^\calV$.
The corresponding isotopies of $B_\calV$ are defined similarly.

\subsection{Straightening}
\label{Sec:Straightening}

First we \emph{straighten}.
We start with $B^t$ in dual position (shown in \reffig{DualUpperBranchedSurface}) and note that we have crimped.
For a half-tetrahedron $h$ we take the sectors of $B^h$ that do not intersect any longitudinal crimped edge;
we move those sectors to coincide with the (images after crimping of the) half-diamond of $h$.

The resulting position of the upper branched surface, in the various crimped half-tetrahedra, is shown in Figures~\ref{Fig:UpperHalfTet},~\ref{Fig:LowerHalfFan}, and~\ref{Fig:LowerHalfToggle}.
Each figure has a $180^\circ$ symmetry about its central vertical axis.
The resulting position of $B^\calV$, in a piece of a crimped shearing region, is shown in \reffig{StraightenedCrimpedShearingRegion}.

\begin{figure}[htbp]
\centering
\subfloat[Three-quarter view.]{
\includegraphics[height=4cm]{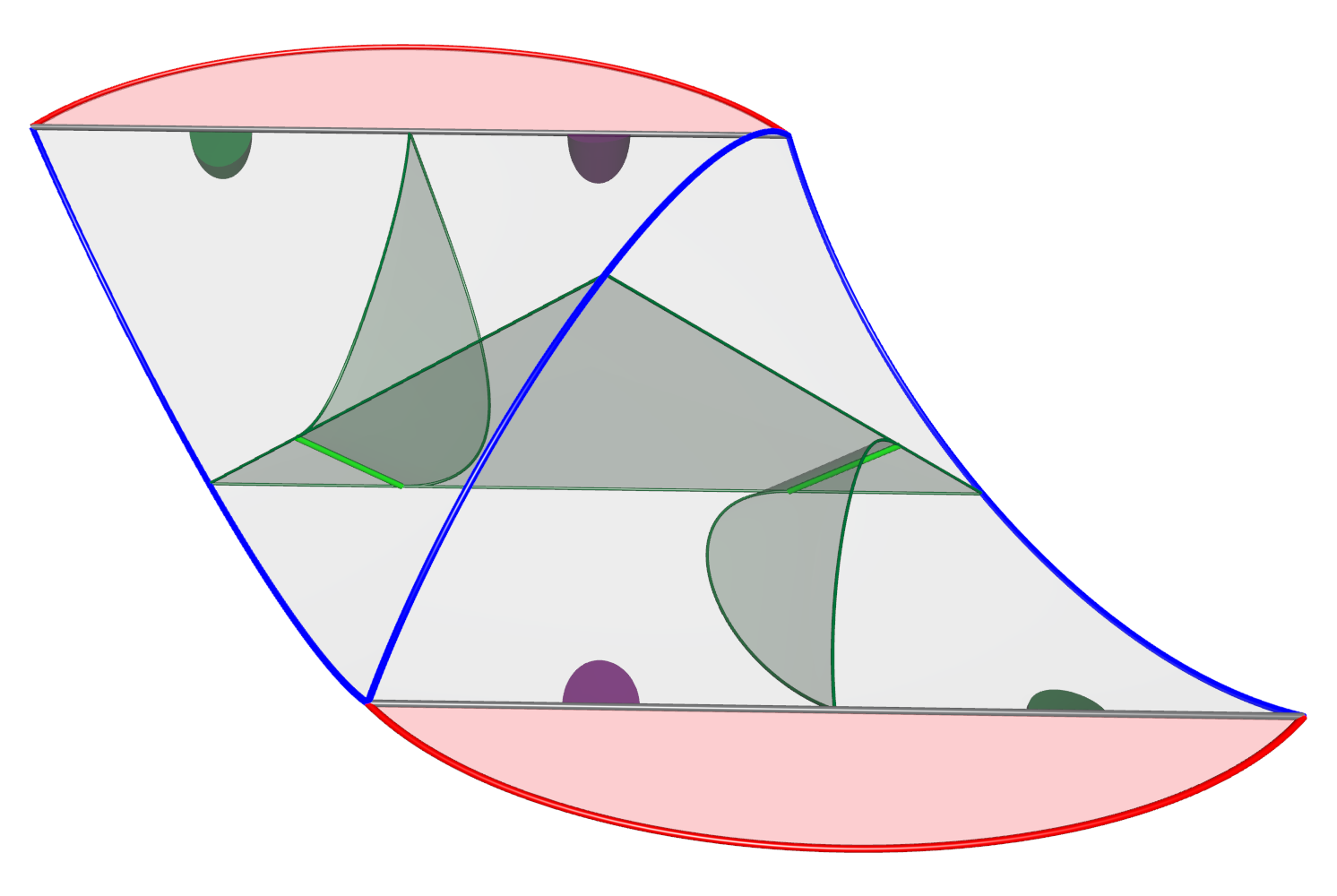}
}
\subfloat[Top view.]{
\includegraphics[height=4cm]{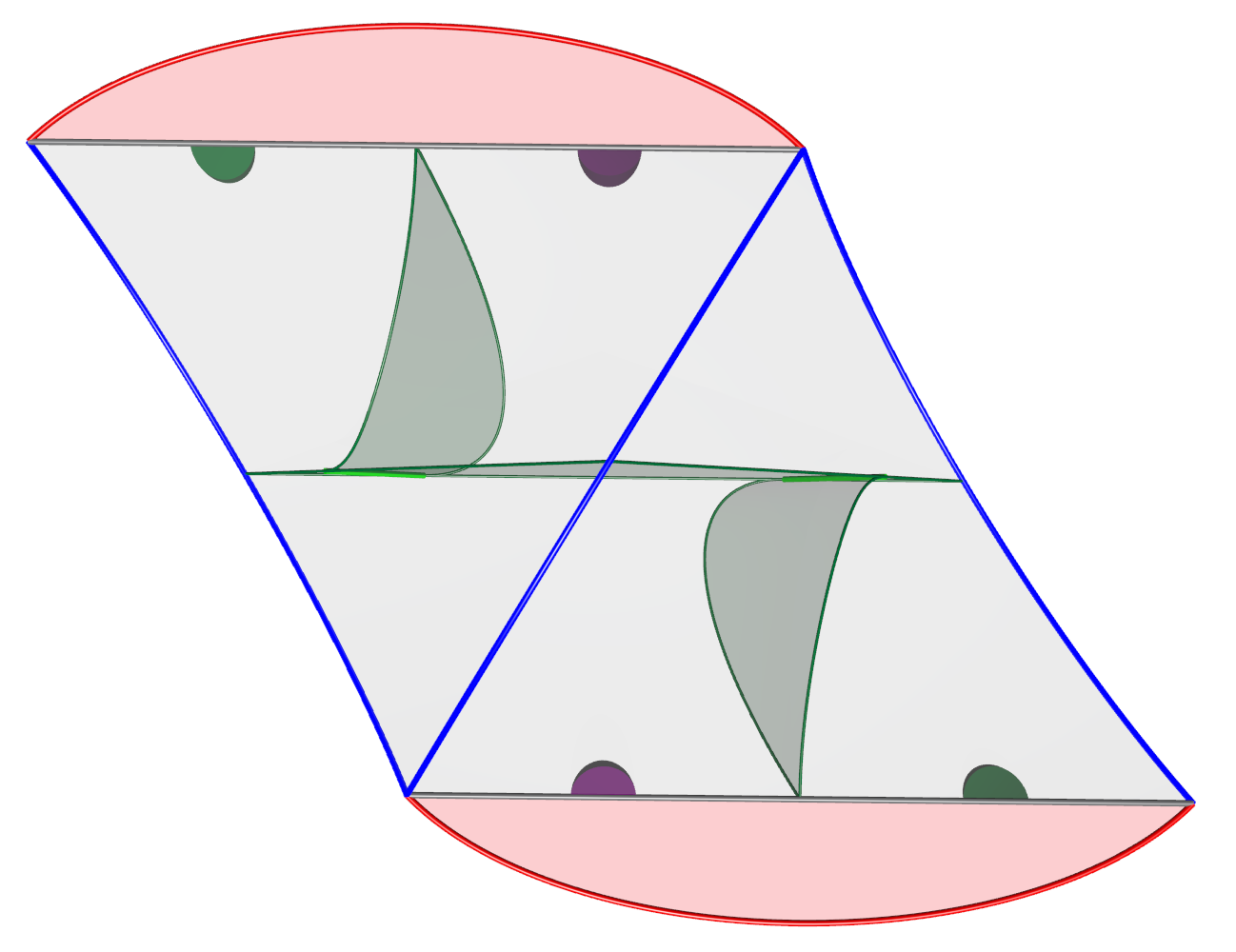}
}
\caption{Straightened $B^t$ in an upper half-tetrahedron (toggle or fan).}
\label{Fig:UpperHalfTet}
\end{figure}

\begin{figure}[htbp]
\centering
\subfloat[Three-quarter view.]{
\includegraphics[height=4cm]{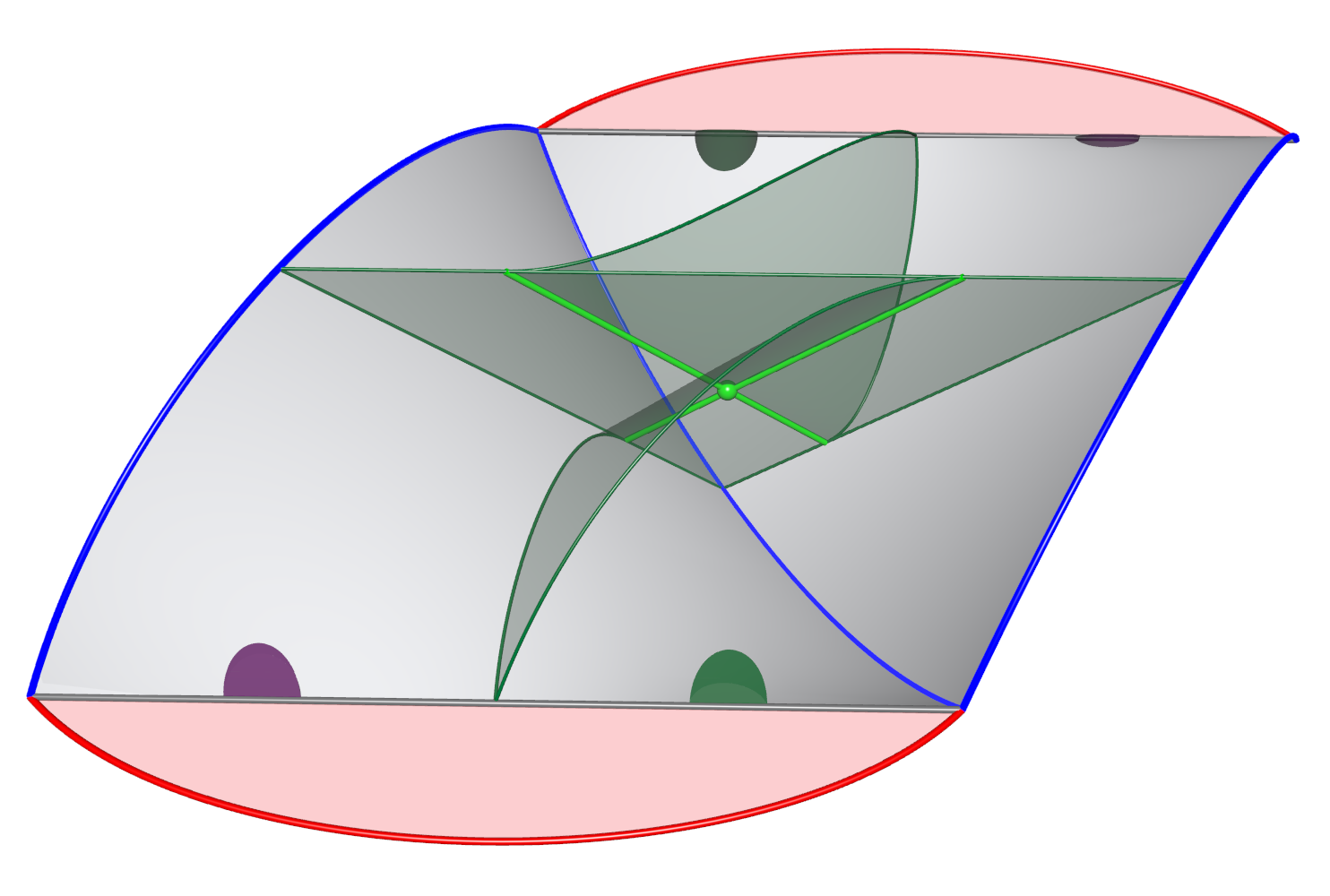}
}
\subfloat[Top view.]{
\includegraphics[height=4cm]{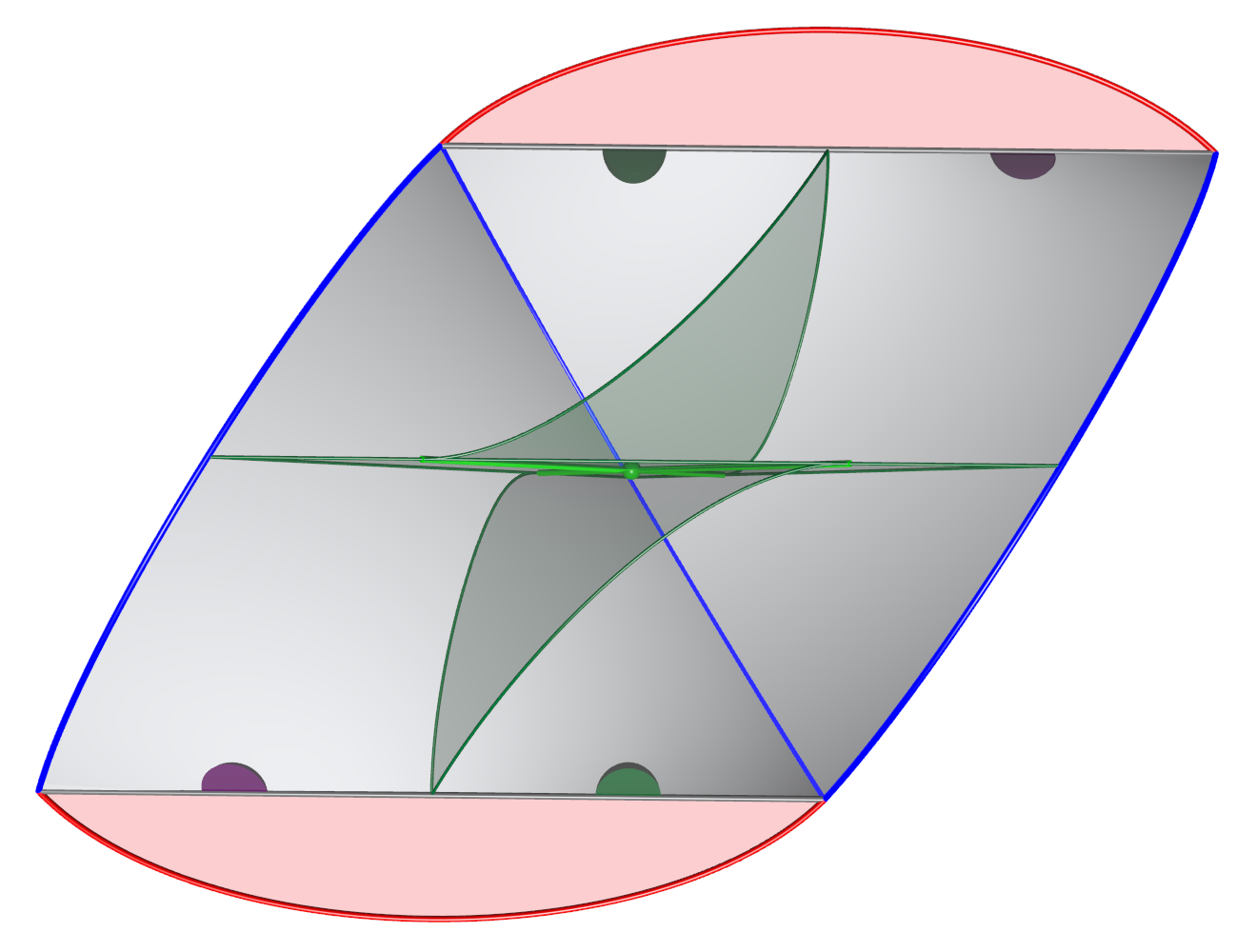}
}
\caption{Straightened $B^t$ in a lower half-tetrahedron (fan).}
\label{Fig:LowerHalfFan}
\end{figure}

\begin{figure}[htbp]
\centering
\subfloat[Three-quarter view.]{
\includegraphics[height=4cm]{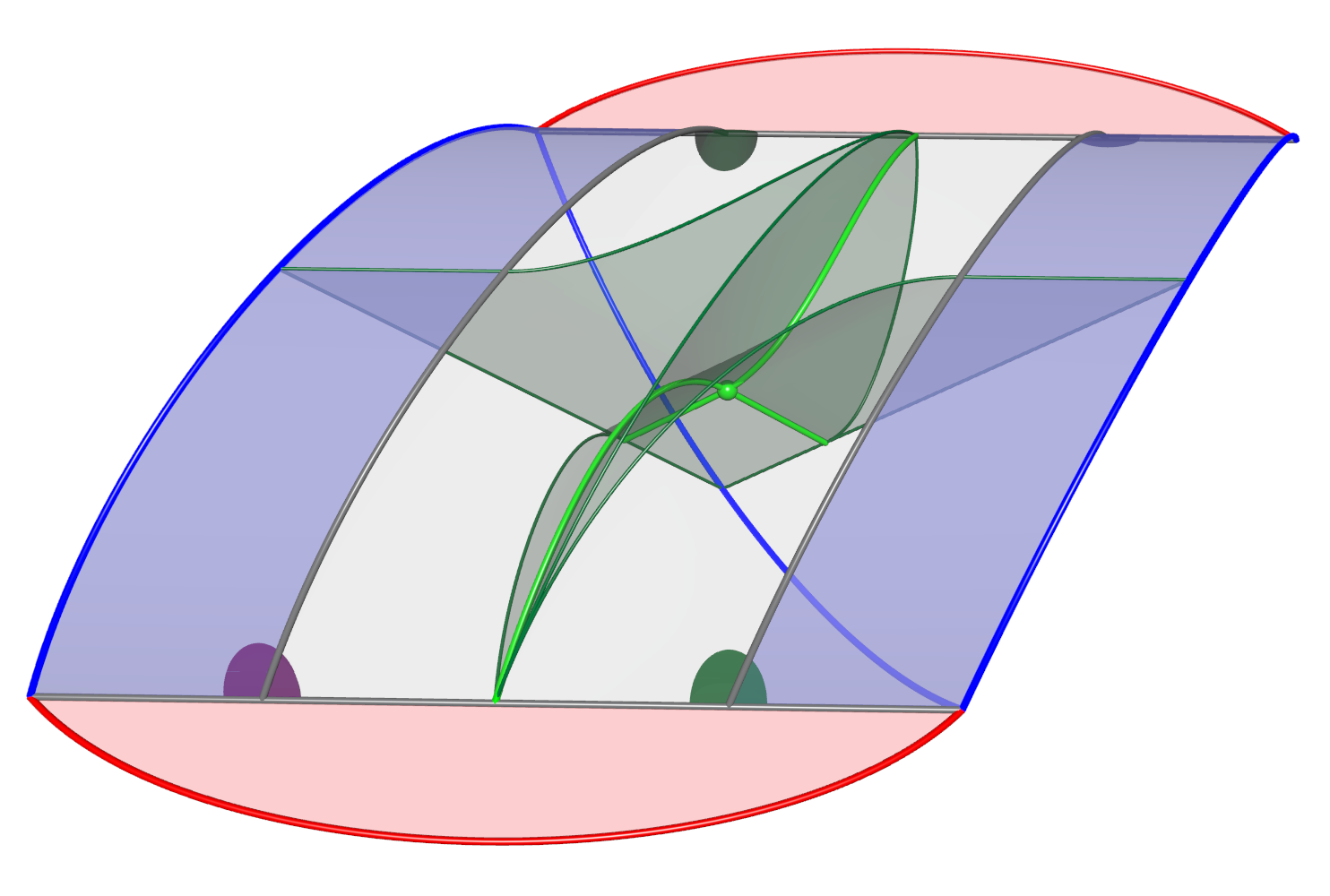}
}
\subfloat[Top view.]{
\includegraphics[height=4cm]{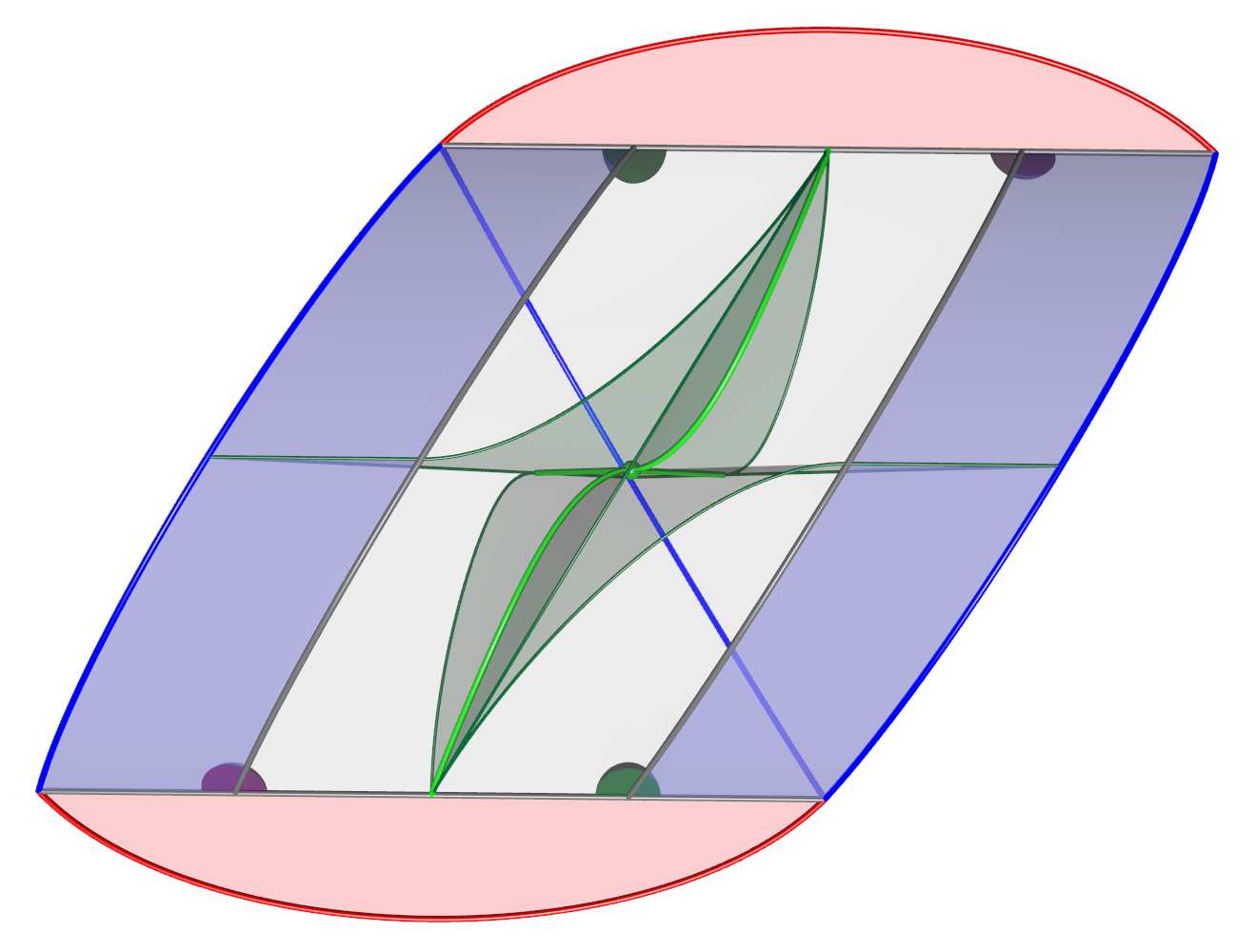}
}
\caption{Straightened $B^t$ in a lower half-tetrahedron (toggle).}
\label{Fig:LowerHalfToggle}
\end{figure}

\begin{figure}[htbp]
\centering
\subfloat[Three-quarter view.]{
\includegraphics[width=0.97\textwidth]{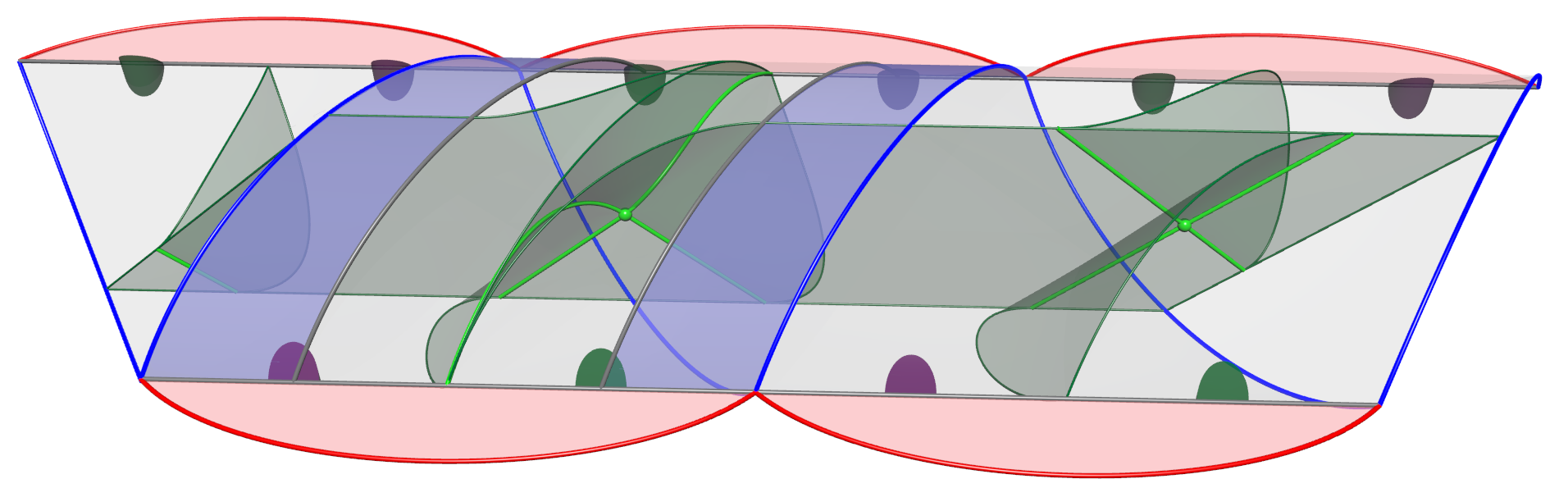}
}

\subfloat[Top view.]{
\includegraphics[width=0.97\textwidth]{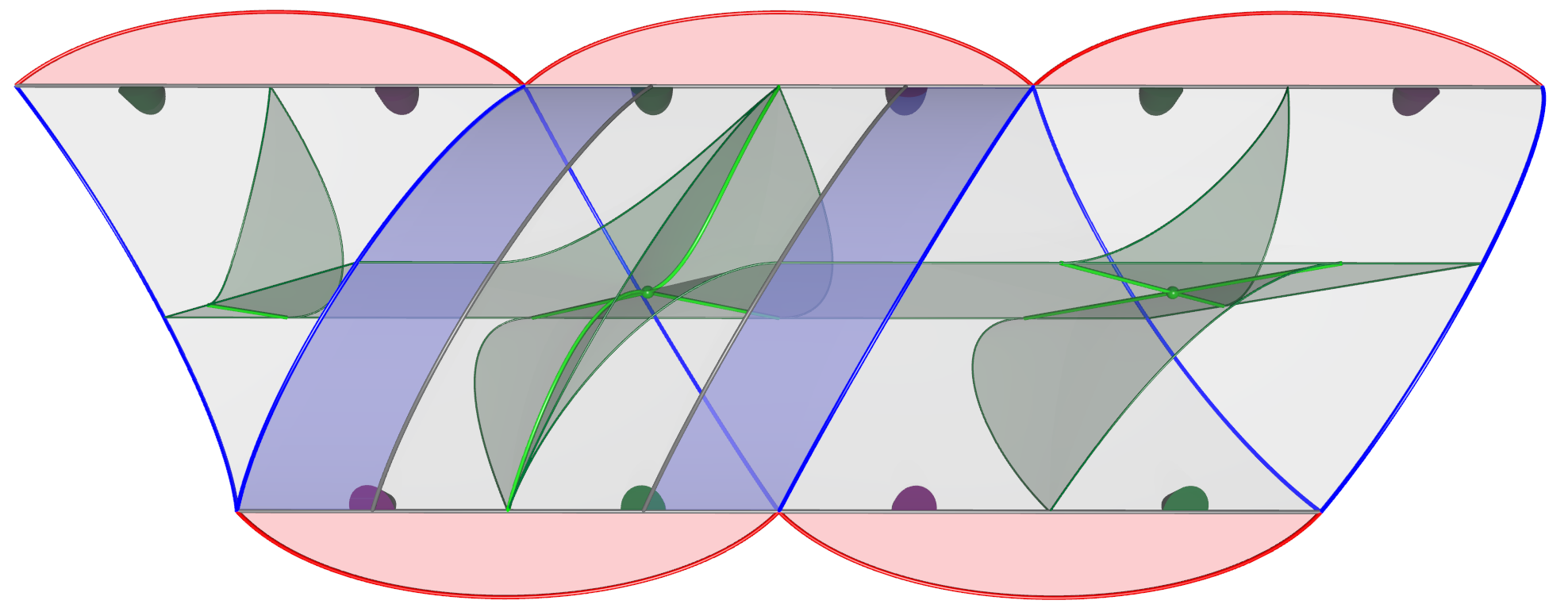}
}
\caption{Straightened $B^\calV$ in a crimped shearing region.}
\label{Fig:StraightenedCrimpedShearingRegion}
\end{figure}

We illustrate our construction with a running example.
The example is \usebox{\BigExVeer}, chosen from the veering census~\cite{GSS19}.
\reffig{m115_straight} shows the result of straightening in this example, in various cross-sections.

\begin{remark}
In our pictures of cross-sections we shade all toggle squares in grey and all crimped bigons the colour of their veering edge.
Along a branch interval of $B^\calV$ within a crimped solid torus, track-cusps are labelled with the same letter.
As we move from an upper boundary to a lower boundary the labels (on track-cusps of $B^\calV$) advance by one letter.
Track-cusps of $B_\calV$ are indicated with small triangles.
\end{remark}

\begin{figure}[htbp]
\subfloat[Blue crimped solid torus.]{
\centering
\labellist
\small\hair 2pt
\pinlabel {$\bdy^+ \Theta^\calV$} [r] at 10 1340
\pinlabel {$\bdy^- \Theta^\calV = \bdy^+ \Theta_\calV$} [r] at 10 760
\pinlabel {$\bdy^- \Theta_\calV$} [r] at 10 178
\endlabellist
\includegraphics[height=15cm]{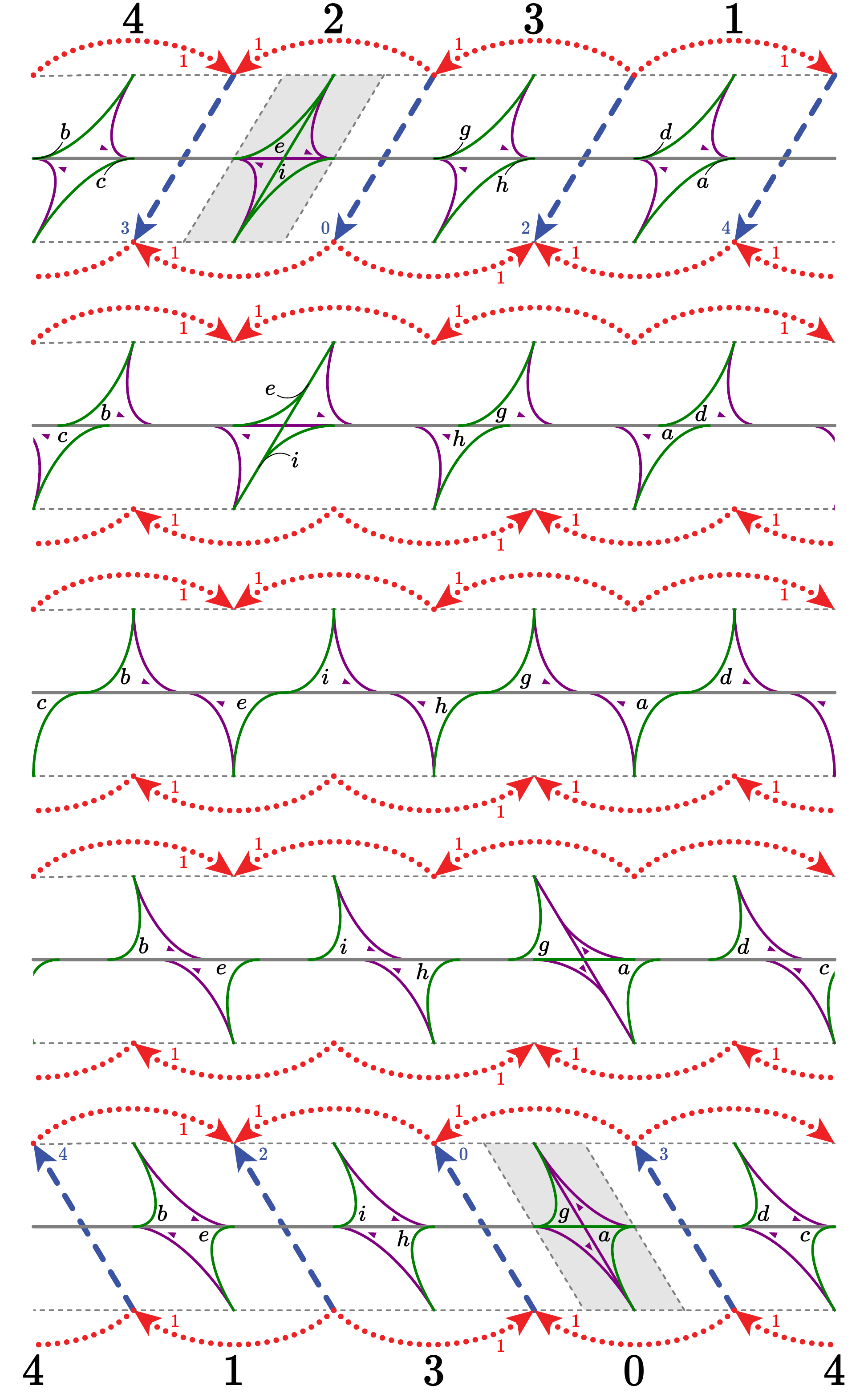}
}
\subfloat[Red.]{
\centering
\includegraphics[height=15cm]{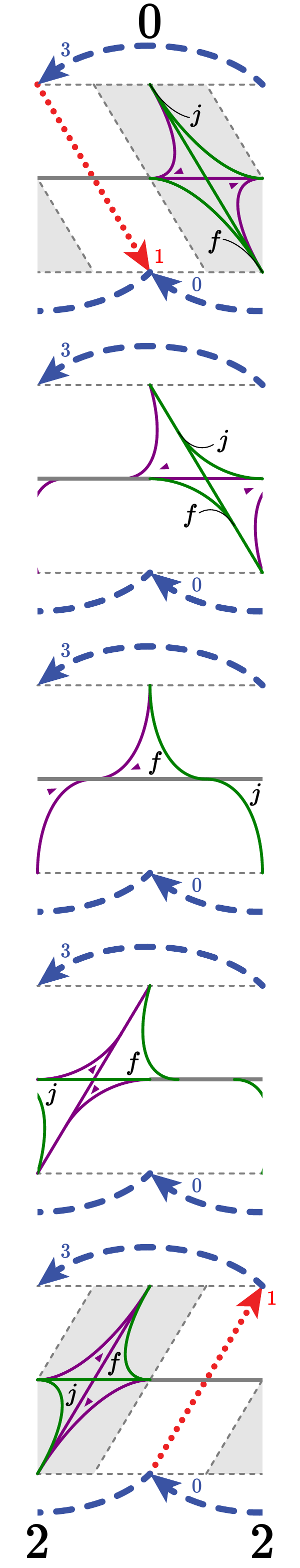}
}
\caption{The intersection of $B^\calV$ (and $B_\calV$), after straightening, with various horizontal cross-sections of the crimped shearing decomposition of \usebox{\BigExVeer}.
Compare with \reffig{fLLQccecddehqrwjj}.
We indicate the position of track-cusps with letters or small triangles; sometimes we use a ``whisker'' pointing from a letter or triangle to the track-cusp itself.
The stations are not drawn.}
\label{Fig:m115_straight}
\end{figure}


\begin{remark}
In \reffig{m115_straight} the upper boundary of the blue crimped solid torus $U$ is glued to the lower boundary of $U$ along the fan squares, by a $180^\circ$ rotation and a (left) shear.
As a result, the blue helical veering edges and the red longitudinal veering edges (adjacent to fan squares) match on the top and bottom of $U$.
The red longitudinal veering edges adjacent to the toggle squares do not match.
This is because they are glued to the red crimped solid torus $V$.
The upper and lower boundaries of $V$ are also glued, by a $180^\circ$ rotation and a (right) shear, along the red crimped bigons.
\end{remark}

\begin{remark}
\label{Rem:TangentsShear}
Suppose that $U$ is a crimped shearing region.
Let $H = \bdy^- U = \bdy^- \Theta_U$ and $K = \bdy^+ U = \bdy^+ \Theta^U$.
Let $\tau^H$ and $\tau^K$ be the intersections of $B^\calV$ with $H$ and $K$, respectively.
So $\tau^H$ and $\tau^K$ are train-tracks.
We arrange matters so that $\tau^H$ meets longitudinal crimped (helical veering) edges of $H$ with a tangent vector which is parallel to the helical veering (longitudinal crimped) edges of $H$; see \reffig{m115_straight}.
We do the same for $\tau^K$.
This ensures that tangent vectors match up when sheared by the gluing maps.

Suppose that $H_s$ parametrises the cross-sections of $U$, with $H_0 = \bdy^- \Theta_U$, with $H_{1/2} = \bdy^+ \Theta_U = \bdy^- \Theta^U$ and with $H_1 = \bdy^+ \Theta^U$.
As $s$ increases from $0$ to $1$, the tangent vectors of branches meeting longitudinal crimped edges shear.
Again, see \reffig{m115_straight}.
\end{remark}

\begin{remark}
Observe that all vertices of $B^\calV$ now lie along the central curves of the middle cross-sections of the crimped shearing regions.
That is, the vertices lie in the intersection of
\begin{itemize}
\item
the middle cross-section $\bdy^- \Theta^\calV = \bdy^+ \Theta_\calV$ and
\item
the mid-surface $\calS = \calS_R \cup \calS_B$. \qedhere
\end{itemize}
\end{remark}

\begin{definition}
\label{Def:Projection}
Suppose that $U$ is a crimped shearing region.
Let $A = A(U)$ be the mid-band in $U$.
Let $U'$ equal $U$ minus its longitudinal crimped edges.
We define the \emph{shearing projection} $\rho_U \from U' \to A$ as follows.
\begin{itemize}
\item
In every cross-section, $\rho_U$ projects along lines in bigon coordinates.
\item
In $\bdy^\pm U$ these lines are parallel to the helical veering edges.
\item
In the cross-sections between the upper and lower boundaries of $U$ the direction of projection interpolates linearly.
\end{itemize}
Suppose that $U$ and $V$ are crimped shearing regions of the same colour with $\bdy^+ U$ intersecting $\bdy^- V$.
Then, by construction, $\rho_U$ and $\rho_V$ agree on $\bdy^+ U' \cap \bdy^- V'$ (a union of fan squares and crimped bigons).
So, for any union $\calU$ of crimped shearing regions, all of the same colour, we may define $\rho_\calU = \bigcup \rho_U$ where the union ranges over $U$ in $\calU$.
\end{definition}

\begin{remark}
Suppose that $\calU$ is a union of crimped shearing regions, all of the same colour.
Suppose that $C$ is the intersection of the branch lines of $B^\calV$ (in straightened position) with $\calU$.
We draw the projection $\rho_\calU(C)$ on the mid-surface $\calS_\calU$ in a particular example in \reffig{m115_side_straight}.

In straightened position, a branch interval (that is, a component of $C$) lies in the mid-surface until slightly below the toggle square it exits through.
Thus that sub-interval and its projection under $\rho_\calU$ agree and are (almost) straight.
Just below the exiting toggle square, the track-cusp continues to move at constant speed (with respect to the $x$--coordinate).
However, the shearing of the projections exactly cancel that motion.
As a result, the projection of the remaining sub-interval is (almost) vertical.
Finally, we note that the branch intervals and their projections are smooth curves.
What we have drawn in \reffig{m115_side_straight} is thus a (highly accurate) approximation of the actual position.
\end{remark}

\begin{figure}[htbp]
\subfloat[Blue mid-surface.]{
\centering
\includegraphics[height=7cm]{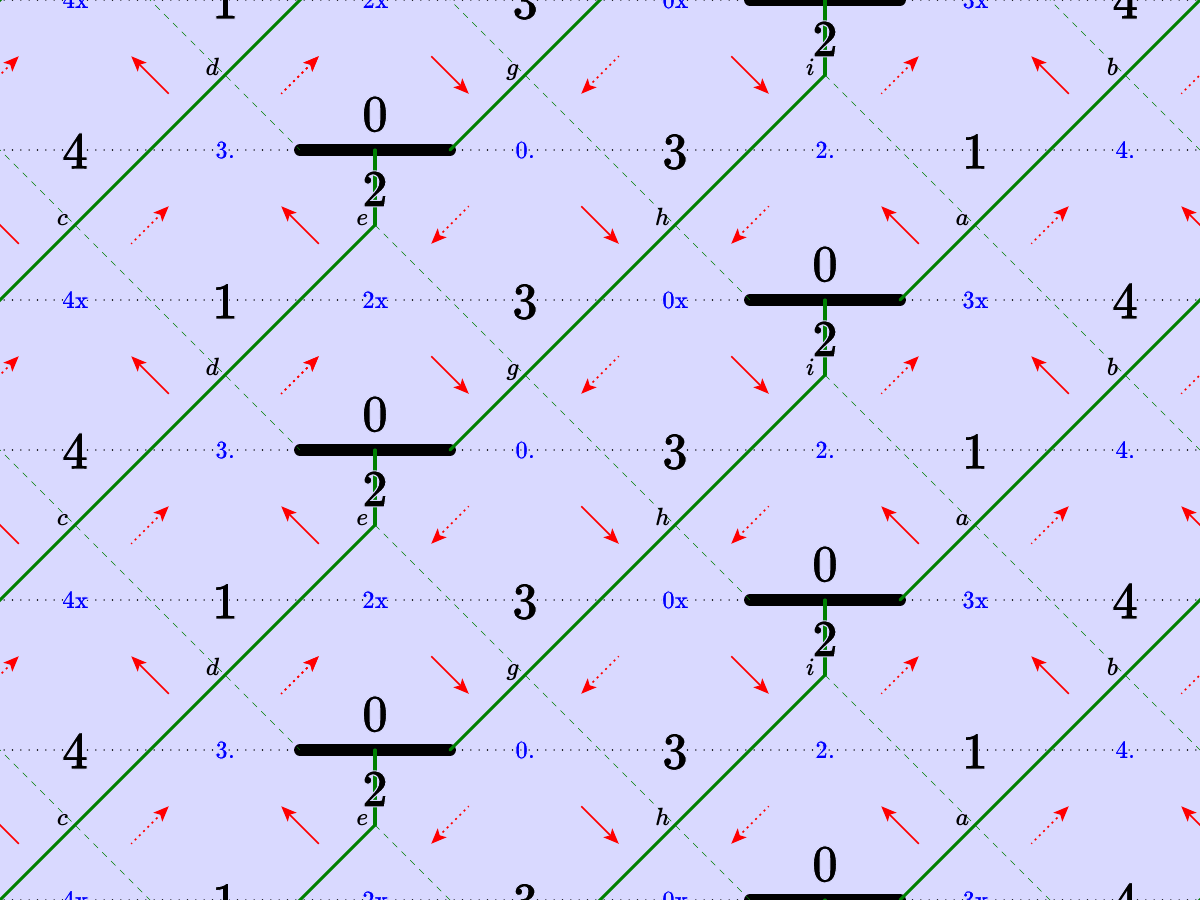}
}
\subfloat[Red.]{
\centering
\includegraphics[height=7cm]{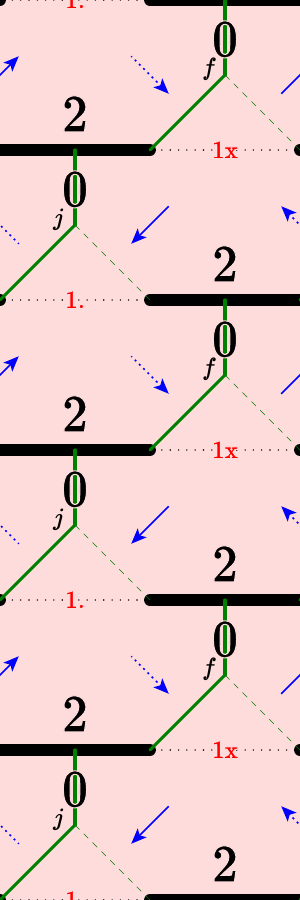}
}
\caption{The branch lines of $B^\calV$, after straightening, projected to (a $\ZZ^2$ cover of each component of) the mid-surface $\calS$.
Compare with \reffig{fLLQccecddehqrwjj}.
Note that we have drawn six fundamental domains, stacked vertically.
The projection is drawn with a solid (dotted) line exactly as the corresponding track-cusps is in front of (behind) the branch surface.}
\label{Fig:m115_side_straight}
\end{figure}

\begin{remark}
\label{Rem:StraightenedDynamic}
As noted in \refcor{DualDynamic} the branched surface $B^\calV$, when in dual position, is dynamic.
Straightening makes parts of $B^\calV$ vertical.
However, the branch locus remains transverse, and not orthogonal, to vertical.
Thus the straightened $B^\calV$ is again dynamic.
\end{remark}

\subsection{Shrinking}
\label{Sec:Shrinking}

Next we \emph{shrink}: in each crimped shearing region $U$,
we form a very small collar $\Gamma^U$ of $\bdy^+ U$, obtained as a union of horizontal cross-sections $Z(p, U)$.
Note that $\Gamma^U$ is disjoint from the vertices of $B^\calV$.
We now move $B^\calV$ by a proper isotopy of $U$ which preserves $x$ and $y$ coordinates (in bigon coordinates) and permutes the cross-sections $Z(p, U)$.
The isotopy carries the bottom of $\Gamma^U$ downwards to $\bdy^- \Theta^U$ and evenly redistributes the cross-sections below $\Gamma^U$ inside of $\Theta_U$.

Before the isotopy, $B^\calV$ was transverse to the equatorial squares.
After the isotopy, $B^\calV$ is almost vertical in all of $\Theta^U$.
The intersections of $B^\calV$ with $\bdy^+ U$ and $\bdy^- U$
are unchanged by the shrinking isotopy.
Note that the shrinking isotopy maintains the $180^\circ$ symmetry of the branched surfaces $B^t$.
In \reffig{m115_shrink} we show the intersection of the shrunken $B^\calV$ (and $B_\calV$) with various horizontal cross-sections.

\begin{figure}[htbp]
\subfloat[Blue crimped solid torus.]{
\centering
\labellist
\small\hair 2pt
\pinlabel {$\bdy^+ \Theta^\calV$} [r] at 10 1340
\pinlabel {$\bdy^- \Theta^\calV = \bdy^+ \Theta_\calV$} [r] at 10 760
\pinlabel {$\bdy^- \Theta_\calV$} [r] at 10 178
\endlabellist
\includegraphics[height=15cm]{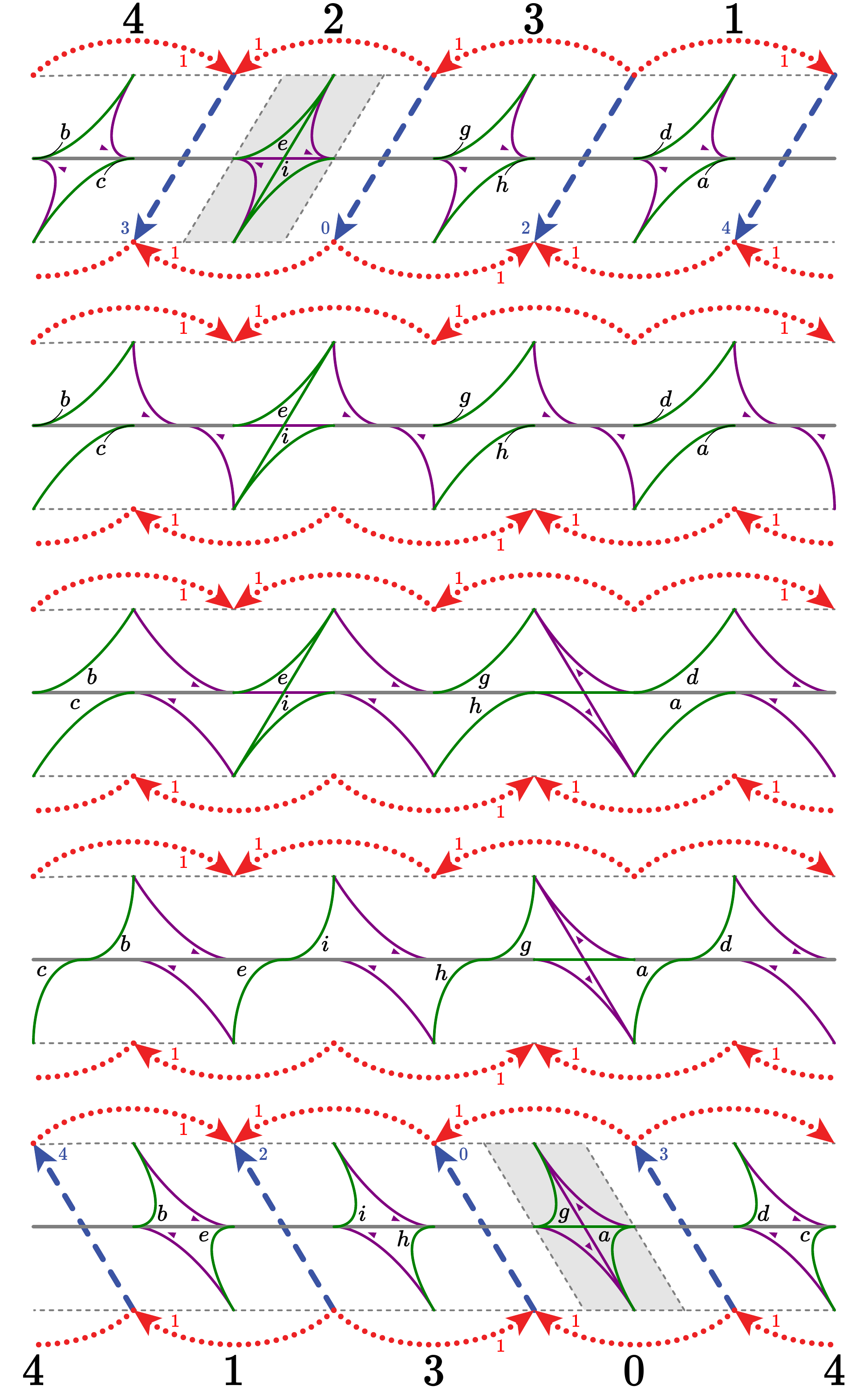}
}
\subfloat[Red.]{
\centering
\includegraphics[height=15cm]{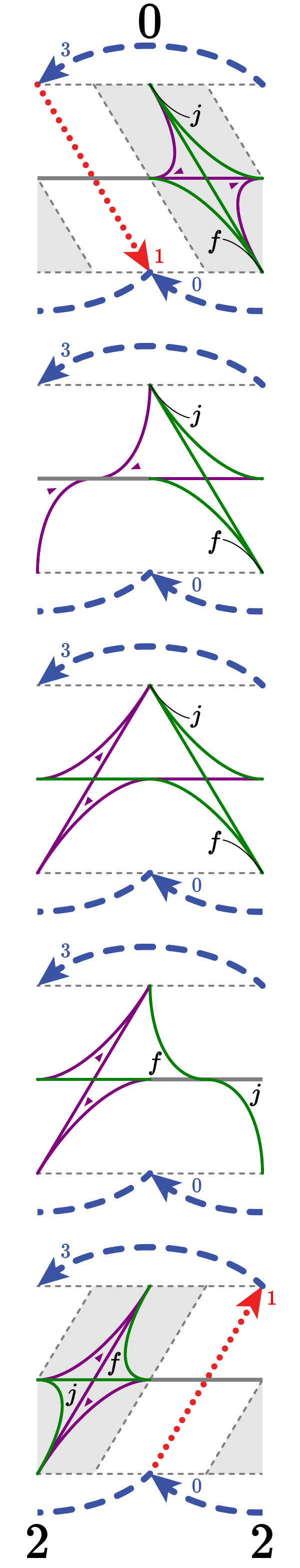}
}
\caption{The intersection of $B^\calV$ (and $B_\calV$), after shrinking, with various horizontal cross-sections of the crimped shearing decomposition of \usebox{\BigExVeer}.
Compare with \reffig{m115_straight}.
}
\label{Fig:m115_shrink}
\end{figure}

\begin{figure}[htbp]
\subfloat[Blue mid-surface.]{
\centering
\includegraphics[height=7cm]{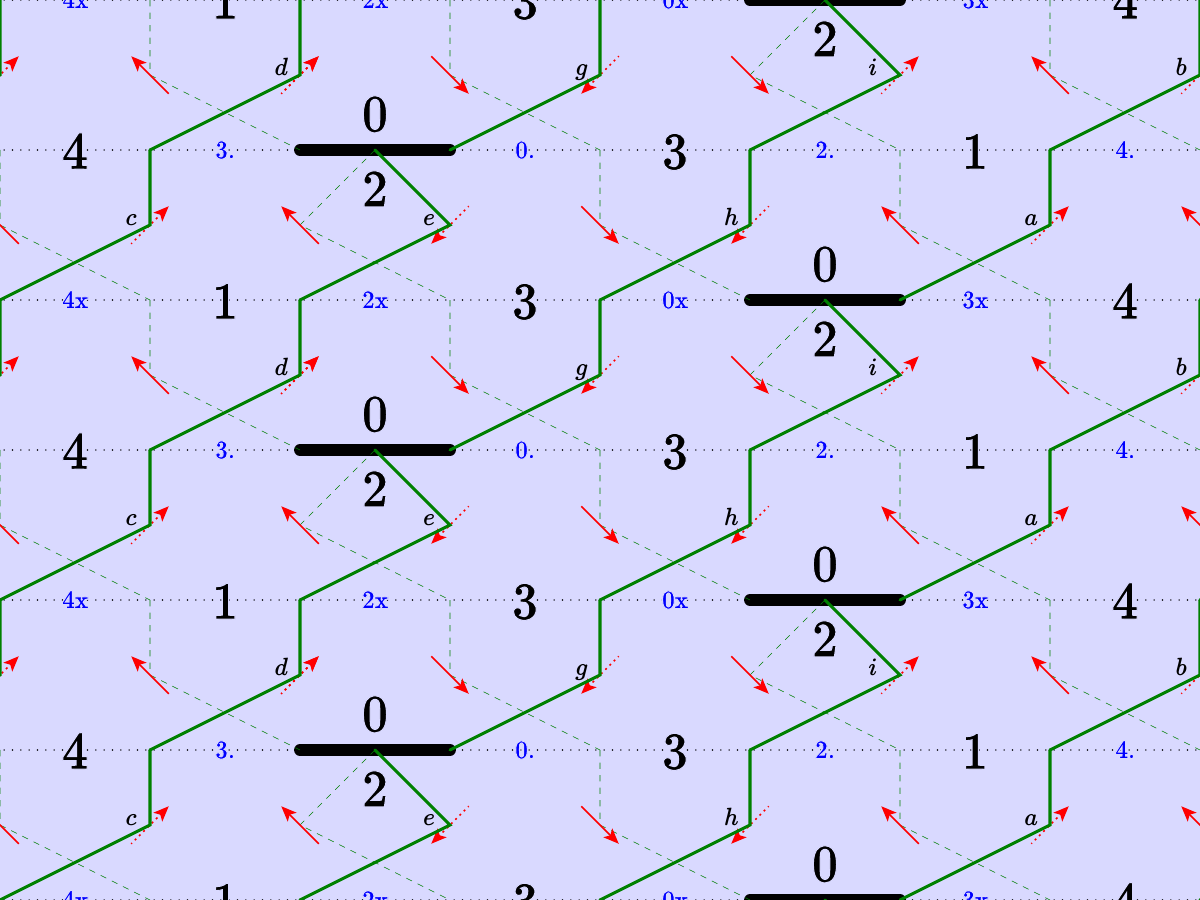}
}
\subfloat[Red.]{
\centering
\includegraphics[height=7cm]{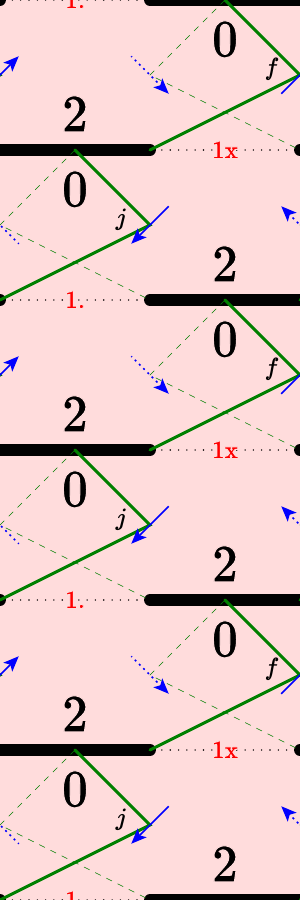}
}
\caption{The branch lines of $B^\calV$, after shrinking, projected to the mid-surfaces.
Compare with \reffig{m115_side_straight}.}
\label{Fig:m115_side_shrunk}
\end{figure}

\begin{remark}
\label{Rem:ShrunkenTangentsShear}
Note that the shearing of tangent vectors, as in \refrem{TangentsShear}, now occurs in $\Theta_\calV$ for $B^\calV$ (and in $\Theta^\calV$ for $B_\calV$).
\end{remark}

\begin{remark}
\label{Rem:ShrunkenDynamic}
Shrinking permutes cross-sections; thus by \refrem{StraightenedDynamic} the shrunken branched surface $B^\calV$ is again dynamic.
\end{remark}

\section{Parting}
\label{Sec:Parting}

Here we define the \emph{parting isotopies}.
These are applied to the upper and lower branched surfaces $B^\calV$ and $B_\calV$, placing them in \emph{parted position}.
These isotopies are almost local:
in each tetrahedron, outside of the stations, they depend only on the combinatorics of that tetrahedron and its immediate neighbours.
The way in which the branched surface intersects crimped edges inside of stations is more delicate and is dealt with in \refsec{Junctions}.


Concentrating on $B^\calV$, we now sketch the construction before giving the details.
We start in shrunken position (shown in \reffig{m115_shrink}).
In each cross-section of $\Theta^\calV$, and near each crimped edge, we will move $B^\calV$ towards a chosen station (corner) of the relevant toggle square.
We will also isotope branches of $B^\calV$ in cross-sections of $\Theta^\calV$ to be (almost) line segments (in bigon coordinates).
As for shrunken position, the parted position of $B^\calV$ in $\Theta^\calV$ will be almost a product.

This done, we will move $B^\calV$ downward in $\Theta_\calV$.
This makes the intersection of $B^\calV$ with the cross-sections into a sequence of train-tracks as follows.
As they move up through $\Theta_\calV$ they first perform a \emph{splitting} of the  track-cusps along their \emph{parting routes}.
They next perform a \emph{graphical isotopy} where the track-cusps are (almost) motionless and the branches straighten to become (almost) line segments.

The branched surface $B_\calV$ moves in a similar way but swapping $\Theta^\calV$ and $\Theta_\calV$.
The combined procedure of splitting along routes and graphical isotopy will be used once (in space) in this section and three more times (twice in time and once in space) in \refsec{Draping}.
We use these to fill in the isotopy from parted position to their final \emph{draped position}.

\begin{figure}[htbp]
\subfloat[Blue crimped solid torus.]{
\label{Fig:m115_prepared_blue}
\centering
\labellist
\small\hair 2pt
\pinlabel {$\bdy^+ \Theta^\calV$} [r] at 10 1340
\pinlabel {$\bdy^- \Theta^\calV = \bdy^+ \Theta_\calV$} [r] at 10 760
\pinlabel {$\bdy^- \Theta_\calV$} [r] at 10 178
\endlabellist
\includegraphics[height=15cm]{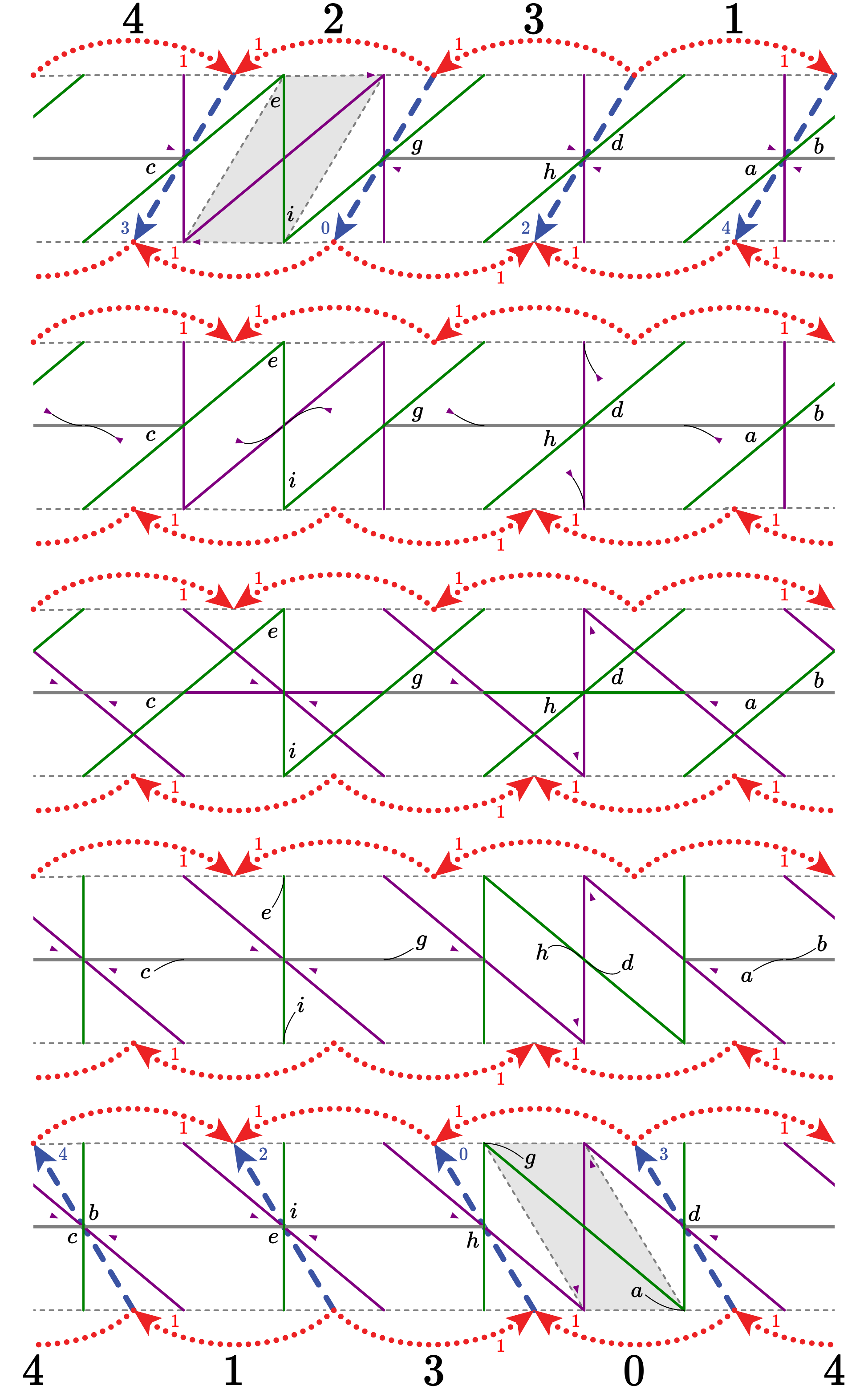}
}
\subfloat[Red.]{
\label{Fig:m115_prepared_red}
\centering
\includegraphics[height=15cm]{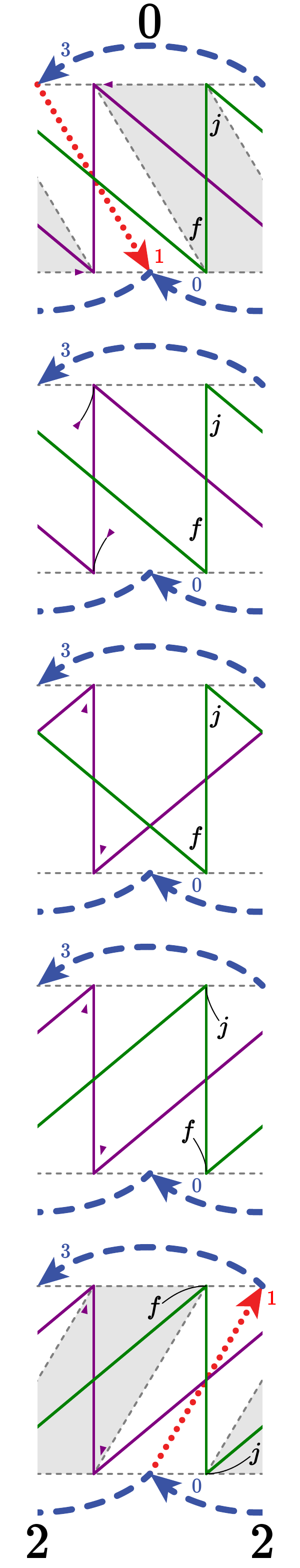}
}
\caption{The intersection of $B^\calV$ (and $B_\calV$), after parting, with various cross-sections of the crimped shearing decomposition of \usebox{\BigExVeer}.
The branched surfaces intersect the longitudinal crimped edges within stations.
As in \reffig{m115_shrink}, the branched surface $B^\calV$ is almost vertical in $\Theta^\calV$.
Likewise, $B_\calV$ is almost vertical in $\Theta_\calV$.
}
\label{Fig:m115_prepared}
\end{figure}

\begin{figure}[htbp]
\subfloat[Blue mid-surface.]{
\centering
\includegraphics[height=7cm]{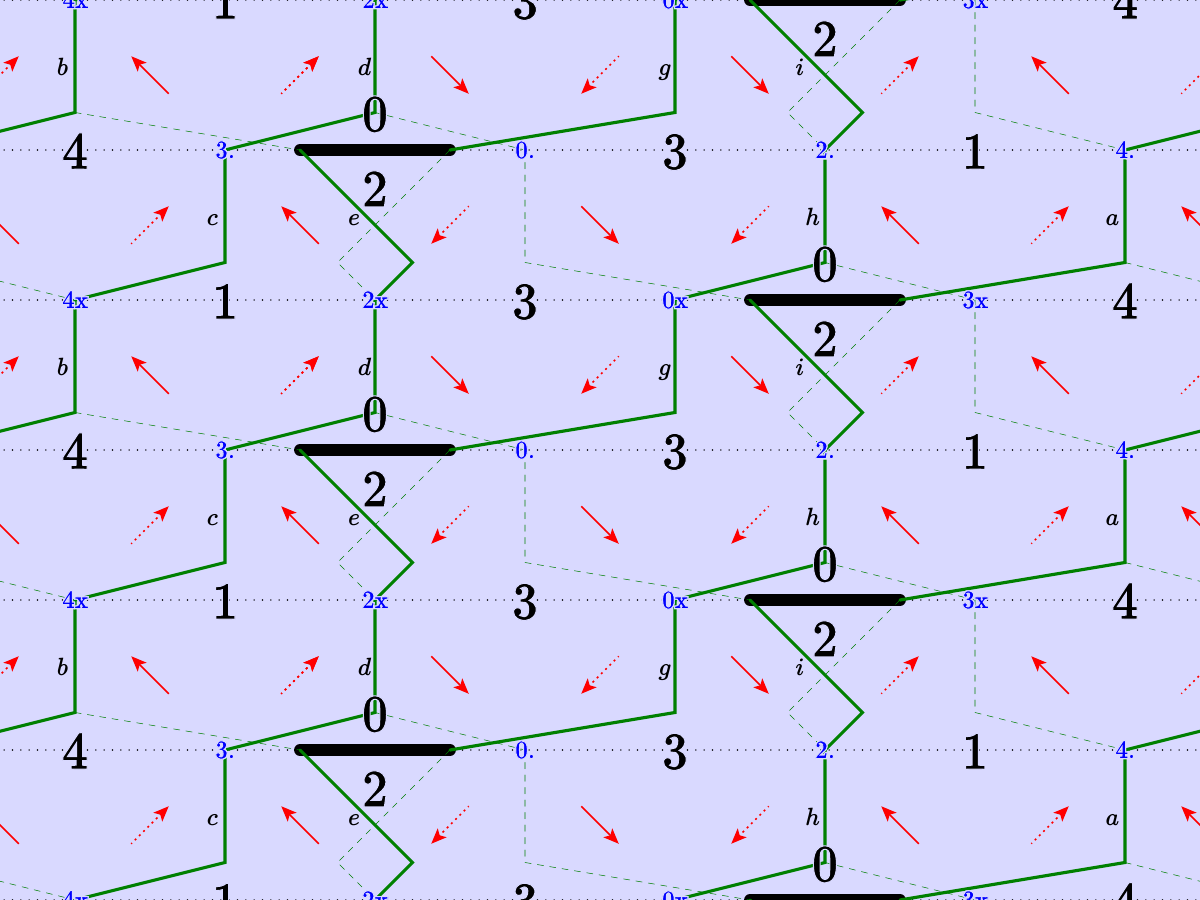}
}
\subfloat[Red.]{
\centering
\includegraphics[height=7cm]{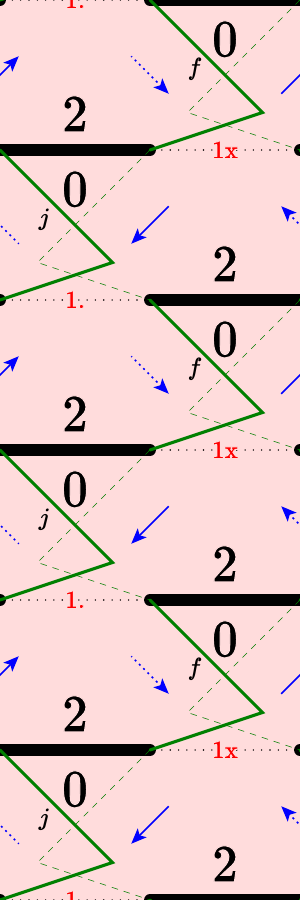}
}
\caption{The branch lines of $B^\calV$, after parting, projected to the mid-surfaces.
Compare with \reffig{m115_side_shrunk}.}
\label{Fig:m115_side_parted}
\end{figure}

\subsection{Parting in \texorpdfstring{$\Theta^\calV$}{Theta sup V}}
\label{Sec:PartingUp}

We now describe the parting isotopy in $\Theta^\calV$.

Suppose that $U$ is a crimped blue shearing region.
Suppose that $e'$ is a longitudinal crimped edge for $U$.
Suppose that $e$ is the associated red veering edge and let $C$ be the crimped bigon which $e$ and $e'$ cobound.
Suppose that $S$ is the upper toggle square meeting $e'$.
We equip $C$ with the anti-clockwise orientation, as viewed from above.
This induces orientations on $e$ and $e'$.
Let $c = C \cap B^\calV$.
The parting isotopy in $\Theta^U$ fixes $c \cap e$ and moves $c \cap e'$ along $e'$, \emph{against} the orientation of $e'$ (given just above), until it arrives at the station cutting off the cusp of the toggle square $S$.
(If, instead, $U$ is red, then we move $c \cap e'$ along $e'$, \emph{following} the orientation of $e'$, again until it arrives at its station.)
To see this motion, compare top lines of Figures~\ref{Fig:m115_shrink} and~\ref{Fig:m115_prepared}.

In $\bdy^+ U$ we also move track-cusps outwards in fan squares until they arrive close to the midpoint of a helical edge.
In other cross-sections of $\Theta^U$ we do the same, but now moving track-cusps until they almost meet the projection (in bigon coordinates) of the midpoint of a helical edge.

\begin{remark}
\label{Rem:AlmostProduct}
Thus $B^\calV$ is almost a product in $\Theta^U$.
The track-cusps move very slowly forward in cross-sections to preserve dynamism.
Track-cusps outside of stations all move at the same speed.
The motion of track-cusps inside of stations is described in \refsec{Junctions}.
Where a train-track meets a longitudinal crimped edge, its tangent vector remains parallel to the (projection of the) helical veering edges in $\bdy^+U$.
See \refrem{ShrunkenTangentsShear}.
\end{remark}

\begin{remark}
\label{Rem:UpGivesDown}
Parting in $\Theta^\calV$ determines parted position in $\bdy^- \Theta_\calV$.
The general picture is shown in \reffig{MagnifyParted}.
The running example shows a particular case -- see the bottom row of \reffig{m115_prepared}.
\end{remark}

\begin{remark}
\label{Rem:JunctionSizes}
Once in parted position, in any cross-section the train-track intersects crimped edges only within stations.
Again, the exact location where a branch intersects a crimped edge within a station is set in \refsec{Junctions}.
\end{remark}

\begin{remark}
\label{Rem:AlmostLocal}
Outside of stations, the parting isotopy in $\Theta^U$ depends only on whether a branch is below a toggle square or a fan square in $\bdy^+U$.
\end{remark}

\begin{figure}[htbp]
\subfloat[Fan square in $\bdy^+ U$.]{
\centering
\includegraphics[width =
0.48\textwidth]{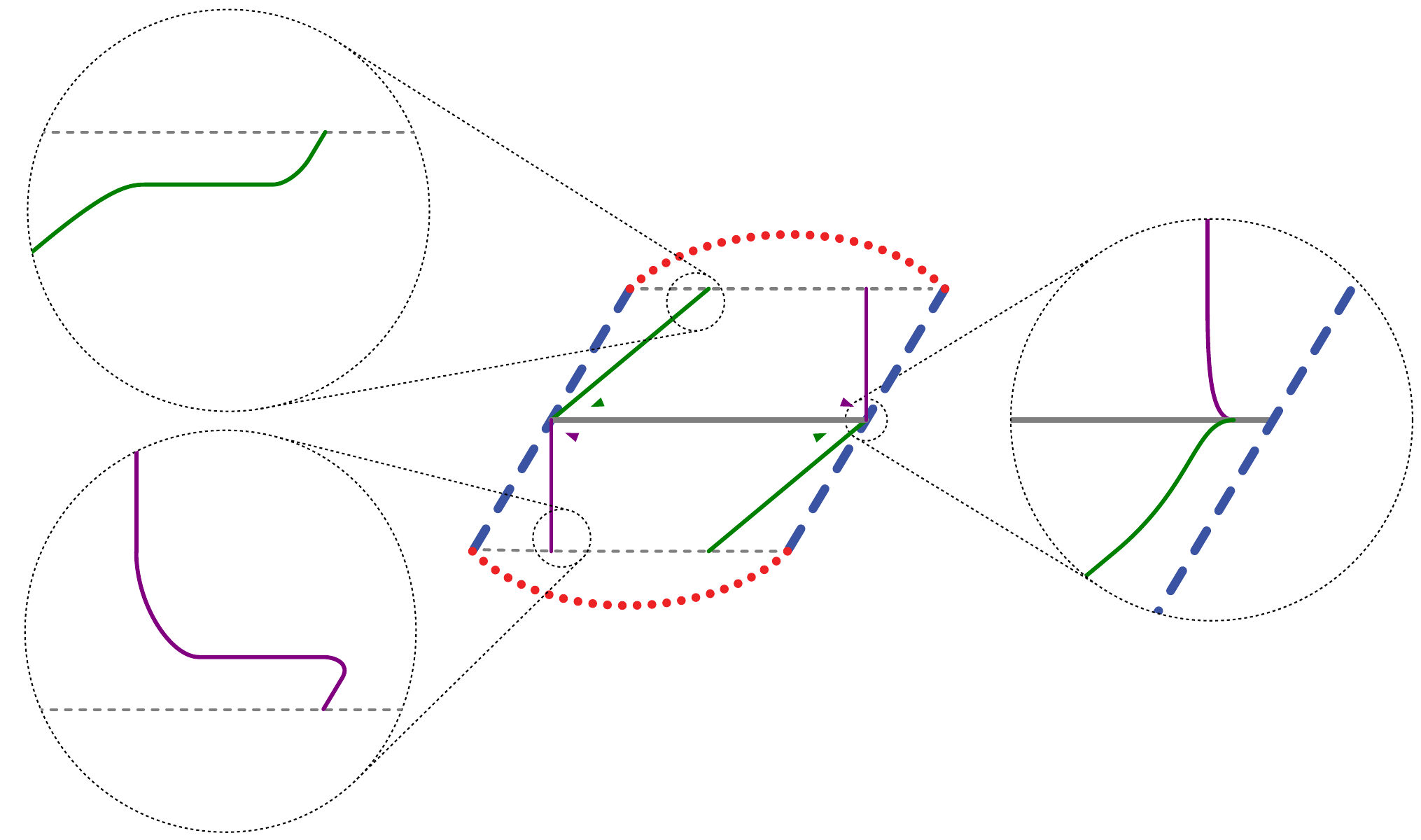}
\label{Fig:MagnifyPartedTopFan}
}
\subfloat[Fan square in $\bdy^- U$.]{
\centering
\includegraphics[width =
0.48\textwidth]{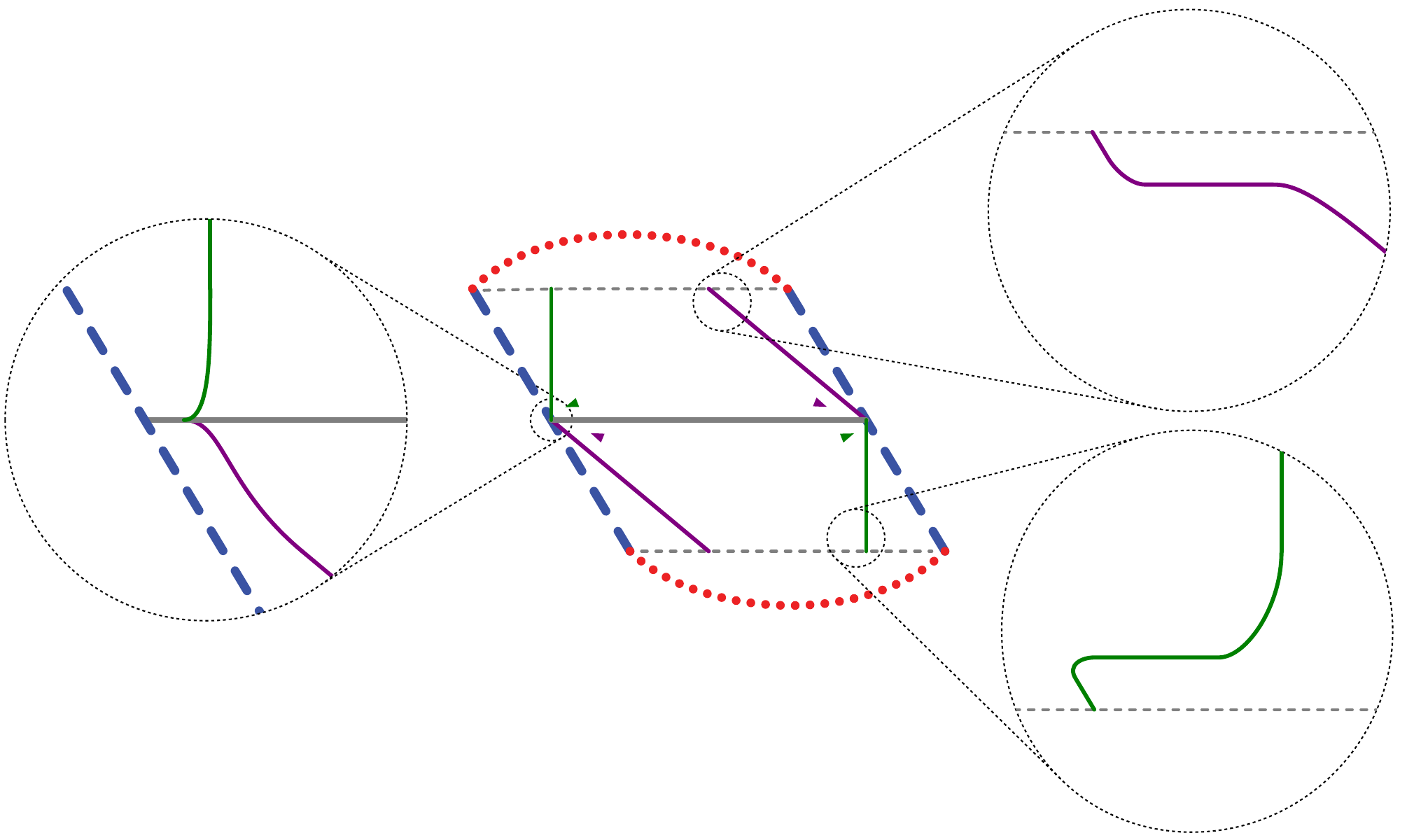}
\label{Fig:MagnifyPartedBottomFan}
}

\subfloat[Toggle square in $\bdy^+ U$.]{
\centering
\includegraphics[width = 0.48\textwidth]{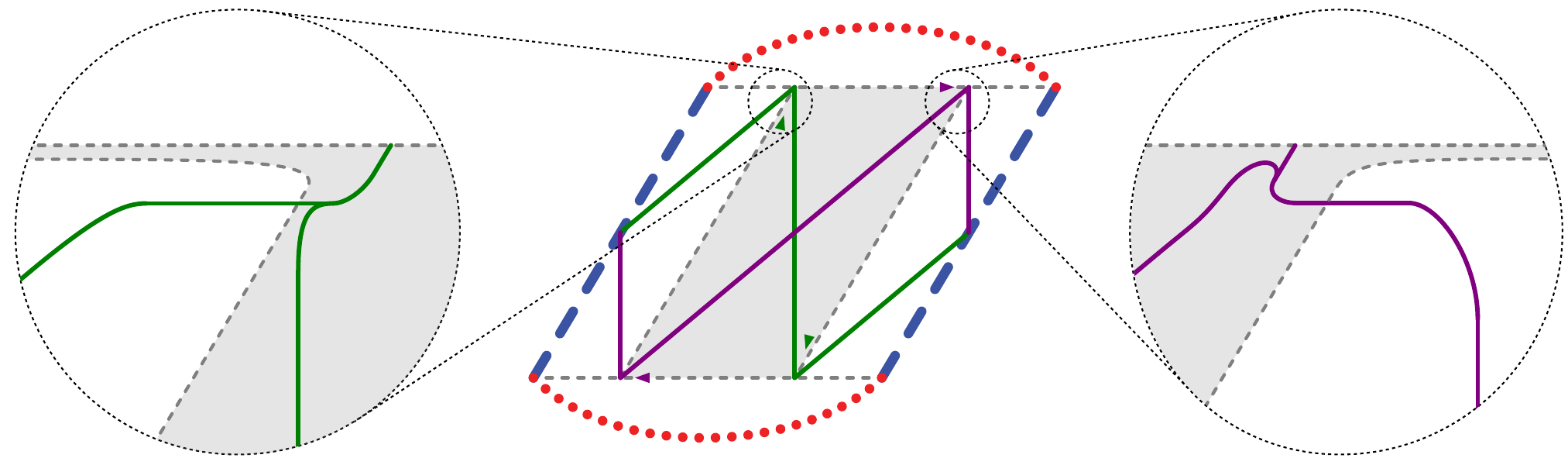}
\label{Fig:MagnifyPartedTopToggle}
}
\subfloat[Toggle square in $\bdy^- U$.]{
\centering
\includegraphics[width = 0.48\textwidth]{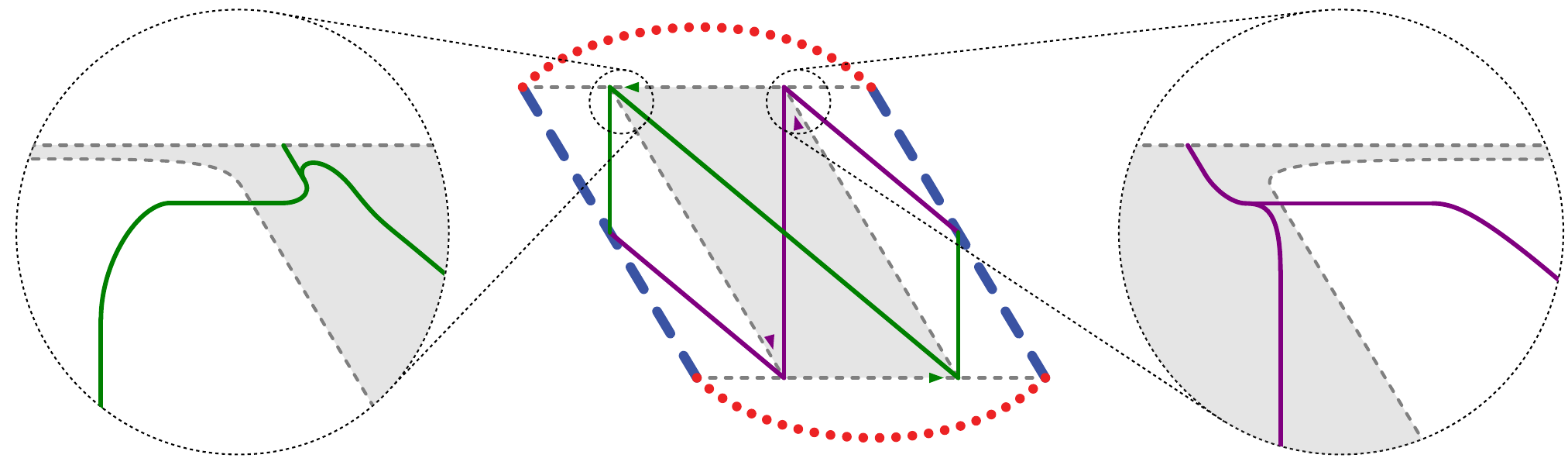}
\label{Fig:MagnifyPartedBottomToggle}
}
\caption{In parted position the intersections of $B^\calV$ and $B_\calV$ with cross-sections are line segments except for (a) inside of the stations and (b) very close to the midpoints of helical edges.
Note that here $U$, the containing crimped shearing region, is blue.
}
\label{Fig:MagnifyParted}
\end{figure}

\subsection{Graphical tracks and isotopies}

In order to organise isotopies in cross-sections in $\Theta_\calV$ we require the following.

\begin{definition}
\label{Def:GraphicalTrack}
Suppose that $U$ is a crimped shearing region.
Suppose that $H$ is a cross-section in $\Theta_U$.
We consider the foliation of $H$ by lines (in bigon coordinates) parallel to the helical veering edges in $\bdy^- U$.
We say that a smooth arc $\alpha$ in $H$ is \emph{lower graphical} if $\alpha$ is transverse to this foliation (except possibly at its endpoints).
Suppose that $\tau$ is a train-track in $H$.
If all branches of $\tau$ are lower graphical then we say that $\tau$ is \emph{lower graphical}.

We define \emph{upper graphical} for cross-sections in $\Theta^U$ using the helical veering edges in $\bdy^+ U$.
If a track is upper or lower graphical we simply call it \emph{graphical}.
\end{definition}

The definition implies the following.

\begin{lemma}
\label{Lem:SplittingPreservesGraphical}
Suppose that $\tau$ is a graphical train-track.
Suppose that $\alpha$ is a train route in $\tau$.
Then the result of splitting $\tau$ along $\alpha$, in a small neighbourhood of $\alpha$, is again graphical. \qed
\end{lemma}

\begin{lemma}
\label{Lem:PartedGraphical}
Suppose that $B^\calV$ is in parted position in $\Theta^\calV$ (and thus in $\bdy^- \Theta_\calV$).
Suppose that $H$ is either a cross-section in $\Theta^\calV$ or in $\bdy^- \Theta_\calV$.
Then the track $\tau^H = B^\calV \cap H$ is lower graphical.
\end{lemma}

\begin{proof}
Suppose that $H$ is a cross-section in $\bdy^+ \Theta^\calV$.
As discussed in \refrem{AlmostLocal}, parted position of $\tau^H$ (outside of stations) is defined locally.
Also, as shown in \reffig{MagnifyParted} all branches of $\tau^H$ outside of stations are straight and not parallel to the lower helical slope.
Thus, outside of stations, all branches of are lower graphical.
Inside of stations branches of $\tau^H$ are laid out according to Figures~\ref{Fig:MagnifyPartedTopFan} and~\ref{Fig:MagnifyPartedTopToggle};
these are also lower graphical.

A similar argument, referring instead to Figures~\ref{Fig:MagnifyPartedBottomFan} and~\ref{Fig:MagnifyPartedBottomToggle}, deals with the case where $H = \bdy^- \Theta_\calV$.

Finally suppose that $H$ is a cross-section inside of $\Theta^\calV$.
By \refrem{AlmostProduct}, the train-track $\tau^H$ is a projection (in bigon coordinates) of a (very slight) folding of the train-track in $\bdy^+ \Theta^\calV$.
Thus $\tau^H$ is again lower graphical.
\end{proof}

\begin{definition}
\label{Def:GraphicalIsotopy}
Suppose that $H$ is a cross-section.
Suppose that $b$ and $b'$ are arcs in $H$ transverse to the lower foliation.
Suppose that $b$ and $b'$ have the same endpoints.
The \emph{graphical isotopy} of $b$ to $b'$ is as follows.
For every leaf $\ell$ of the lower foliation intersecting $b$, we move the point $b \cap \ell$ to the point $b' \cap \ell$, along $\ell$, at constant speed so that its journey takes the entire time of the isotopy.
\end{definition}

\subsection{Junctions}
\label{Sec:Junctions}

We introduce \emph{junctions} as well as their \emph{heights} and \emph{widths}.
We use these dimensions to determine the position of the \emph{siding} inside of a junction, as well as its decomposition into \emph{blocks}.
This allows us to describe the intersection of the branched surface with a (small neighbourhood of a) longitudinal edge.
This will resolve the issues raised in Remarks~\ref{Rem:AlmostProduct} and~\ref{Rem:JunctionSizes}.

\begin{remark}
The exact geometry of junctions is only used to make sure that track-cusps always move forward and never ``overtake'' each other, and to ensure that the branched surfaces do not have ``accidental'' intersections. The material on junctions, sidings, and blocks can safely be ignored on a first reading.
\end{remark}

\begin{definition}
\label{Def:Junction}
Suppose that $\Sigma$ is an upper station.
Suppose that $S$ is a toggle square intersecting $\Sigma$.
The \emph{upper junction} $J = J(\Sigma,S)$ is obtained as follows.
Suppose that $S'$ is the next toggle square that $\Sigma$ intersects, travelling downwards from $S$.
Let $U$ and $U'$ be the crimped shearing regions directly below $S$ and directly above $S'$.
Then $J = J(\Sigma, S)$ is the (closure in the path metric of the) component of $\Sigma - (\bdy^+ U \cup \bdy^- U')$ intersecting $U$ and $U'$.
The \emph{height} of $J$, denoted $h_J$, is the number of crimped shearing regions whose interior it intersects.
Note that the height is finite by \reflem{DualMeetsToggles}.
We define lower junctions, and their heights, similarly.
\end{definition}

For the rest of this section, we fix the following.
Suppose that $J = J(\Sigma, S)$ is the upper junction associated to the station $\Sigma$ and the toggle square $S$.
Let $U$ be the crimped shearing region directly below $S$.
Let $e$ be the crimped longitudinal edge of $U$ that intersects $J$.

\begin{definition}
Let $c$ be the cusp of $e$ closest to $J$.
Let $e'$ be the other crimped longitudinal edge of $U$ meeting $c$.
See \reffig{CrimpedShearingRegion}.
Let $J'$ be the upper junction intersecting $U$ and $e'$.
The \emph{width} of $J$, denoted $w_J$, is defined to be the height of $J'$.
We make a similar definition for lower junctions.
\end{definition}

In the following definition it will be useful to consult \reffig{CrimpedShearingRegionAboveMidAnnulus}.
\begin{definition}
\label{Def:JunctionRadius}
Let $c$ be the cusp of $e$ closest to $J$.
Let $e'$ be the other crimped longitudinal edge of $U$ meeting $c$.
Let $J'$ be the lower junction intersecting $U$ and $e'$.
Let $c'$ be the other cusp of $e$.
Let $e''$ be the other longitudinal crimped edge of $U$ meeting $c'$.
Let $J''$ be the upper junction intersecting $U$ and $e''$.
We take $\epsilon$ to be a small universal constant (smaller than the ``slightly more'' used in \refrem{NiceBigonCoords}\refitm{Station}).
The \emph{radius} of $J$, denoted $r_J$, is defined to be $\epsilon$ divided by the larger of
\[
h' + w' \quad \mbox{and} \quad h'' + w''
\]
where these are the heights and widths of $J'$ and $J''$ respectively.
We make a similar definition for lower junctions.
\end{definition}

We use this to control the cross-sectional radius of stations (\refdef{Station}).
Suppose that $\Sigma$ is a station.
Suppose that $J$ is a junction contained in $\Sigma$.
Let $S'$ be the toggle square meeting the lower boundary of $J$.
Let $U'$ be the crimped shearing region immediately above $S'$.
For any cross-section $H$ meeting $J$ other than those in the lower $1/8^\thsup$ of $U'$, the radius of $J \cap H$ is $r_J$.
The junction $J$ meets two other junctions along $\bdy^- U'$, say $J'$ through $S'$ and $J''$ not.
In the lower $1/16^\thsup$ of $U'$ the radius of $J \cap H$ is the larger of $r_{J'}$ and $r_{J''}$.
In the second $1/16^\thsup$ of $U'$ the radius interpolates linearly between its values at its top and bottom.

%
%

\begin{definition}
\label{Def:Blocks}
Let $e'$ be the helical crimped edge in $\bdy^+ U$ which intersects $J$.
Let $b$ be the branch of $B^\calV \cap \bdy^+ U$ (in parted position) that intersects $e'$.
Let $p$ be the point of intersection between $b$ and $e'$.
Let $h_J$ and $w_J$ be the height and width of $J$ respectively.
We set the $y$-coordinate of $p$ to be $r_J/(2h_J + w_J)$, as measured in bigon coordinates from $e$.
See \reffig{Blocks}.

Outside of a small neighbourhood of the longitudinal crimped edges, the helical crimped edge $e'$ is a line segment in $\bdy^+U$.
Extend this line (in bigon coordinates) until it intersects $e$ at a point $q$.
We take $p'$ to be the point of intersection between $x(p, U)$ and $y(q, U)$.
Twice the distance between $p$ and $p'$, along $x(p, U)$, is the \emph{block length}, denoted $b_J$.
The point $p'$ is the \emph{last block boundary}.
The \emph{siding} in $J$ is the component of $\big(x(p', U) \cap J\big) - p'$ containing $p$.
The other \emph{block boundaries} are the points of the siding which are an integer multiple of the block length away from $p'$.
The segments of the siding between block boundaries are called \emph{blocks}.
See \reffig{Blocks}.
\end{definition}

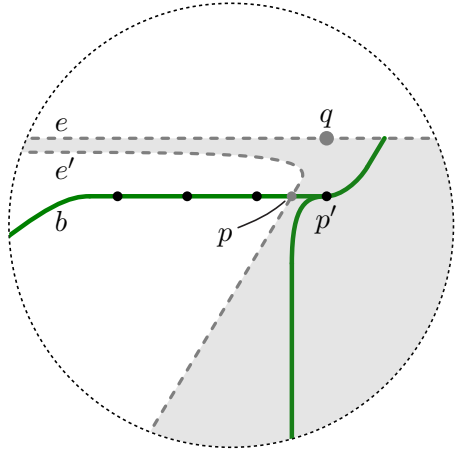
\begin{figure}[htbp]
\labellist
\small\hair 2pt
\pinlabel $e$ [l] at 35 207
\pinlabel $e'$ [l] at 35 180
\pinlabel $b$ [l] at 35 147
\pinlabel $q$ [b] at 207 202
\pinlabel $p'$ [t] at 207 160
\pinlabel $p$ [tr] at 151 145
\endlabellist
\includegraphics[width = 0.50\textwidth]{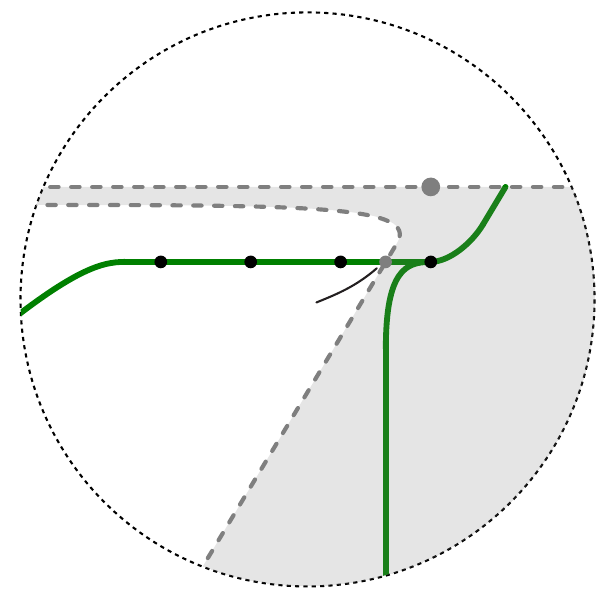}
\caption{The block boundaries are marked by black dots, with the last block boundary inside of the toggle square $S$.  The labels are used in \refdef{Blocks}.
}
\label{Fig:Blocks}
\end{figure}

\begin{remark}
Note that each crimped edge appears in the various crimped shearing regions as a helical crimped edge exactly twice, and otherwise as a longitudinal crimped edge.
When it appears as a helical crimped edge, it is once on an upper boundary of a crimped shearing region (see \reffig{MagnifyPartedTopToggle}), and once on a lower boundary (see \reffig{MagnifyPartedBottomToggle}).
Thus \refdef{Blocks} is well-defined.
\end{remark}

With conventions (on bigon coordinates) as given in \refrem{NiceBigonCoords} the block length for $J$ is
\[
b_J = \frac{2}{\sqrt{3}} \cdot \frac{r_J}{2h_J + w_J}
\]

Let $(U_i)_{i = 1}^N$ be the crimped shearing regions meeting the interior of $J$.
Thus $U_1 = U'$ and $U_N = U$.
Let $H$ be a cross-section of any one of the $U_i$ other than in the lowest $1/8^\thsup$ of $U_1$.
We project the siding and its blocks downward (using bigon coordinates) to obtain a siding, and its blocks, in $H$.
The projection from $U_{i+1}$ down to $U_i$ has the effect of shearing the blocks back by a single block length.
The distance between the siding and the longitudinal edge does not change.
By our choice of radius the siding in $J \cap H$ contains $2h_J + w_J$ blocks.
The train track in $J \cap H$, in parted position, is required to contain all of these blocks.
(If $H$ lies in the lowest $1/16^\thsup$ of $U_1$ then the siding of $J\cap H$ is inherited from the junction immediately below $J \cap \bdy^- U_1$ not meeting the toggle square.
In the second $1/16^\thsup$ we interpolate.)

Recall that the toggle square $S$ lies in $\bdy^+ U$.
Let $H_s$, for $s \in [1/2, 1]$ be the cross-section of $\Theta^U$ at height $s$ in bigon coordinates.
Let $k$ be the track-cusp of $B^\calV \cap H_1$ (in parted position) contained in $J$.
Let $K$ be the branch line running through $k$.
Let $k_s = K \cap H_s$.
We require that $k_s$ lies in the last block of the siding in $J \cap H_s$.
Furthermore we require that $k_{1/2}$ lies at the middle of the block while $k_1$ meets the last block boundary.
Finally, we require that the $k_s$ move at constant speed.

\subsection{Parting routes}
\label{Sec:PartingRoutes}

Fix $U$ a crimped shearing region.
Let $H_s$, for $s \in [0, 1/2]$ be the cross-section of $\Theta_U$ at height $s$ in bigon coordinates.
We describe the \emph{parted position} of $B^\calV$ in $\Theta_U$ as an isotopy of the train-tracks $\tau^s = B^\calV \cap H_s$.
As $s$ ranges over $[0, 1/4]$ we will \emph{split} the tracks along \emph{parting routes}.
As $s$ ranges over $[1/4, 1/2]$ we will perform a \emph{graphical isotopy}.
We now turn to the details.

\begin{definition}
\label{Def:PartingRoutes}
Suppose that $k$ is a track-cusp of $B^\calV \cap \bdy^- U$, whose position is determined by \refrem{UpGivesDown}.
Let $K$ be the branch line of $B^\calV$ (in shrunken position) determined by $k$.
Let $\ell$ be the track-cusp of $K$ lying in $\bdy^+ U$, in prepared position.
Let $\ell'$ be the projection of $\ell$ (via bigon coordinates) to $\bdy^- U$.
Then the \emph{parting route} $\alpha(k)$ is the unique route from $k$ to (just before) $\ell'$ carried by the parted track in $\bdy^- U$.

When $\ell'$ lies in a junction (equivalently, if $\ell$ lies in a toggle square), the route $\alpha(k)$ ends in the same block as $\ell'$, but three-quarters of a block length before $\ell'$.
\end{definition}


\begin{figure}[htbp]
\centering
\subfloat[Blue crimped solid torus.]{
\includegraphics[height=3.2cm]{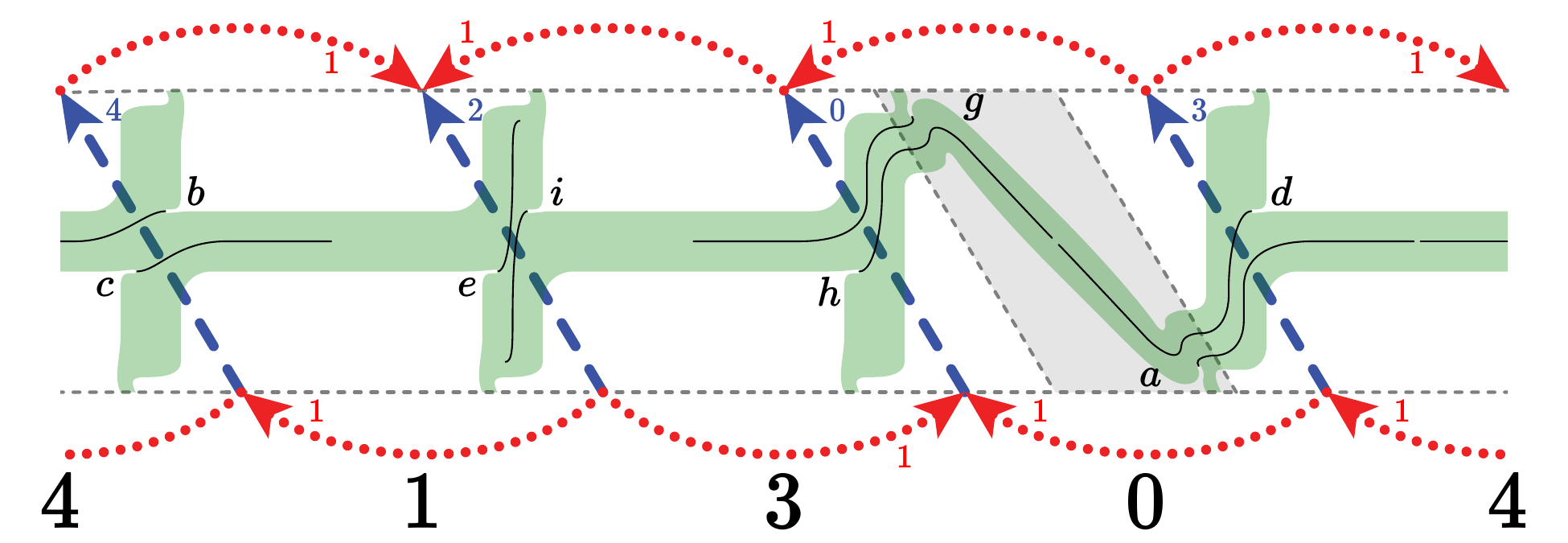}
\label{Fig:PartingRoutesBlue}
}
\subfloat[Red.]{
\includegraphics[height=3.2cm]{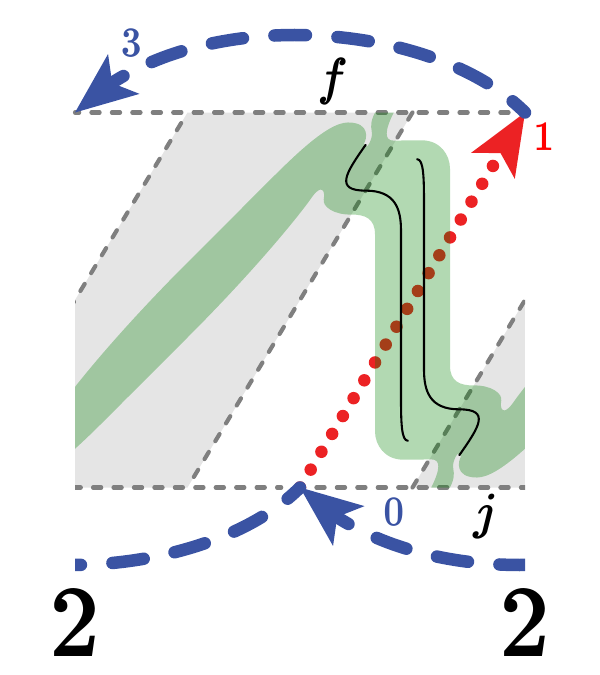}
\label{Fig:PartingRoutesRed}
}

\subfloat[The sixth case.]{
\includegraphics[height=2.7cm]{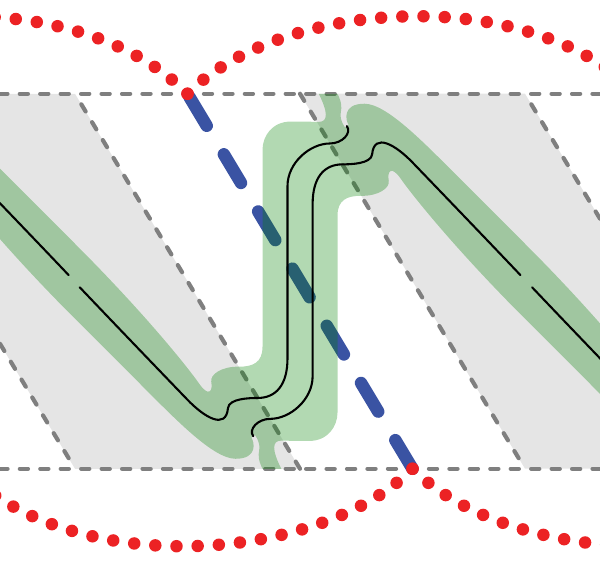}
\label{Fig:PartingRoutesSixth}
}
\caption{Parting routes for the track-cusps in $\bdy^-\Theta_\calV$, where $\calV$ is \usebox{\BigExVeer}.
Here we draw a regular neighbourhood of the train-track in green.
Compare with the last line of \reffig{m115_prepared}.
}
\label{Fig:PartingRoutes}
\end{figure}

By \refrem{AlmostLocal}, the parting isotopy in $\Theta^\calV$ is local (outside of stations).
Thus there are only finitely many (in fact six) combinatorial possibilities for $\alpha(k)$.
These are all shown in \reffig{PartingRoutes}.
\begin{itemize}
\item
Suppose that $\ell$ lies in a toggle square in $\bdy^+ U$.
\begin{itemize}[label=$\bullet$]
\item
If $k$ also lies in a toggle square (in $\bdy^- U$) then we obtain the examples $f$ and $j$ in \reffig{PartingRoutesRed}.
\item
Otherwise $k$ does not lie in a toggle square and we obtain the examples $e$ and $i$ in \reffig{PartingRoutesBlue}.
\end{itemize}
\item
Suppose that $\ell$ does not lie in a toggle square.
\begin{itemize}[label=$\bullet$]
\item
Suppose that $\ell'$ lies in a toggle square.
\begin{itemize}[label=$\bullet$]
\item
If $k$ lies in a toggle square then we obtain the examples shown in \reffig{PartingRoutesSixth}.
\item
Otherwise $k$ does not lie in a toggle square and we obtain the examples $d$ and $h$ in \reffig{PartingRoutesBlue}.
\end{itemize}
\item
Suppose that $\ell'$ does not lie in a toggle square.
\begin{itemize}[label=$\bullet$]
\item
If $k$ lies in a toggle square then we obtain the examples $a$ and $g$ in \reffig{PartingRoutesBlue}.
\item
Otherwise $k$ does not lie in a toggle square and we obtain the examples $b$ and $c$ in \reffig{PartingRoutesBlue}.
\end{itemize}
\end{itemize}
\end{itemize}

\subsection{Splitting along parting routes}
\label{Sec:SplittingInParting}

Suppose that $U$ is a crimped shearing region.
Let $H_s$ be the family of cross-sections of $\Theta_U$, with $H_0 = \bdy^- \Theta_U$ and $H_{1/2} = \bdy^+ \Theta_U$.
Recall that $B^\calV$ in parted position is already specified in $H_0$ and $H_{1/2}$.
Instead of parametrising the parting isotopy explicitly, we specify parted position in
$B^\calV \cap H_s$ by giving a family of train-tracks.

\begin{figure}[htbp]
\subfloat[Blue crimped solid torus.]{
\centering
\labellist
\small\hair 2pt
\pinlabel {$\bdy^+ \Theta_\calV$} [r] at 10 1327
\pinlabel {\vbox{Graphical\\
isotopy}} at -120 1040
\pinlabel {\vbox{Splitting along\\
parting routes }} at -120 470
\pinlabel {$\bdy^- \Theta_\calV$} [r] at 10 179
\endlabellist
\includegraphics[height=15cm]{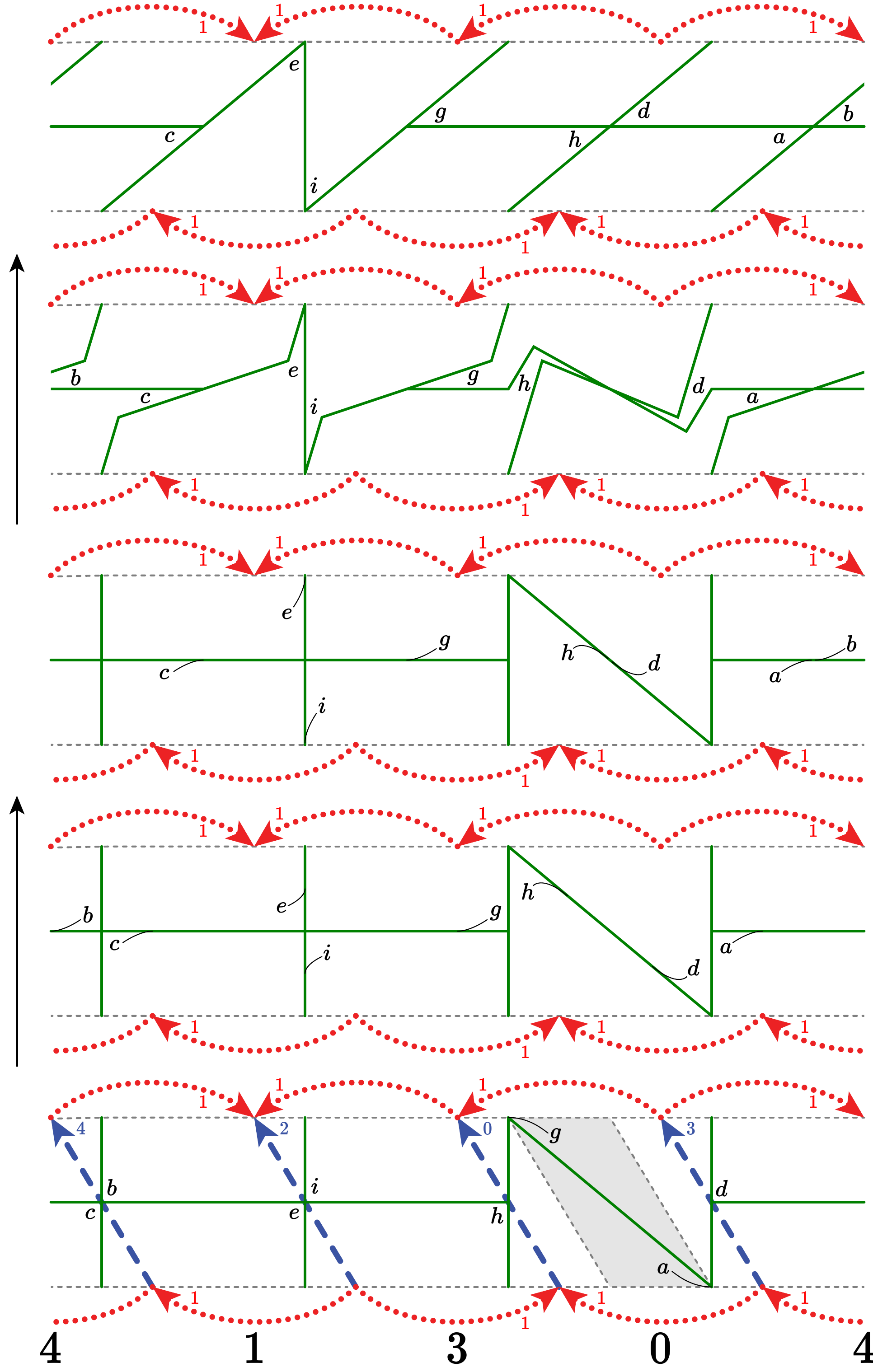}
\label{Fig:m115_parting_in_theta_U_blue}
}
\subfloat[Red.]{
\includegraphics[height=15cm]{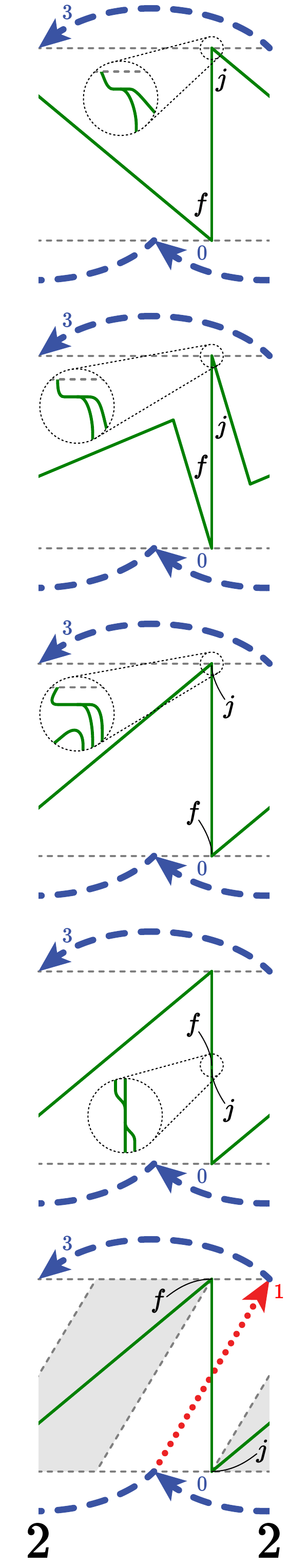}
\label{Fig:m115_parting_in_theta_V_red}
}
\caption{The result $B_1$ of the parting isotopy in $\Theta_\calV$ where $\calV$ is  \usebox{\BigExVeer}.
The five diagrams show (from the bottom moving up) $B_1 \cap C_s$ for $s \in (0,1/8,1/4,3/8,1/2)$.
Thus this figure interpolates the lower three lines of \reffig{m115_prepared}.
The bottom cross-section contains blue helical edges.
}
\label{Fig:m115_parting_in_theta_U}
\end{figure}

As $s$ ranges over $[0, 1/4]$ the intersections of $B^\calV$ (in parted position) with the cross-sections $H_s$ show a movie of a splitting along all of the parting routes.
In detail: if $k$ is a track-cusp in $H_0$ we split $k$ forward in a small neighbourhood of its parting route $\alpha(k)$.
The result in an example is shown in the lower three rows of \reffig{m115_parting_in_theta_U}.
When two track-cusps $k$ and $\ell$ meet, travelling in opposite directions, they split past each other.
(If $U$ is blue and there is (not) a toggle square above, this is a left (right) split.
If $U$ is red the directions swap.)
Outside of stations each track-cusp moves so that
\begin{itemize}
\item
its $x$--coordinate moves at constant speed and
\item
its journey takes all of $[0, 1/4]$.
\end{itemize}
Inside of stations we follow the same rule with one exception;
track-cusps inside of toggle squares split to the boundary of their square and then move as above.
See \reffig{MagnifyPartedTopToggle}.

The construction of the parting routes (\refdef{PartingRoutes}) implies that track-cusps in $H_{1/4}$ either lie very near to the centre line of the cross-section or lie in stations.
See the middle row of \reffig{m115_parting_in_theta_U}.
This describes \emph{splitting along the parting routes}.

\subsection{The graphical isotopy for parted position}
\label{Sec:GraphicalIsotopyInParting}

\begin{remark}
\label{Rem:QuarterPartedGraphical}
Let $\tau^{1/4}$ be the train-track in $H_{1/4}$ given by splitting along parting routes.
Lemmas~\ref{Lem:PartedGraphical} and~\ref{Lem:SplittingPreservesGraphical} ensure that all branches of $\tau^{1/4}$ are lower graphical.
Where $\tau^{1/4}$ meets a longitudinal crimped edge, its tangent vector is parallel to the (projection of the) helical veering edges in $\bdy^-U$.

Let $\tau^{1/2}$ be the train-track in $H_{1/2}$ given by parted position in $\Theta^U$.
\reflem{PartedGraphical} implies that $\tau^{1/2}$ is lower graphical.
Where $\tau^{1/2}$ meets a longitudinal crimped edge, its tangent vector is parallel to the (projection of the) helical veering edges in $\bdy^+U$.
\end{remark}

For $s \in [1/4, 1/2]$, we perform a \emph{lower graphical isotopy} from $\tau^{1/4}$ to $\tau^{1/2}$, as follows.
By \refrem{QuarterPartedGraphical} both train-tracks are lower graphical.
Also they are combinatorially isomorphic and their track-cusps are in (almost) the same places in bigon coordinates.

We now isotope $\tau^{1/4}$ to $\tau^{1/2}$ as follows.
\begin{itemize}
\item
We move each track-cusp slightly forward from its position in $\tau^{1/4}$ to its position in $\tau^{1/2}$.
\item
At the same time we perform a lower graphical isotopy (as in \refdef{GraphicalIsotopy}), moving every branch from its position in $\tau^{1/4}$ to its position in $\tau^{1/2}$.
This isotopy also changes the tangent vector where $\tau^s$ meets a longitudinal crimped edge.
See \reffig{m115_parting_in_theta_V_red}.
\end{itemize}
Points move at constant speed so that their journey takes all of $[1/4, 1/2]$.
This describes the \emph{lower graphical isotopy} from $\tau^{1/4}$ to $\tau^{1/2}$.

An example is given in the upper three rows of \reffig{m115_parting_in_theta_U}.

\begin{lemma}
\label{Lem:PartedIsotopicAndDynamic}
The branched surface $B^\calV$ in parted position is dynamic and is isotopic to shrunken position.
\end{lemma}

\begin{proof}
The intersection with each cross-section is a train-track.
Moreover, by construction the track-cusps always move forwards as we move up through cross-sections.
Therefore the branched surface is dynamic.

In \refsec{PartingUp} we explicitly describe the isotopy between the shrunken branched surface and the parted branched surface in $\Theta^\calV$.
Thus in $\Theta_\calV$ the shrunken branched surface and the constructed branched surface meet
$\bdy^-\Theta_\calV$ and $\bdy^+\Theta_\calV$ with the same combinatorics.
Thus the constructed branched surface is isotopic to the shrunken branched surface.
\end{proof}

We call the result \emph{parted position} for $B^\calV$.
We define parted position for the lower branched surface $B_\calV$ analogously.

\section{Draping}
\label{Sec:Draping}

Here we define the \emph{draping isotopies}.
These are applied to the upper and lower branched surfaces $B^\calV$ and $B_\calV$ starting from parted position and ending in \emph{draped position}.

As usual we concentrate on the upper case.
The name \emph{draped} comes from the final position of the branch lines.
Suppose that $C$ is a branch line.
Suppose that $\calU$ is a component of the monochromatic decomposition, as in \refdef{Monochromatic}.
Suppose that $C' = C \cap \calU$ is a branch interval.
Suppose that the initial point of $C'$ lies in $\bdy^- \Theta_U$:
here $U$ is a crimped shearing region in $\calU$.
In draped position $C'$ runs just above $\bdy^- \Theta_U$, until it encounters the downwards projection (in bigon coordinates) of a toggle square.
At that point $C'$ moves sharply upwards to get just above that toggle square.
The process then repeats until $C'$ exits $\calU$ through some toggle square.
Thus the image of $C'$, under the shearing projection into the midsurface (\refdef{Projection}), is ``draped'' over the images of the toggle squares.
See \reffig{m115_side_draped} for the images of the draped branch intervals in our running example.

We begin with an outline of the construction.
The branched surface $B^\calV$ begins in parted position, as provided in \refsec{Parting}.
Justifying \refrem{SemiLocal}, the draping isotopy is fixed on the union of the toggle squares.
That is, it is fixed on the boundaries of the components of the monochromatic decomposition.

We use $B^\calV_t$ to denote the image of $B^\calV$ at time $t \in [0, 1]$.
In $\Theta^\calV$, where $B^\calV_0$ is almost a product, on each cross-section we will perform (in time)
\begin{itemize}
\item an almost identical splitting along the \emph{draping routes}, and
\item an almost identical lower graphical isotopy.
\end{itemize}
This will determine $B^\calV_t \cap \Theta^\calV$ for all $t$, and thus will fix $B^\calV_t \cap \bdy^\pm \Theta_\calV$.

As with the parting isotopy, the motion in the interior of $\Theta_\calV$ is significantly more complicated.
Suppose that $U$ is a crimped shearing region.
To build $B^\calV_1 \cap \Theta_U$ we start from $B^\calV_1 \cap \bdy^- \Theta_U$.
Let $H_{1/4}$ be the central cross-section of $\Theta_U$. We will perform (in space)
\begin{itemize}
\item
a splitting along \emph{suffix routes} to produce $B^\calV_1 \cap H_{1/4}$ followed by
\item
a lower graphical isotopy to $B^\calV_1 \cap \bdy^+ \Theta_U$.
\end{itemize}
Finally, suppose that $H_s$ is any cross-section of $\Theta_U$.
To build $B^\calV_t \cap H_s$ we start from $B^\calV_0 \cap H_s$.
We will then perform (in time)
\begin{itemize}
\item
a splitting along \emph{prefix routes} to produce $B^\calV_{1/2} \cap H_s$ followed by
\item
a lower graphical isotopy to $B^\calV_1 \cap H_s$.
\end{itemize}
In \reffig{IsotopyDomain} we indicate in the domain of the isotopy which splitting routes are used where and when, and also the supports of the lower graphical isotopies.

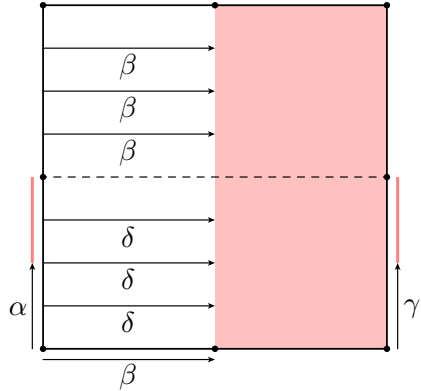
\begin{figure}[htbp]
\labellist
\small\hair 2pt
\pinlabel $\alpha$ [r] at 10 95
\pinlabel $\beta$ [t] at 170 10
\pinlabel $\beta$ [t] at 170 388
\pinlabel $\beta$ [t] at 170 460
\pinlabel $\beta$ [t] at 170 532
\pinlabel $\gamma$ [l] at 625 95
\pinlabel $\delta$ [t] at 170 98
\pinlabel $\delta$ [t] at 170 170
\pinlabel $\delta$ [t] at 170 242
\endlabellist
\includegraphics[height=5cm]{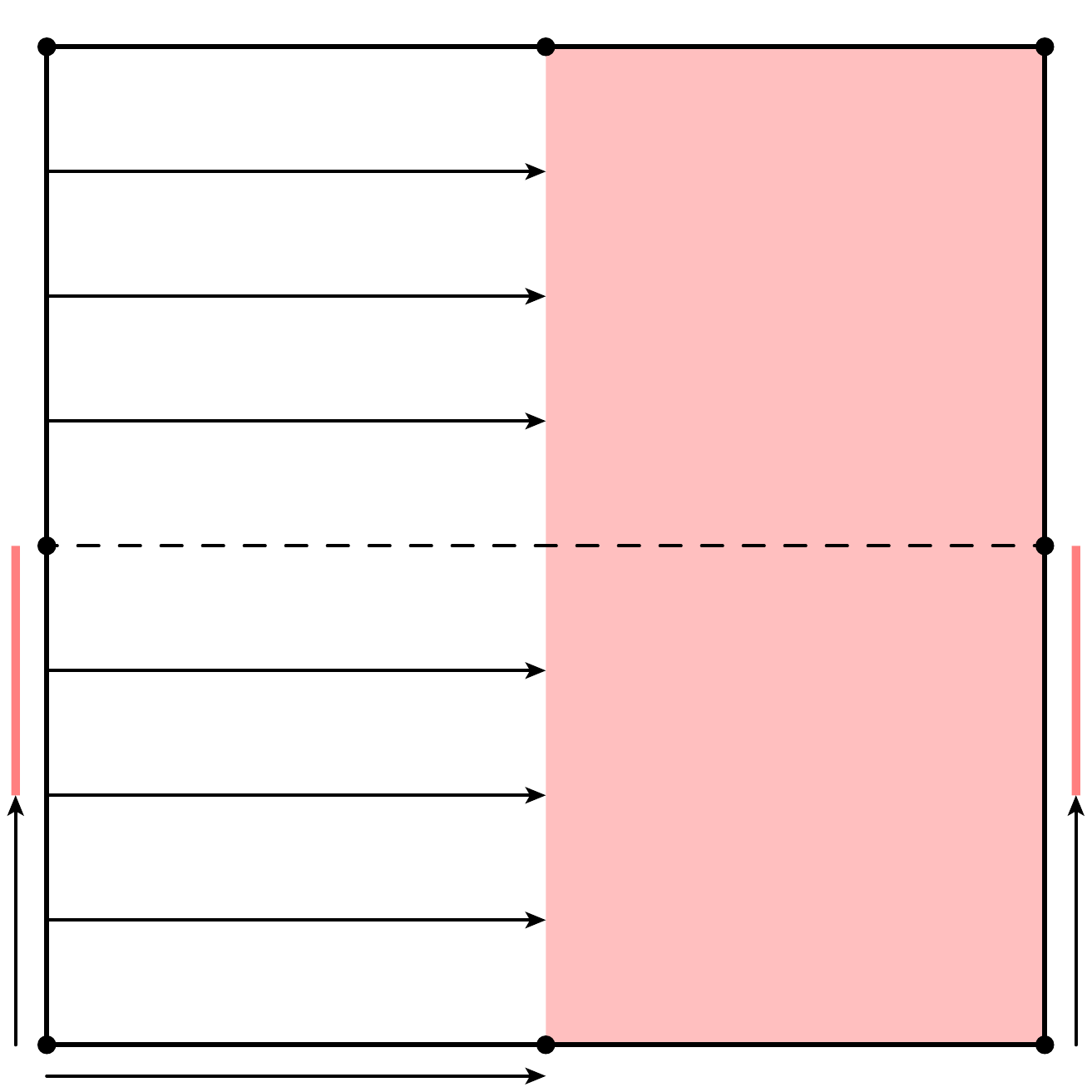}
\caption{The domain of the draping isotopy into $U$.
The time coordinate $t$ runs from left to right in the interval $[0, 1]$;
the space coordinate $s$ runs from bottom to top in the interval $[0, 1]$.
The dashed line separates the preimage of $\Theta^U$ from that of $\Theta_U$.
The labels $\alpha$, $\beta$, $\gamma$, and $\delta$ indicate the parting, draping, suffix, and prefix routes used for the four splittings.
The shaded intervals and region are the supports of the lower graphical isotopies.}
\label{Fig:IsotopyDomain}
\end{figure}

\subsection{Draping routes}

\begin{definition}
\label{Def:DrapingRoutes}
Suppose that $B^\calV$ is in parted position.
Suppose that $K$ is a branch line.
Suppose that $U$ is a crimped shearing region.
Suppose that $K$ intersects $\bdy^-U$ at a point $k_0$.
Starting at $k_0$, we follow $K$ upwards until it hits, for the first time, a toggle square $S = S(k_0)$ not containing $k_0$.
(This exists by \reflem{DualMeetsToggles}.)
Let $N$ be the number of crimped shearing regions meeting $K$ strictly between $k_0$ and $S$.
Let $k_N$ be the intersection of $K$ and $S$.
Note that $N$ is bounded from above by the height of the junction containing $k_N$.

Denote these crimped shearing regions as $U_i$ with increasing index as we ascend $K$,
starting with $U = U_1$.
Let $k_i$ be the intersection of $K$ with $\bdy^+U_i$.
Thus $S$ lies in $\bdy^+U_N$.
We parametrise the subinterval $[k_0, k_N]$ of $K$ by $[0,N]$ so that $[i,i+1]$ maps to $[k_i, k_{i+1}]$.


We now define the \emph{draping routes} $\beta(k_r)$ for $r \in [0,N]$ by a (downwards) recursion.
As our base case, we take $\beta(k_N)$ to be the train route with length zero carried by $B^\calV_0 \cap \bdy^+ U_N$ which starts and ends at $k_N$.
Since $\beta(k_N)$ has length zero, it consists only of a tangent vector based at $k_N$ and pointing away from the track-cusp.
See the station shown on the left-hand side of \reffig{MagnifyPartedTopToggle}.

Fix $r$ and $s$ in $[0,N]$ with $r < s$.
Suppose that $[k_r, k_s]$ is contained inside of $U_i$, one of the crimped shearing regions.
Thus $i \leq r < s \leq i+1$.
Let $H_r$ ($H_s$) be the cross-section of $U_i$ through $k_r$ ($k_s$).
Suppose now that we are given the draping route $\beta(k_s)$, carried by $B^\calV \cap H_s$.
The recursive hypothesis tells us that $\beta(k_s)$ runs from $k_s$ to a point inside of a junction $J$.
We now form a train route $\eta(k_r)$, carried by $B^\calV \cap H_r$, as follows.
\begin{itemize}
\item The start of $\eta(k_r)$ is $k_r$.
\item The $x$--coordinate of the end of $\eta(k_r)$ is $(s - r) b_J$
behind the end of $\beta(k_s)$.
(Here $b_J$ is the block length of $J$, as given in \refdef{Blocks}.)
Thus the end of $\eta(k_r)$ lies in $J$ and is in the same block as the end of $\beta(k_s)$.
\end{itemize}

To define $\beta(k_r)$ (and so complete the recursion) we consider cases.
\begin{enumerate}
\item Suppose that $k_r$ is in the interior of $U_i$.
In this case we take $\beta(k_r) = \eta(k_r)$.

\item Suppose instead that $k_r$ lies in the lower boundary of $U_i$.
Suppose, in addition, that $k_r$ lies in a toggle square.
In this case $r = 0$ and we take $\beta(k_r)$ to have length zero.

\item Suppose instead that $k_r$ does not lie in a toggle square.
\label{Itm:LowerNotInToggleSquare}
\begin{enumerate}
\item
\label{Itm:DoesNotMeetToggleSquare}
Suppose, in addition, that $\eta(k_r)$ does not meet any toggle squares.
In this case we again take $\beta(k_r) = \eta(k_r)$.
Note that the end point of $\beta(k_r)$ lies inside of the same junction (and same block) as the end point of $\beta(k_s)$.

\item
\label{Itm:DoesMeetToggleSquare}
Suppose instead that $\eta(k_r)$ does meet a toggle square;
in this case we \emph{truncate}.
We delete from $\eta(k_r)$ all intersections with toggle squares and keep only the segment meeting $k_r$, to obtain $\eta'(k_r)$.
Note that $\eta'(k_r)$ ends in a junction $J'$, on a helical crimped edge.
Let $n$ be the number of helical veering edges crossed by $\eta'(k_r)$.
We finally obtain $\beta(k_r)$ by removing a segment of length $(n - 1/2)b_{J'}$ from the end of $\eta'(k_r)$.  \qedhere
\end{enumerate}
\end{enumerate}
\end{definition}

For examples see Figures~\ref{Fig:DrapingRoutes} and~\ref{Fig:TruncateExample}.

\begin{figure}[htbp]
\centering
\labellist
\small\hair 2pt
\pinlabel {\vbox{Graphical\\
isotopy}} at -110 1040
\pinlabel {\vbox{Splitting along\\
parting routes }} at -110 470
\pinlabel {$\bdy^+ \Theta_U$} [r] at 10 1337
\pinlabel {$\bdy^- \Theta_U$} [r] at 10 179
\endlabellist
\includegraphics[width=0.9\textwidth]{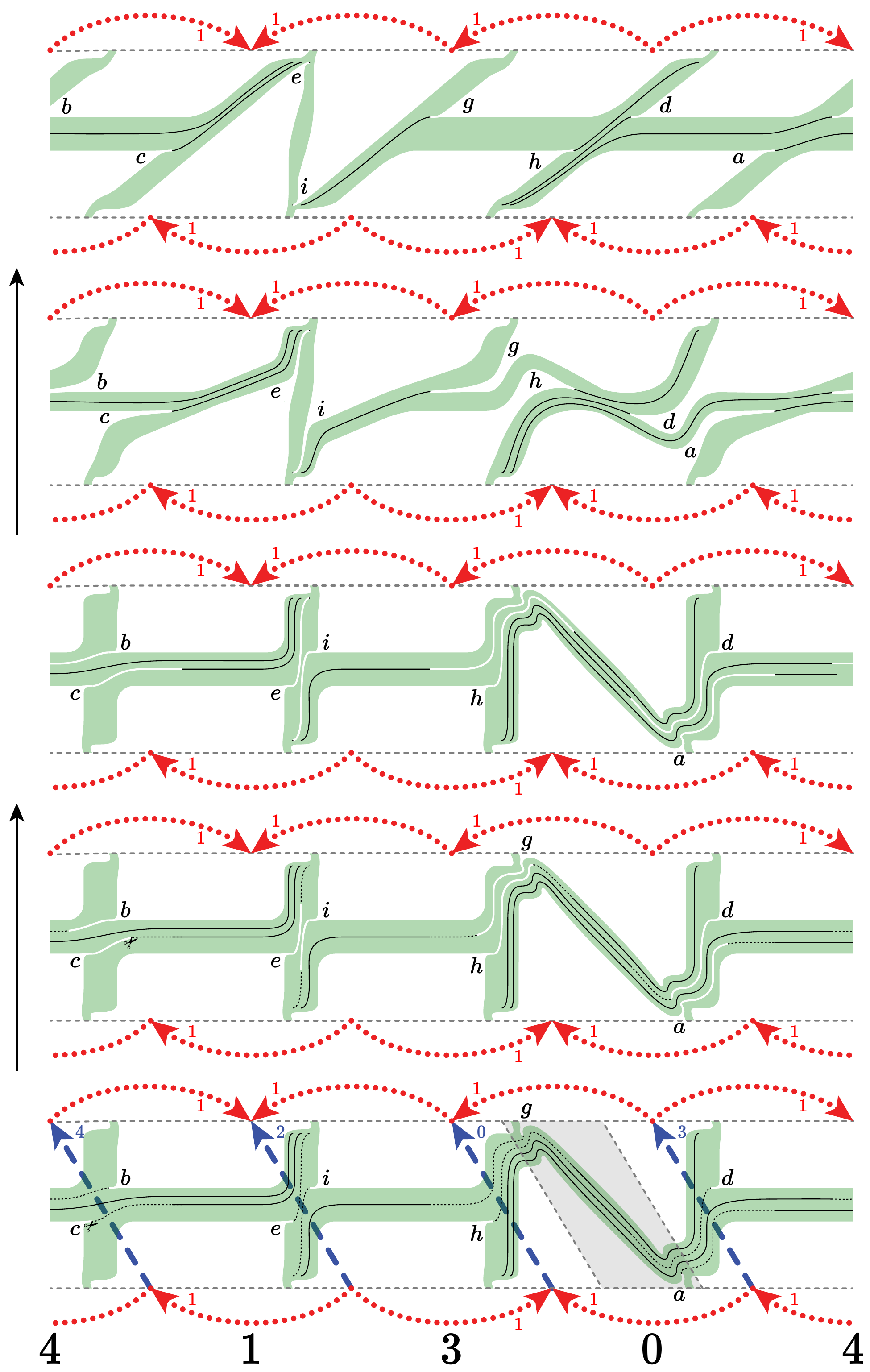}
\caption{Draping routes $\beta(a)$ through $\beta(i)$ for the track-cusps in cross-sections of $\Theta_U$.
Here $U$ is the blue crimped solid torus of \usebox{\BigExVeer};
compare with \reffig{m115_parting_in_theta_U_blue}.
The lowest line shows the $\eta$ routes in $\bdy^-U$.
Note that this figure cannot be drawn to scale;
because we have thickened the tracks we cannot represent the exact positions of ends of routes.
Likewise, we do not see the difference between the $\eta$ and the $\eta^*$ routes (\refdef{PartingInDraping}).
This, and \reflem{PartingInDraping}, allows us to draw the parting $\alpha$ routes (dotted) as subsets of the $\eta$ routes.
}
\label{Fig:DrapingRoutes}
\end{figure}

\begin{figure}[htbp]
\centering
\subfloat[Before truncation.]{
\includegraphics[width=0.3\textwidth]{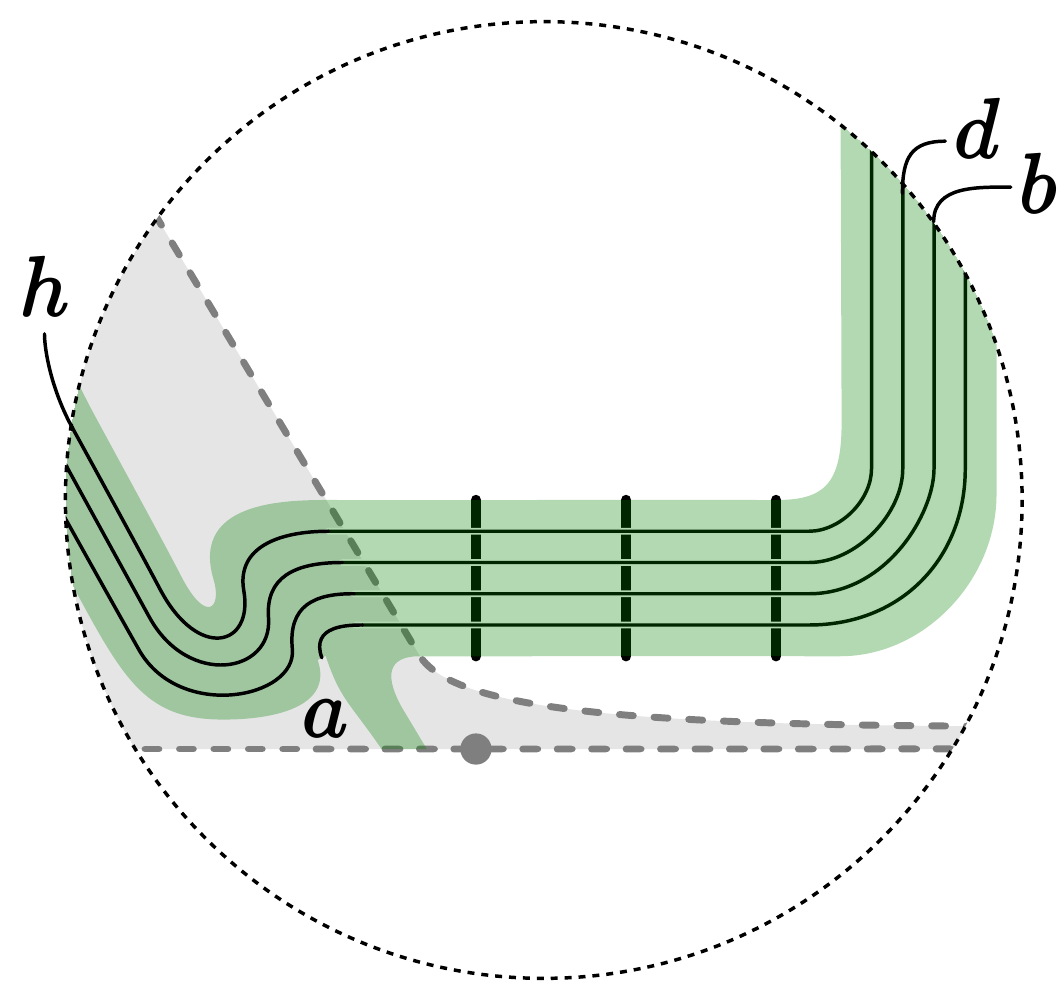}
\label{Fig:BeforeTruncation}
}
\thinspace
\subfloat[First truncation.]{
\includegraphics[width=0.3\textwidth]{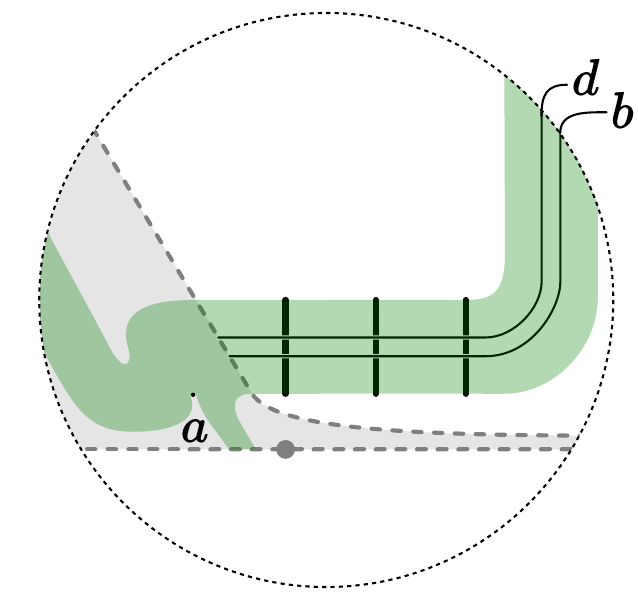}
\label{Fig:MidTruncation}
}
\thinspace
\subfloat[Second truncation.]{
\includegraphics[width=0.3\textwidth]{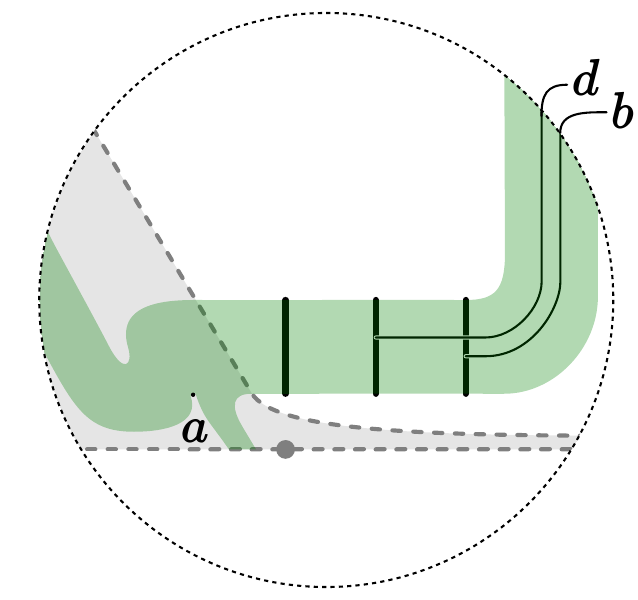}
\label{Fig:AfterTruncation}
}
\caption{
A magnification of the lower right of the toggle square in \reffig{DrapingRoutes}.
The routes $\eta$ (before truncation) are on the left.
The routes $\eta'$ (after first truncation) are in the centre.
(The $\eta'$ routes appear to end at points with different $x$--coordinate;
this is an artefact of drawing the train-track with positive width.)
The draping routes $\beta$ (after second truncation) are on the right.
}
\label{Fig:TruncateExample}
\end{figure}

\begin{lemma}
The draping routes given in \refdef{DrapingRoutes} are well-defined.
\end{lemma}

\begin{proof}
We must check in case \refitm{LowerNotInToggleSquare} of \refdef{DrapingRoutes} that the end points of the draping routes lie inside of junctions.
Moreover, we must also check that there is at least one more block free behind the endpoint.
This is necessary because routes in the lower boundary of a crimped shearing region (but not in a toggle square) also appear in the upper boundary of the crimped shearing region immediately below, sheared back by one block.

We first fix notation.
Suppose that $U$ is a crimped shearing region.
Suppose that $H = \bdy^-U$.
Suppose that $k$ is a track-cusp of $\tau^H$.
Consider the draping route $\beta(k)$.

Now suppose that we are in case \refitm{DoesNotMeetToggleSquare}.
Suppose that the endpoint of $\beta(k)$ lies in junction $J$ in the $m^\thsup$ block (as measured from the last block boundary as in \refdef{Blocks}).
By induction, $m$ is at most $w_J + 2N$ where $N$ is the number of crimped shearing regions between $H$ and the toggle square meeting the upper boundary of $J$.
Thus $N \leq h_J$.
Since the number of blocks in $J$ is $w_J + 2h_J$ we are done.

Instead, suppose that we are in case \refitm{DoesMeetToggleSquare}.
When we truncate $\eta(k)$, the end of $\eta'(k)$ lands in a junction $J'$, say.
Suppose that to obtain $\beta(k)$ we truncate a segment of length $(n - 1/2)b_{J'}$ from the end of $\eta'(k)$.
The definition of $n$ ensures that there are no toggle squares in $H$ between the beginning and the end of $\beta(k)$.
Let $K$ be the branch line containing $k$.
Set $k_0 = k$ and in general let $k_i$ be the intersection between $K$ and the lower boundary of the $i^\thsup$ crimped shearing region above $k$ as we travel up $K$.
Let $\Sigma'$ be the station containing $J'$.
The recursive construction of $\beta(k_i)$ ensures that for $i \leq n$ there are no toggle squares intersecting $\eta(k_i)$ between $k_i$ and $\Sigma'$.
Thus $n \leq w_{J'}$.
Since the number of blocks in $J'$ is $w_{J'} + 2h_{J'}$ we are done.
See \reffig{dwell_time_block_number}.
\end{proof}

\begin{figure}[htbp]
\labellist
\small\hair 2pt
\pinlabel {$H$} [r] at 0 220
\pinlabel {$k_0$} [r] at 170 220
\pinlabel {$k_1$} [r] at 242 292
\pinlabel {$k_2$} [r] at 314 364
\pinlabel {$k_3$} [r] at 386 436
\pinlabel {$J'$} [l] at 764 148
\pinlabel {$J''$} [l] at 620 148
\pinlabel {$\Sigma'$} [b] at 333 577
\endlabellist
\includegraphics[width = 0.9\textwidth]{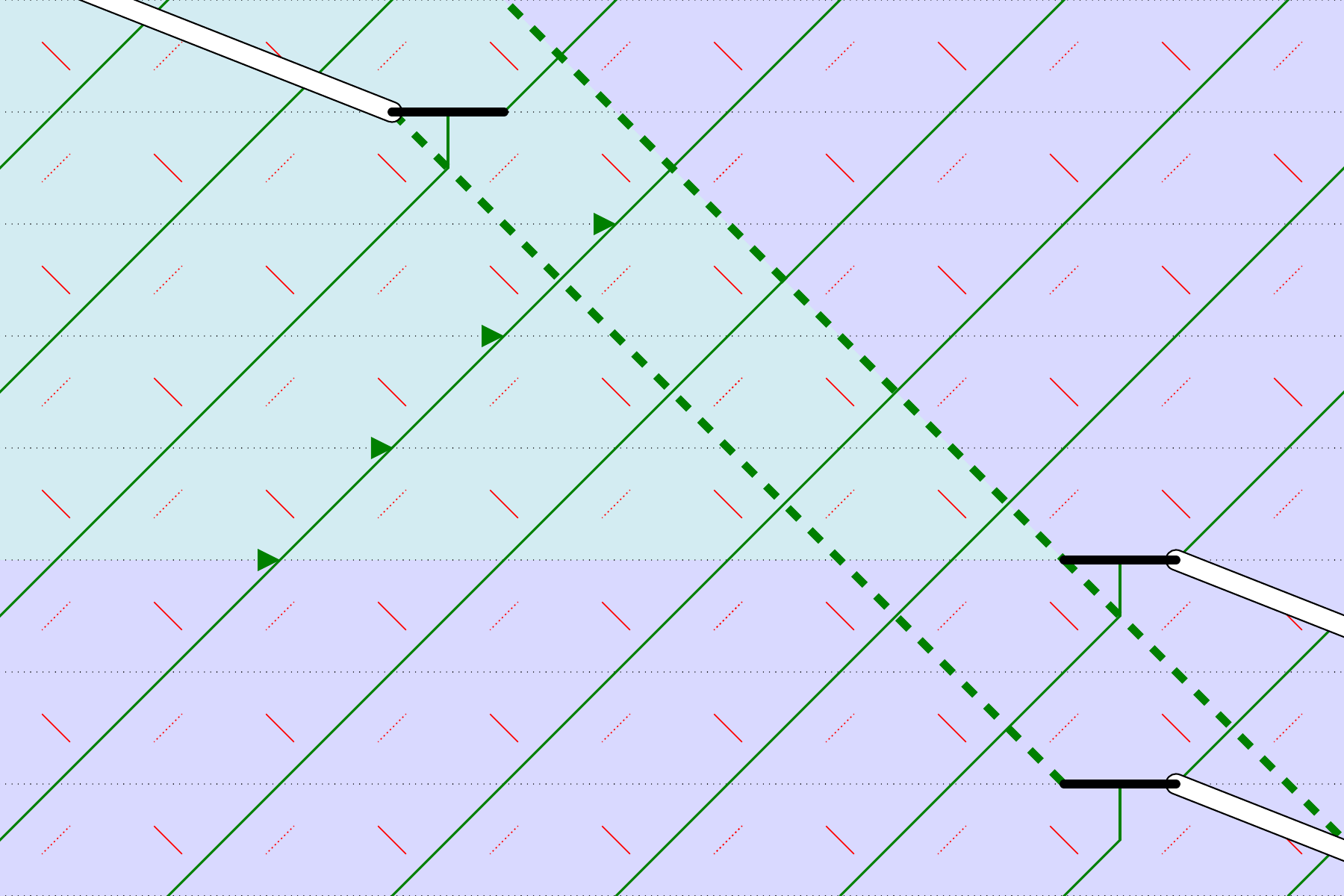}
\caption{
Straightened branch lines of $B^\calV$, projected to the mid-surface $\calS$.
The station $\Sigma'$ and the junctions $J'$ and $J''$ are marked with dashed lines.
In this example the height of $J''$ (and thus the width of $J'$) is six.
The cross-section $H$ and the cross-sections $k_i$ are also indicated.
The draping route $\beta(k_0)$ lies in $H$ and runs from $k_0$ to a point in the upper boundary of $J'$.
Note that $w_{J'}$ only gives an upper bound for the number of track-cusps we must accommodate in $J' \cap H$. The actual count is the number of branch lines in the green wedge that separate $J'$ from any other toggle square in the wedge.
}
\label{Fig:dwell_time_block_number}
\end{figure}
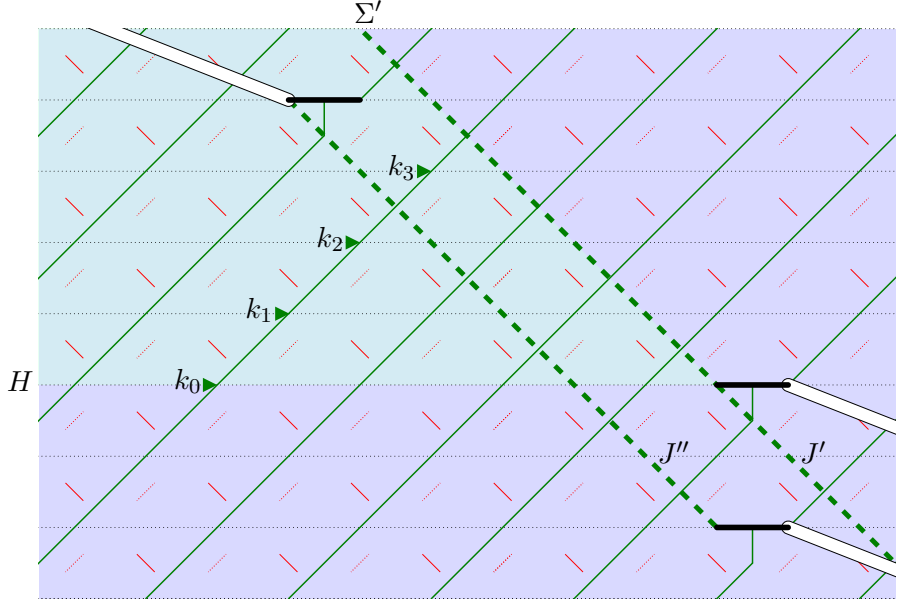

\begin{definition}
\label{Def:PartingInDraping}
Suppose that $k$ is a track-cusp in $B^\calV \cap \bdy^- \Theta_\calV$.
Let $\eta^*(k)$ be the route $\eta(k)$ (as in \refdef{DrapingRoutes}) extended by $1/4$ of a block.
\end{definition}

\begin{lemma}
\label{Lem:PartingInDraping}
The parting route $\alpha(k)$ is a subset of $\eta^*(k)$.
\end{lemma}

\begin{proof}
Suppose that $\alpha(k)$ and $\eta^*(k)$ are contained in $\bdy^- U$ where $U$ is a crimped shearing region.
Both $\alpha(k)$ and $\eta(k)$ are created by folding routes in $\bdy^+ U$ downwards.
The route creating $\alpha(k)$ has length zero and is the initial point of the route creating $\eta(k)$.
The extra $1/4$ of a block is needed because we truncated a $3/4$ block to form $\alpha(k)$ but a full block to form $\eta(k)$.
\end{proof}

See Figures~\ref{Fig:IsotopyDomain} and~\ref{Fig:DrapingRoutes}.

\begin{lemma}
\label{Lem:DrapingRoutesNoCross}
Suppose that $H$ is any cross-section.
Suppose that $k$ and $\ell$ are track-cusps of $B^\calV \cap H$.
Then $\beta(k)$ and $\beta(\ell)$ do not cross
(that is, there is a small motion of $\beta(k)$ making the two routes disjoint).

Furthermore, if $H$ is the lower boundary of a crimped shearing region $U$ then $\eta^*(k)$ and $\beta(\ell)$ do not cross.
\end{lemma}

\begin{proof}
We use the notation of \refdef{DrapingRoutes}.
Let $[k_0, k_1]$ and $[\ell_0, \ell_1]$ be the resulting branch intervals in the branch lines $K$ and $L$ containing $k$ and $\ell$ respectively.
Let $k_r$ and $\ell_s$ be the last points in these branch intervals for which there is a horizontal cross-section $H'$ containing both.
We deduce that $H'$ is the upper boundary of some crimped shearing region $U$.

\begin{claim}
$\beta(k_r)$ and $\beta(\ell_s)$ are disjoint, thus they do not cross.
\end{claim}

\begin{proof}
If $r = 1$ then $\beta(k_r)$ is contained in a station.
In this case, if $\beta(\ell_s)$ meets $\beta(k_r)$ then (due to the truncation step of the construction) we find that
$\beta(k_r) = \beta(\ell_s)$.
Thus $k_0 = \ell_0$ and we are done.

A similar proof deals with the case that $s = 1$.
We may now suppose that $r < 1$ and $s < 1$.
Let $T'$ be the union of the toggle squares of $H'$.
Define $H'' = H' - T'$.
Note that each component of $H''$ also appears as a subsurface of the lower boundary of some crimped shearing region.
Since $k_r$ and $\ell_s$ are the last points of $[k_0, k_1]$ and $[\ell_0, \ell_1]$ in a common cross-section, we find that $k_r$ and $\ell_s$ are necessarily in different components of $H''$.
By construction $\beta(k_r)$ and $\beta(\ell_s)$ are also contained in these components, so are disjoint.

Now suppose that $\eta^*(k)$ and $\beta(\ell)$ lie in the lower boundary of a crimped shearing region.
By continuity $\eta(k)$ and $\eta(\ell)$ do not cross.
Thus $\eta^*(k)$ and $\beta(\ell)$ do not cross.
\end{proof}

We now reparametrise $[k_0, k_r]$ and $[\ell_0, \ell_s]$ by the unit interval and re-choose our notation so that,
for all $q \in [0, 1]$, the track-cusps $k_q$ and $\ell_q$ lie in the same cross-section $H_q$.
By the claim, when $q = 1$ the routes $\beta(k_q)$ and $\beta(\ell_q)$ are disjoint in $H_q$.
Let $\tau^q = B^\calV \cap H_q$.
The tracks $\tau^q$ \emph{fold} as $q$ decreases.
Folding preserves the property of not crossing, and we are done.
\end{proof}

\subsection{The upper draping isotopy in \texorpdfstring{$\Theta^\calV$}{Theta sup V}}
\label{Sec:DrapingIsotopyMiddleUpper}

In $\Theta^\calV$, where $B^\calV_0$ is almost a product, we perform (in time) an almost product splitting along the draping routes and then a lower graphical isotopy.

\subsubsection{Splitting along draping routes}
\label{Sec:SplittingInDraping}

Suppose that $H$ is a horizontal cross-section in $\Theta^\calV$ and suppose that $k \in B^\calV \cap H$ is a track-cusp.
Let $D$ be the difference between the $x$--coordinates of the beginning and end of $\beta(k)$.
For $t \in [0, 1/2]$ we split $k$ forward in a small neighbourhood of its draping route $\beta(k)$ at speed $2D$ (as measured in the $x$--coordinate).

Applying \reflem{DrapingRoutesNoCross}, when two track-cusps $k$ and $\ell$ meet, travelling in opposite directions, they split past each other.
As they pass, they split to the left or right as determined by the combinatorics of their draping routes.
Note that each track-cusp moves at the constant speed required for its journey to take all of $[0, 1/2]$.
Thus, by \reflem{DrapingRoutesNoCross}, track-cusps travelling in the same direction never meet.

This describes \emph{splitting along draping routes}.
The motion of the track-cusps, in an example, is shown in the lower three rows of
\reffig{m115_green_splitting_sequence_top_of_Omega_V}.

\subsubsection{The graphical isotopy}
\label{Sec:GraphicalIsotopyInDraping}

All track-cusps of $B^\calV_{1/2} \cap H$ lie in junctions.
Also, by Lemmas~\ref{Lem:PartedGraphical} and~\ref{Lem:SplittingPreservesGraphical}, all branches of $B^\calV_{1/2} \cap H$ are graphical with respect to the lower foliation.

So, for $t \in [1/2, 1]$, we do the following.
Outside of junctions we perform a lower graphical isotopy to make all branches straight in bigon coordinates.
Due to our choice of radius of junctions (\refdef{JunctionRadius}) the straightened branches do not intersect junctions (other than the ones at their endpoints).
Inside of junctions, the sidings remain fixed and other branches move so that their slopes (at the boundary) match the slopes outside of stations.
See the upper three rows of \reffig{m115_green_splitting_sequence_top_of_Omega_V}.

This describes the \emph{lower graphical isotopy}.

\begin{figure}[htbp]
\centering
\labellist
\small\hair 2pt
\pinlabel {\vbox{Graphical\\
                 isotopy}} at 120 1000
\pinlabel {\vbox{Splitting along\\
                 draping routes }} at 120 430
\endlabellist
\includegraphics[height=15cm]{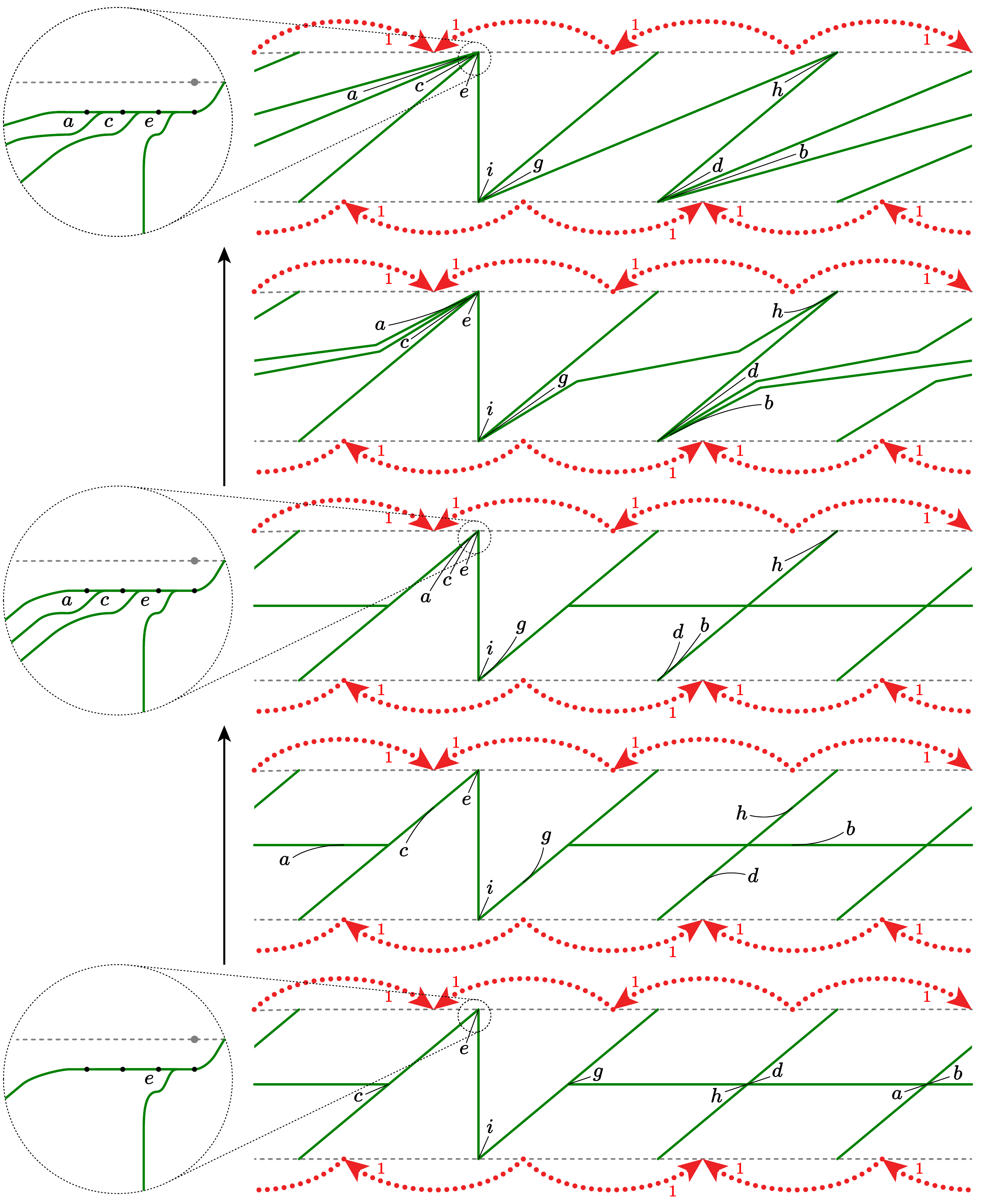}
\caption{The draping isotopy for the cross-section $\bdy^- \Theta^\calV = \bdy^+ \Theta_\calV$ at the middle of $U$.
The five diagrams show (from the bottom moving up) $B^\calV_t \cap H_{1/2}$ for $t \in (0,1/4,1/2,3/4,1)$.
Here $U$ is the blue crimped solid torus for \usebox{\BigExVeer}.
Following \refdef{Blocks} the number of blocks in the junction is eleven;
however in each figure we only draw the blocks needed for the track-cusps in that figure.}
\label{Fig:m115_green_splitting_sequence_top_of_Omega_V}
\end{figure}

\begin{remark}
\label{Rem:TrackCuspNoGo}
Suppose that $H$ is a cross-section in $\Theta^\calV$.
By construction (\refdef{DrapingRoutes}) after splitting along draping routes and performing the lower graphical isotopy all track-cusps of $B^\calV \cap H$ are in blocks in junctions.

As $H$ interpolates from $\bdy^- \Theta^U$ to $\bdy^+ \Theta^U$, the only change in the train-tracks $B^\calV \cap H$ is that track-cusps move forwards in their blocks. (Here we compare positions by projection in bigon coordinates.)
\end{remark}

\begin{remark}
\label{Rem:UpGivesDownAgain}
As in \refrem{UpGivesDown}, the intersection of the image of the upper draping isotopy with cross-sections in $\Theta^\calV$ determines the intersection of the image of the upper draping isotopy with $\bdy^-\Theta_\calV$.
For the result in our running example see \reffig{m115_green_splitting_sequence_bottom_of_Omega_V}.
\end{remark}

\begin{figure}[htbp]
\centering
\labellist
\small\hair 2pt
\pinlabel {\vbox{Graphical\\
isotopy}} at -140 1030
\pinlabel {\vbox{Splitting along\\
draping routes }} at -140 470
\endlabellist
\includegraphics[height=15cm]{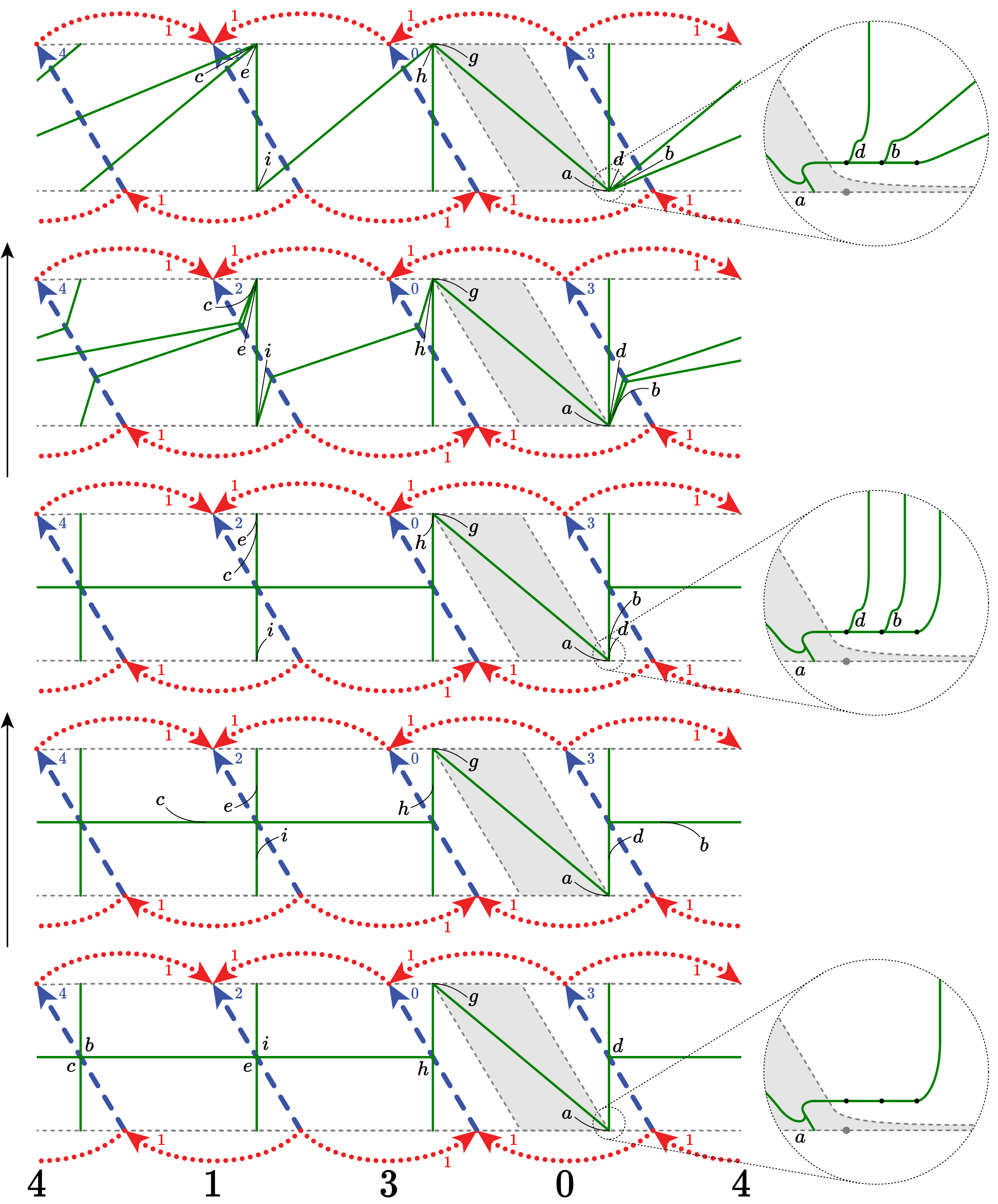}
\caption{The draping isotopy for the cross-section $\bdy^-\Theta_U$.
The five diagrams show (from the bottom moving up) $B^\calV_t \cap H_{0}$ for $t \in (0,1/4,1/2,3/4,1)$.
Here $U$ is the blue crimped solid torus for \usebox{\BigExVeer}.}
\label{Fig:m115_green_splitting_sequence_bottom_of_Omega_V}
\end{figure}

\subsection{The upper draping isotopy in \texorpdfstring{$\Theta_\calV$}{Theta sub V}}
\label{Sec:DrapingIsotopyLower}

Fix $U$, a blue crimped shearing region.
We use $H_s$ to denote the cross-section of $\Theta_U$ at height $s \in [0, 1/2]$.
It remains to describe the intersections $B^\calV_t \cap H_s$.
The intersections $B^\calV_0 \cap H_s$ are given by parted position.
Also, $B^\calV_t \cap H_{1/2}$ and (by \refrem{UpGivesDownAgain}) $B^\calV_t \cap H_0$ are determined by the splitting and lower graphical isotopy given in \refsec{DrapingIsotopyMiddleUpper}.
This gives three sides of the boundary of the isotopy.

\subsubsection{Suffix routes}
\label{Sec:SuffixRoutes}

We now describe the fourth side;
that is, we describe $B^\calV_1 \cap H_s$ for $s \in [0, 1/2]$.
To do this we start from the given $B^\calV_1 \cap H_0$ and perform (in space):
\begin{itemize}
\item
a splitting along \emph{suffix routes} to produce $B^\calV_1 \cap H_{1/4}$ followed by
\item
a lower graphical isotopy to $B^\calV_1 \cap H_{1/2}$.
\end{itemize}

\begin{definition}
Suppose that $k_{0, 0}$ is a track-cusp of $B^\calV_0 \cap H_0$.
We take $\beta(k_{0, 0})$ as given by \refdef{DrapingRoutes}.
We take $\eta^*(k_{0, 0})$ as given by \refdef{PartingInDraping}.
We define $\gamma^* = \eta^*(k_{0, 0}) - \beta(k_{0, 0})$.
This is well-defined because $\beta(k_{0, 0}) \subset \eta(k_{0, 0})$.

Let $k_{1, 0}$ be the track-cusp of $B^\calV_1 \cap H_0$ which is the endpoint of $\beta(k_{0, 0})$.
By \reflem{DrapingRoutesNoCross} none of the draping routes in $B^\calV_0 \cap H_0$ cross $\gamma^*$.
Therefore we may define $\gamma(k_{1,0})$, the \emph{suffix route} for $k_{1,0}$, by taking the image of $\gamma^*$ under the splitting and lower graphical isotopy defined in \refsec{DrapingIsotopyMiddleUpper}.
Note that $\gamma(k_{1,0})$ starts at $k_{1,0}$ and is carried by $B^\calV_1 \cap H_0$.
\end{definition}


As in Sections~\ref{Sec:SplittingInParting} and~\ref{Sec:SplittingInDraping}, we now perform a splitting (in space) along the suffix routes $\gamma$.
As $s$ progresses through $[0, 1/2]$ we split each track-cusp $k_{1,0}$ forward along its suffix route $\gamma(k_{1,0})$.

With that done, we perform (in space) a lower graphical isotopy, analogous to the ones described in Sections~\ref{Sec:GraphicalIsotopyInParting} and~\ref{Sec:GraphicalIsotopyInDraping}.
As $s$ progresses through $[1/4, 1/2]$, outside of junctions we straighten the train-track by a graphical isotopy along the lower foliation.
Since this is an isotopy in space, just as in \refsec{GraphicalIsotopyInParting}, inside of junctions track-cusps move forward to maintain dynamism (at constant speed within their blocks) and branches move via graphical isotopy along the lower foliation as needed.
See \reffig{m115_space_splitting_sequence_blue}.

\begin{figure}[htbp]
\centering
\labellist
\small\hair 2pt
\pinlabel {\vbox{Graphical\\
isotopy}} at 110 1040
\pinlabel {\vbox{Splitting along\\
suffix routes }} at 110 470
\pinlabel {$\bdy^+ \Theta_U$} [r] at -20 1337
\pinlabel {$\bdy^- \Theta_U$} [r] at -20 179
\endlabellist
\includegraphics[width =0.97\textwidth]{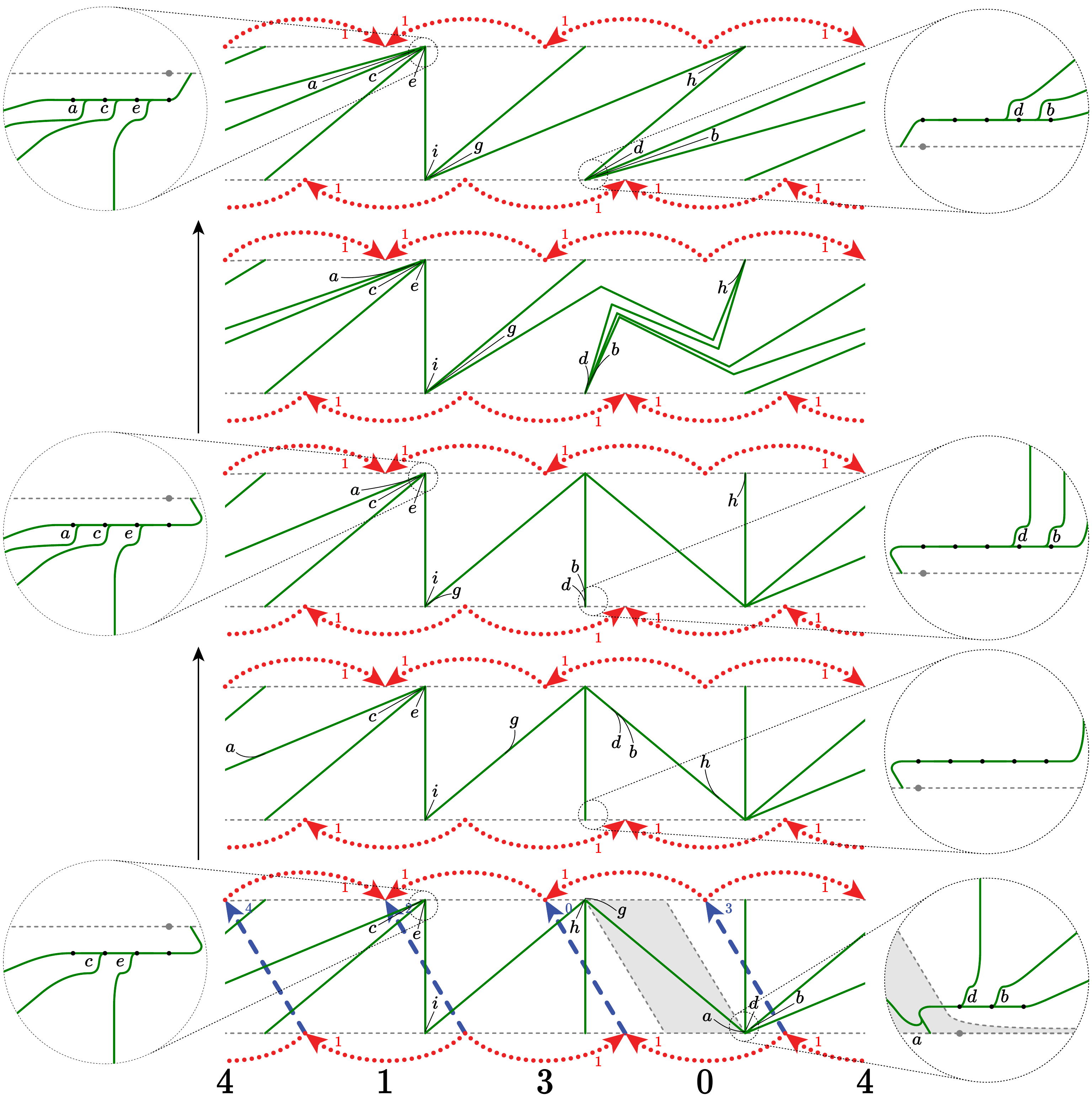}
\caption{The result $B^\calV_1$ of the draping isotopy in $\Theta_U$ where $U$ is the blue crimped solid torus for \usebox{\BigExVeer}.
The five diagrams show (from the bottom moving up) $B^\calV_1 \cap H_s$ for $s \in (0,1/8,1/4,3/8,1/2)$.
}
\label{Fig:m115_space_splitting_sequence_blue}
\end{figure}
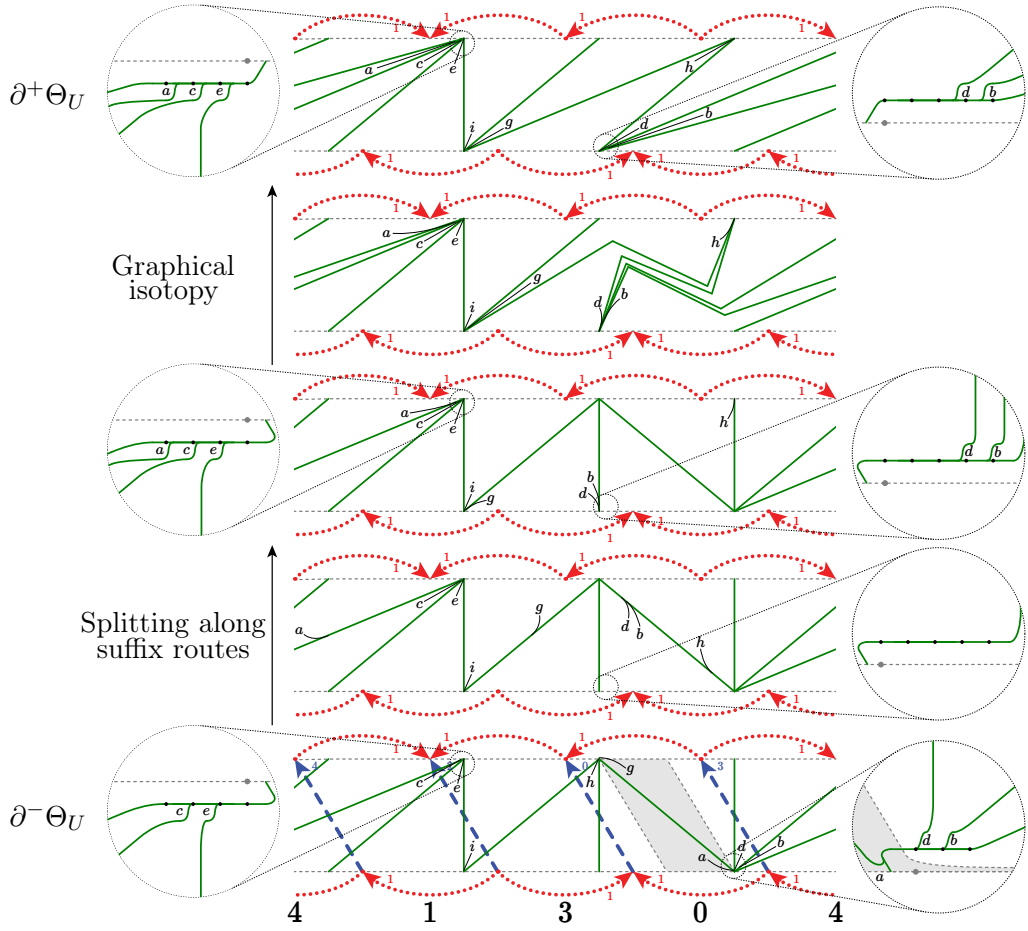

\subsubsection{Prefix routes}

Now that we have constructed the four sides of the isotopy we fill in the interior.
That is, we describe the train-tracks $B^\calV_t \cap H_s$ for $s \in (0,1/2)$ and $t \in (0,1)$.
To build $B^\calV_t \cap H_s$ we start from $B^\calV_0 \cap H_s$ and perform (in time)
\begin{itemize}
\item
a splitting along \emph{prefix routes} to produce $B^\calV_{1/2} \cap H_s$ followed by
\item
a lower graphical isotopy to $B^\calV_1 \cap H_s$.
\end{itemize}

\begin{notation}
Where defined, we take $k_{t,s}$ to be the track-cusp of $B^\calV_t \cap H_s$ corresponding to $k_{0,0}$.
Let $x_{t, s}$ be the $x$--coordinate of $k_{t, s}$ (in bigon coordinates).
\end{notation}

\begin{remark}
Suppose that $s$ lies in $(0, 1/2)$.
Suppose that $J$ is the junction containing the endpoint of $\beta(k_{0,s})$.
Breaking symmetry, suppose that the orientation of $\beta(k_{0, 0})$ points in the direction of increasing $x$--coordinate.
Thus the same holds for the orientation of $\beta(k_{0,s})$.
With the above notation, the endpoints of $\beta(k_{0,s})$ have $x$--coordinates
\[
x_{0,s} \quad \mbox{and} \quad
x_{1,1/2} - (1/2 - s)b_J
\]
where $b_J$ is the block length defined in \refdef{Blocks}.
\end{remark}

\begin{definition}
Suppose that $s$ lies in $(0,1/2)$.
We define the \emph{prefix route} $\delta(k_{0,s})$ to be the prefix of $\beta(k_{0,s})$ which ends at the point with $x$--coordinate equal to $x_{1,s}$.
\end{definition}


\begin{lemma}
With the above notation, for all $s \in (0, 1/2)$ we have
\[
x_{0,s} \leq x_{1,s} \leq x_{1,1/2} - (1/2 - s)b_J
\]
Thus the coordinate $x_{1,s}$ lies in the interval of $x$--coordinates of the draping route $\beta(k_{0,s})$.
Thus the prefix route $\delta(k_{0,s})$ is well-defined.
\end{lemma}

\begin{proof}
Again breaking symmetry, suppose that the orientation of $\beta(k_{0, 0})$ points in the direction of increasing $x$--coordinate.
It follows that the same holds for $\alpha(k_{0,0})$ and $\gamma(k_{1,0})$.
The lengths of the various routes, the motion of track-cusps during graphical isotopies, and \reflem{PartingInDraping} give the inequalities shown in \reftab{Forward}.
(The ordering of $x_{1,0}$ and $x_{0,1/2}$ depends on the combinatorial type of the track-cusp $k_{0,0}$.)
\begin{table}[htbp]
\begin{tabular}{ccc}
\vspace{5pt}
$x_{0, 1/2}$ & $\leq$ & $x_{1, 1/2}$ \\ 
\rlt         &        & \rlt         \\ 
\vspace{5pt}
$x_{0, 1/4}$ & $\leq$ & $x_{1, 1/4}$ \\ 
\rlt         &        & \rlt         \\ 
$x_{0, 0}$   & $\leq$ & $x_{1, 0}$      
\end{tabular}
\caption{}
\label{Tab:Forward}
\end{table}


We now show that $x_{0,s} \leq x_{1,s}$.
Suppose that $s$ lies in $(0, 1/4)$.
Note that $x_{0,s}$ is a barycentric combination of $x_{0, 0}$ and $x_{0, 1/4}$.
Likewise $x_{1,s}$ is a barycentric combination of $x_{1, 0}$ and $x_{1, 1/4}$, with the same coefficients.
Thus the first inequality follows from the middle and lower lines of \reftab{Forward}.
The same argument, applied to the upper and middle lines of \reftab{Forward}, deals with $s$ lying in $(1/4, 1/2)$.

We now show that $x_{1,s} \leq x_{1,1/2} - (1/2 - s)b_J$.
Suppose that $s$ lies in $(1/4, 1/2)$.
The track-cusp $k_{1, s}$ is in its block and moving at speed exactly $b_J$.
Note that $x_{1,1/2} - (1/2 - s)b_J$ also moves at speed exactly $b_J$.
This establishes the inequality (in fact equality) for  $s$ in $(1/4, 1/2)$.
Now suppose that $s$ lies in $(0, 1/4)$.
By construction $\gamma(k_{1,0})$ has length at least $\frac{1}{4}b_J$,
so $x_{1, s}$ moves at speed at least $b_J$.
Thus the inequality also holds for $s$ in $(0, 1/4)$.
\end{proof}

For each fixed $s$ in $(0,1/2)$, we split along the prefix routes (for $t \in [0, 1/2]$).
While splitting, each track-cusp moves from its position in $B^\calV_0 \cap H_s$ to its position in $B^\calV_1 \cap H_s$.
We then perform the lower graphical isotopy (for $t \in [1/2, 1]$); since this is an isotopy in time rather than space, the track-cusps do not move.

This completes the definition of the upper draping isotopy; we call the result \emph{draped position}.
Note that the upper draping isotopy is continuous by construction.

The lower draping isotopy of $B_\calV$ is defined analogously, with the roles of $\Theta^U$ and $\Theta_U$ reversed.
For examples, see Figures~\ref{Fig:m115_split}, \ref{Fig:m115_side_draped}, and~\ref{Fig:FinalPositionFig8Sibling}.

\subsection{Draped position}

With draped position in hand we make a sequence of observations.

\begin{figure}[htbp]
\subfloat[Blue crimped solid torus.]{
\centering
\labellist
\small\hair 2pt
\pinlabel {$\bdy^+ \Theta^\calV$} [r] at 10 835
\pinlabel {$\bdy^- \Theta^\calV = \bdy^+ \Theta_\calV$} [r] at 10 505
\pinlabel {$\bdy^- \Theta_\calV$} [r] at 10 185
\endlabellist
\includegraphics[height=9.5cm]{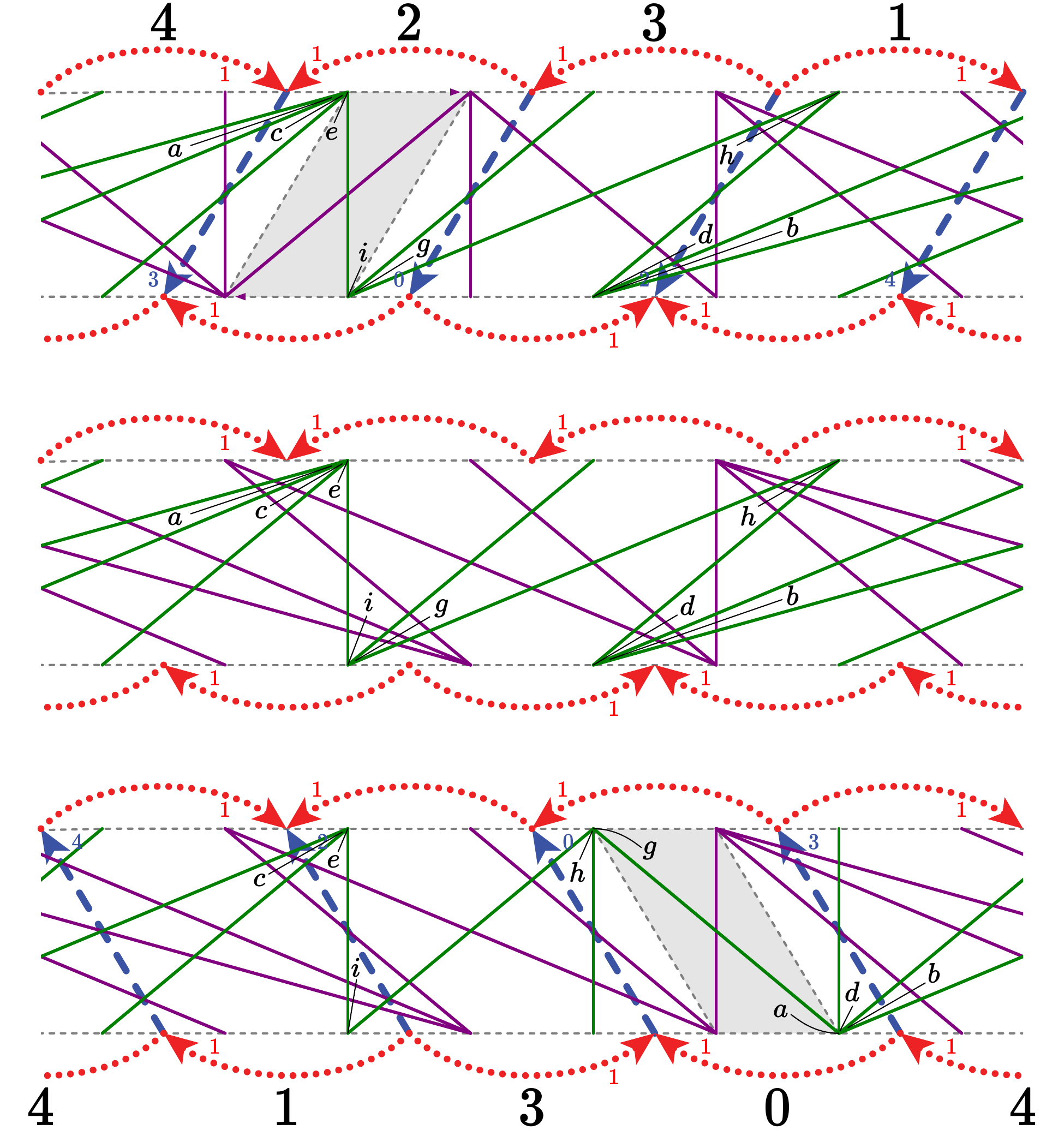}
}
\subfloat[Red.]{
\centering
\includegraphics[height=9.5cm]{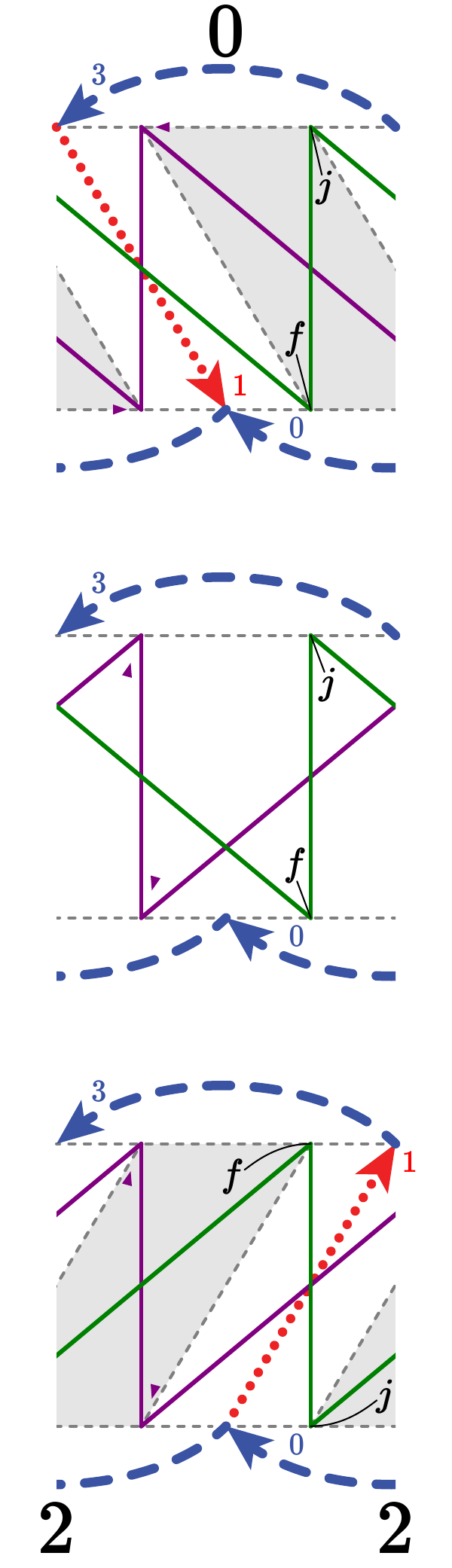}
}
\caption{The intersection of $B^\calV$ (and $B_\calV$), in draped position, with various cross-sections ($s=0, 1/2, 1$) of the crimped shearing decomposition of \usebox{\BigExVeer}.
Compare with \reffig{m115_prepared}.
}
\label{Fig:m115_split}
\end{figure}

\begin{figure}[htbp]
\subfloat[Blue mid-surface.]{
\centering
\includegraphics[height=7cm]{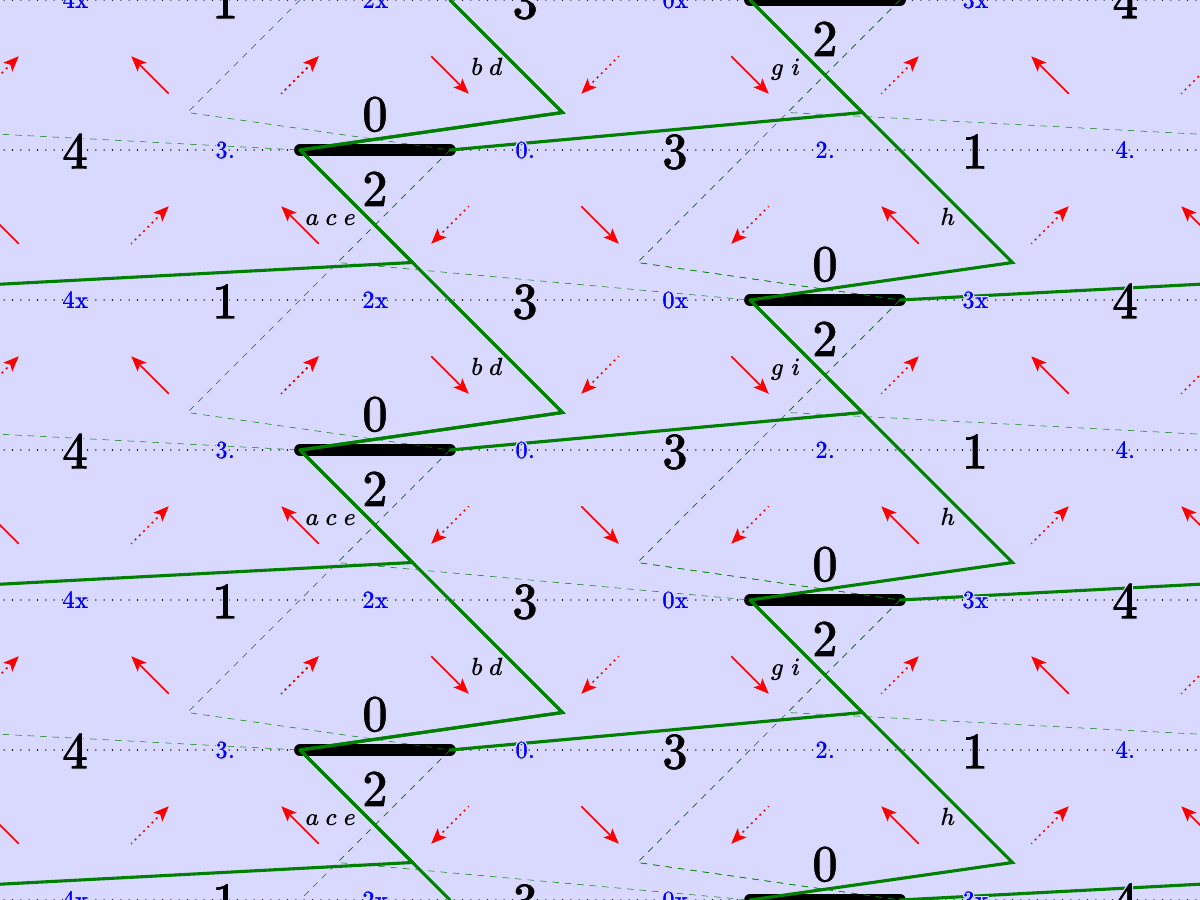}
}
\subfloat[Red.]{
\centering
\includegraphics[height=7cm]{Figures/fLLQccecddehqrwjj_20102_side_red_parted}
}
\caption{The branch lines of $B^\calV$, in draped position, projected to the mid-surfaces. Compare with  \reffig{m115_side_parted}.}
\label{Fig:m115_side_draped}
\end{figure}

\begin{figure}[htbp]
\centering
\subfloat[Blue solid torus.]{
\labellist
\hair 2pt
\pinlabel {$\bdy^+ \Theta^\calV$} [r] at 10 780
\pinlabel {$\bdy^- \Theta^\calV = \bdy^+ \Theta_\calV$} [r] at 10 450
\pinlabel {$\bdy^- \Theta_\calV$} [r] at 10 130
\tiny
\pinlabel \textsc{a} at 120 845
\pinlabel \textsc{c} at 120 780
\pinlabel \textsc{d} at 65 790
\pinlabel \textsc{b} at 65 735

\pinlabel \textsc{a} at 145 490
\pinlabel \textsc{d} at 70 500
\pinlabel \textsc{d} at 220 500
\pinlabel \textsc{b} at 245 425
\pinlabel \textsc{b} at 45 425
\pinlabel \textsc{c} at 110 415
\pinlabel \textsc{c} at 180 415

\pinlabel \textsc{a} at 170 195
\pinlabel \textsc{c} at 170 130
\pinlabel \textsc{d} at 225 140
\pinlabel \textsc{b} at 225 75
\endlabellist
\includegraphics[height=9.5cm]{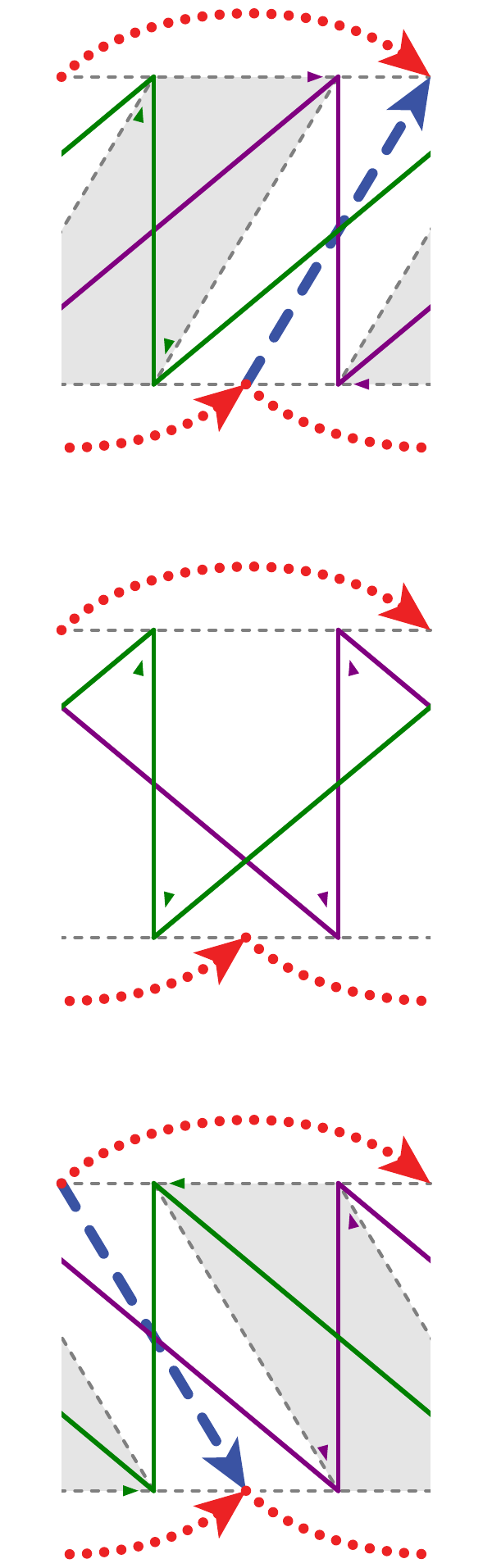}
\label{Fig:m003_final_cross-sections_blue}
}
\quad
\subfloat[Red solid torus.]{
\labellist
\small\hair 2pt
\tiny
\pinlabel \textsc{c} at 170 845
\pinlabel \textsc{b} at 170 780
\pinlabel \textsc{a} at 225 790
\pinlabel \textsc{d} at 225 735

\pinlabel \textsc{c} at 145 490
\pinlabel \textsc{a} at 70 500
\pinlabel \textsc{a} at 220 500
\pinlabel \textsc{d} at 245 425
\pinlabel \textsc{d} at 45 425
\pinlabel \textsc{b} at 110 415
\pinlabel \textsc{b} at 180 415

\pinlabel \textsc{c} at 120 195
\pinlabel \textsc{b} at 120 130
\pinlabel \textsc{a} at 65 140
\pinlabel \textsc{d} at 65 75
\endlabellist
\includegraphics[height=9.5cm]{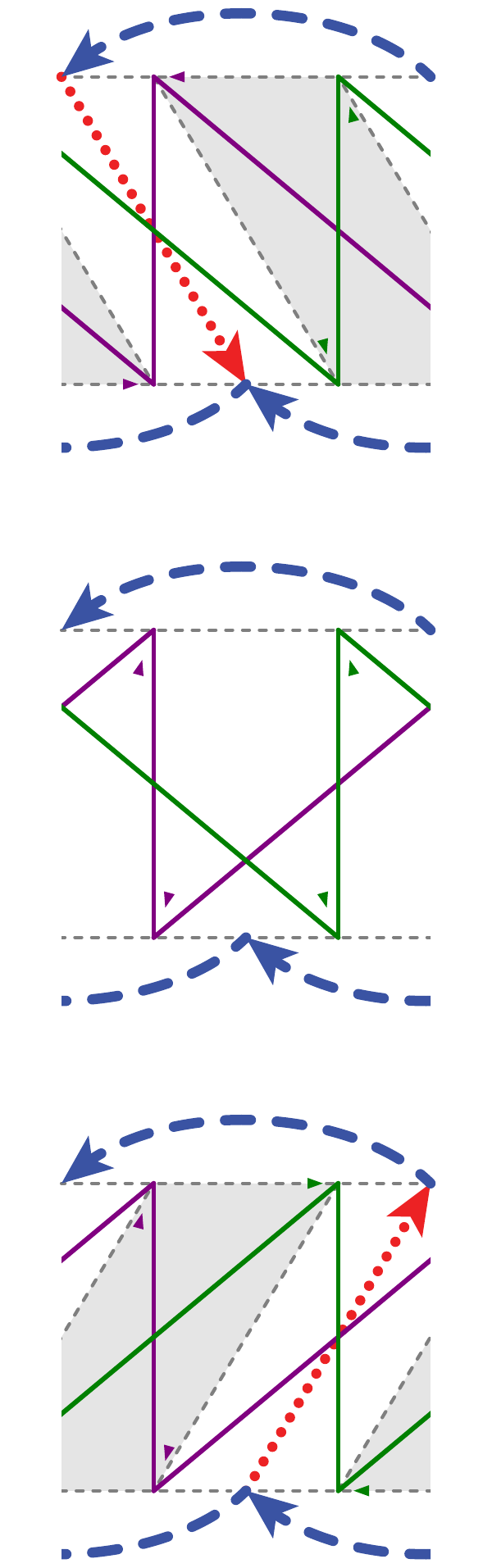}
\label{Fig:m003_final_cross-sections_red}
}
\caption{Draped position for the figure-eight knot sibling with veering triangulation \usebox{\FigEightSisVeer}. The four pinched tetrahedra are labelled \textsc{a} through \textsc{d}. To obtain the pictures for the figure-eight knot complement with veering triangulation \usebox{\FigEightVeer}, alter these figures by requiring that the orientation on every helical edge points upwards. (To relabel the pinched tetrahedra, start with those given at the top of \reffig{m003_final_cross-sections_blue} and propagate outwards.)}
\label{Fig:FinalPositionFig8Sibling}
\end{figure}

\begin{remark}
\label{Rem:KeepYourHeadDown}
Suppose that $H$ is any cross-section.
Suppose that $J$ is a junction which intersects $H$.
Suppose that $B^\calV$ is in draped position.
Let $b$ be a branch of either $\tau^H$ or $\tau_H$.
By our choice of radius $r_J$ (\refdef{JunctionRadius}), the branch $b$ intersects $J$ if and only if (at least) one end of $b$ lies in $J$.
\end{remark}

\begin{lemma}
\label{Lem:DrapedDynamic}
In draped position, the branched surfaces $B^\calV$ and $B_\calV$ are dynamic.
\end{lemma}

\begin{proof}
In draped position the branched surface $B^\calV$ is transverse to the cross-sections of all crimped shearing regions.
Furthermore, we have arranged that track-cusps always move forwards as we move up through cross-sections.
The same argument applies to $B_\calV$.
\end{proof}

Suppose that $H$ is a cross-section in a crimped shearing region $U$.
Suppose that $B^\calV$ and $B_\calV$ are in draped position.
We again define $\tau^H = B^\calV \cap H$ and $\tau_H = B_\calV \cap H$.

\begin{remark}
\label{Rem:BoundaryCrossSectionDraped}
Suppose that $B^\calV$ and $B_\calV$ are in draped position.
Suppose that $U$ is a crimped shearing region.
Suppose that $H$ is either $\bdy^+ U$, the upper boundary of $U$, or $\bdy^- U$, the lower boundary.
\begin{enumerate}
\item
\label{Itm:TrackCusp}
Each track-cusp of $\tau^H$ and $\tau_H$ is in a junction.
\item
\label{Itm:Straight}
Outside of the junctions, the branches of $\tau^H$ and $\tau_H$ are line segments (in bigon coordinates).
\item
\label{Itm:NextToToggleSquare}
Suppose that $e$ is a helical edge in $H$.
Suppose that, of the two equatorial squares adjacent to $e$, at least one contains a toggle square.
Then the upper junctions immediately adjacent to $e$ are connected by a branch of $\tau^H$.
Similarly, the lower junctions are connected by a branch of $\tau_H$.
\item
\label{Itm:OneCusp}
Every component of $H - \tau^H$ contains exactly one track-cusp, and exactly one ideal vertex of $U$.
The same holds for $H - \tau_H$.
\end{enumerate}
Moreover, if $U$ is a blue shearing region then we have the following.
\begin{enumerate}[resume]
\item
\label{Itm:Slopes}
Outside of toggle squares and junctions, the branches of $\tau^H$ have strictly positive slope (in bigon coordinates on $H$) and the branches of $\tau_H$ have strictly negative slope.
\item
\label{Itm:ToggleSquareSlopes}
Inside of toggle squares, outside of junctions, there is exactly one branch of $\tau^H$ and exactly one branch of $\tau_H$.
If $H = \bdy^+ U$ then these branches have strictly positive slope with the slope of $\tau^H$ more positive than that of $\tau_H$.
If instead $H = \bdy^- U$ then these branches have strictly negative slope with the slope of $\tau_H$ more negative than that of $\tau^H$.
\end{enumerate}
When $U$ is a red shearing region, similar statements hold, swapping the signs of slopes.
\end{remark}

In our figures, some branches of the draped train-tracks appear to be vertical, contradicting \refrem{BoundaryCrossSectionDraped}\refitm{Slopes} and \refitm{ToggleSquareSlopes}.
However, due to the perturbation of \refrem{NiceBigonCoords}\refitm{Station} and our choice of junction radius (\refdef{JunctionRadius}), these branches in fact have finite slopes.

We generalise \refrem{BoundaryCrossSectionDraped}\refitm{OneCusp} to other cross-sections as follows.

\begin{proposition}
\label{Prop:OneCusp}
Suppose that $U$ is a crimped shearing region.
Let $H$ be a cross-section of $U$.
Then every component of $H - \tau^H$ contains exactly one track-cusp and exactly one ideal vertex of $U$.
The same holds for $H - \tau_H$.
\end{proposition}

\begin{proof}
The result holds for $H' = \bdy^- U$ by \refrem{BoundaryCrossSectionDraped}\refitm{OneCusp}.
Moving upwards from $H'$ to $H$ we perform splittings and graphical isotopies.
Neither of these changes the combinatorics of a region of $H - \tau^H$.
\end{proof}

From Remarks~\ref{Rem:TrackCuspNoGo} and~\ref{Rem:BoundaryCrossSectionDraped} we deduce the following.

\begin{corollary}
\label{Cor:MidCrossSectionDraped}
Suppose that $B^\calV$ and $B_\calV$ are in draped position.
Suppose that $U$ is a blue shearing region.
Suppose that $H$ is the lower boundary of $\Theta^U$ (which equals the upper boundary of $\Theta_U$).
\begin{enumerate}
\item
\label{Itm:MidTrackCusp}
Each track-cusp of $\tau^H$ and $\tau_H$ is in a junction.
\item
\label{Itm:MidStraight}
Outside of the junctions, the branches of $\tau^H$ and $\tau_H$ are line segments (in bigon coordinates).
\item
\label{Itm:MidSlopes}
Outside of the junctions, the branches of $\tau^H$ have strictly positive slope and the branches of $\tau_H$ have strictly negative slope.
\end{enumerate}
Finally, all of the above again holds, swapping slopes appropriately, when $U$ is a red shearing region. \qed
\end{corollary}


\begin{lemma}
\label{Lem:SlopeNearCusp}
Suppose that $B^\calV$ and $B_\calV$ are in draped position.
Suppose that $U$ is a blue shearing region.
Suppose that $H$ is any cross-section of $U$.
Let $c$ be a cusp of $U$.
Let $E$ be the component of $H - (\tau^H \cup \tau_H)$ meeting $c$.
Then (outside of junctions) the branches of $\tau^H$ appearing in the boundary of $E$ have positive slope;
the branches of $\tau_H$ appearing in the boundary of $E$ have negative slope.
There is a similar statement for a red shearing region.
\end{lemma}

\begin{proof}
We begin by considering branches of the upper train-track in $\bdy^- U$.
By \refrem{BoundaryCrossSectionDraped}\refitm{Slopes} and \refitm{ToggleSquareSlopes},
the only branches of the incorrect slope are in toggle squares.
Appealing to \refrem{BoundaryCrossSectionDraped}\refitm{NextToToggleSquare},
such branches are separated from the cusp $c$ by other branches.

We now move upwards from $\bdy^- U = \bdy^- \Theta_U$ to $ \bdy^+ \Theta_U$.
Splitting along draping routes does not change the slopes of branches appearing in the boundary of $E$.
The graphical isotopy does change slopes.
Positive slopes remain positive, while any negative slopes (above toggle squares) become positive.
Moreover, branches with negative slope never make up part of the boundary of $E$ because the graphical isotopies follow the lower foliation.
\end{proof}

\begin{lemma}
\label{Lem:SplitMeetsToggles}
Each subray of each branch line of $B^\calV$ and of $B_\calV$, in draped position, meets crimped shearing regions of both colours.
\end{lemma}

\begin{proof}
This follows from \reflem{DualMeetsToggles} and the fact that our isotopies do not change combinatorics in toggle squares.
\end{proof}

\section{The dynamic pair}
\label{Sec:DynamicPair}

\begin{theorem}
\label{Thm:DynamicPair}
Suppose that $\calV$ is a transverse veering triangulation.
In draped position, the upper and lower branched surfaces $B^\calV$ and $B_\calV$ form a dynamic pair;
this position is canonical.
Furthermore, if $\calV$ is finite then draped position is produced algorithmically in polynomial time.
Finally, the dynamic train-track $B^\calV \cap B_\calV$ has at most a quadratic number of edges.
\end{theorem}

The branched surfaces $B^\calV$ and $B_\calV$ are individually dynamic by \reflem{DrapedDynamic}.
We now verify the hypotheses of \refdef{DynamicPair}.
Again, it will be convenient to work equivariantly in the universal cover.

\subsection{Transversality}

Let $U$ be a crimped shearing region.

\begin{lemma}
Suppose that $H$ is a cross-section of $U$.
Then the train-tracks $\tau^H$ and $\tau_H$ are transverse.
\end{lemma}

\begin{proof}
Let $H_s$, for $s \in [0,1/2]$, be the cross-sections of $\Theta_U$.
Let $\tau^s = B^\calV \cap H_s$ and let  $\tau_s = B_\calV \cap H_s$.
The train-tracks $\tau^0$ and $\tau_0$ are transverse by \refrem{BoundaryCrossSectionDraped}\refitm{TrackCusp},~\refitm{Slopes}, and~\refitm{ToggleSquareSlopes}, as well as \refrem{KeepYourHeadDown}.
As $s$ increases, the train-tracks $\tau^s$ perform the neighbourhood and then graphical isotopies as described in \refsec{DrapingIsotopyLower}.
Also, by \refrem{TrackCuspNoGo}, the train-tracks $\tau_s$ are all essentially the same in bigon coordinates.

During the splitting (that is, for $s \in [0,1/4]$), the track-cusps of $\tau^s$ split forward in a small neighbourhood of (the projection of) $\tau^0$.
Thus the train-tracks $\tau^s$ and $\tau_s$ are transverse for $s \in [0,1/4]$.

By \refcor{MidCrossSectionDraped}, the train-tracks $\tau^{1/2}$ and $\tau_{1/2}$ are transverse.
We now consider $s \in [1/4,1/2]$.
The graphical isotopy interpolates between $\tau^{1/4}$ and $\tau^{1/2}$.
Let $b^s$ and $c_s$ be branches of $\tau^s$ and $\tau_s$ respectively.

\begin{claim*}
The branches $b^s$ and $c_s$ are transverse.
\end{claim*}

\begin{proof}
By \refrem{TrackCuspNoGo}, in $\Theta_U$ the track-cusps of $\tau_H$ lie within junctions.
The lower graphical isotopy leaves upper track-cusps in junctions, and the upper and lower junctions are disjoint.
Thus the endpoints of $b^s$ and $c_s$ are disjoint.

Let $c_0$ be the projection of $c_s$ down to $\bdy^- U$.
Suppose that $c_0$ lies completely within a toggle square.
If the projection of $b^{1/2}$ misses this toggle square then we are done.
Otherwise let $a^s$ be the linear segment of $b^s$ whose projection meets the toggle square.
Since the isotopy is lower graphical, the slope of $a^s$ is between that of $a^{1/4}$ and $a^{1/2}$.
Applying \refrem{BoundaryCrossSectionDraped}\refitm{ToggleSquareSlopes} and \refcor{MidCrossSectionDraped}\refitm{MidSlopes} we find that the slope of $c_s$ is more negative than that of $a^s$.
We deduce that $c_s$ is transverse to $a^s$ and thus to $b^s$.

Suppose instead that $c_0$ is disjoint from the toggle squares.
In this case the proof is similar, but easier.
Now the slope of $c_s$ is always negative by \refcor{MidCrossSectionDraped}\refitm{MidSlopes}.
Also, the slope of $a^s$ is always positive by \refrem{BoundaryCrossSectionDraped}\refitm{Slopes}, by \refcor{MidCrossSectionDraped}\refitm{MidSlopes}, and by appealing to the lower graphical isotopy.
\end{proof}

Swapping the roles of upper and lower and repeating the argument shows that $\tau^H$ and $\tau_H$ are transverse for every cross-section $H$ in $\Theta^U$.
\end{proof}

This lemma proves that $B^\calV$ and $B_\calV$ are transverse.
The proof above also implies the following.

\begin{remark}
\label{Rem:TrackCuspGo}
As $H$ moves upwards, if a track-cusp of $\tau^H$ moves through $\tau_H$, it does so going forwards.
Similarly, whenever a track-cusp of $\tau_H$ moves through $\tau^H$, it does so going backwards.
\end{remark}

\subsection{Components}
\label{Sec:Components}

We must show that every component $C$ of $M - (B^\calV \cup B_\calV)$ is either a dynamic shell (\refdef{DynamicShell}) or a pinched tetrahedron (\refdef{PinchedTetrahedron}).

\subsubsection{Dynamic shell}

Suppose that $C$ contains one (thus by \refprop{OneCusp}, exactly one) cusp $c$ of $M$.
Let $v$ be a model of $c$ where $v$ is an ideal vertex of a red crimped shearing region $U$.
Let $E = E(v, U)$ be the component of $U - (B^\calV \cup B_\calV)$ incident to $v$.  Our goal now is to prove the following.
\begin{itemize}
\item $E$ is a three-ball,
\item the frontier of $E$ in $U$ consists of two vertical ``half-bigons'' (one from each of $B^\calV$ and $B_\calV$),
\item the boundary of $E$ in $\bdy^+ {U}$ consists of a triangle, intersecting a single helical edge of $\bdy^+ U$, and
\item the boundary of $E$ in $\bdy^- {U}$ consists of a triangle, intersecting a single helical edge of $\bdy^- U$.
\end{itemize}

Fix a cross-section $H$ of $U$.
Looking into  $H$ from the vertex $v$, we see a branch of $\tau_H$ ending on the boundary of $H$ to our left and a branch of $\tau^H$ ending on the boundary of $H$ to our right.
Appealing to \reflem{SlopeNearCusp}, the frontier of $H \cap E$ lies in a train route of $\tau^H$ and a train route of $\tau_H$.
For each $H$, these routes intersect precisely once.
Stacking the cross-sections together, the routes form the desired half-bigons.
This shows that $E$ is a three-ball with the desired properties.

A similar argument applies for a blue shearing region $U$.
Here the half-bigon of $B^\calV$ is to the left, and the half-bigon of $B_\calV$ is to the right.

Taking the union of the three-balls $E(v,U)$, as $v$ ranges over the models of $c$, gives $C$.
The half-bigons glue to give the stable and unstable faces of $C$.
Note that any one half-bigon meets only finitely many others because the edges of $\calV$ have finite degrees.
Therefore, the $E(v,U)$ glue together to form a dynamic shell.

\subsubsection{Pinched tetrahedron}

Before obtaining our pinched tetrahedra we need the following definition and lemma.

\begin{definition}
Suppose that $H \subset U$ is a cross-section.
We define the \emph{bigon extension} $\bigonx{H}$ as follows.
The boundary of $H$ consists of longitudinal crimped edges.
Each such edge $e$ cobounds a crimped bigon $B$ with a veering edge $e'$.
We obtain $\bigonx{H}$ by gluing each such crimped bigon $B$ to $H$ along its edge $e$.
\end{definition}

We now suppose that $C$ is a component of $M - (B^\calV \cup B_\calV)$ which does not contain any cusp $c$ of $M$.

\begin{lemma}
\label{Lem:PinchedTetrahedraCrossSections}
Suppose that $H$ is a cross-section of a crimped shearing region $U$, meeting $C$.
Then the intersection $C \cap \bigonx{H}$ is either
a trigon or a quadragon, as defined in \refdef{LifeAndDeath}.
Moreover, as $H$ moves up through $U$, components change according to the sequence given in \refdef{LifeAndDeath}.
\end{lemma}


\begin{proof}
Suppose that $U$ is a red crimped shearing region.
Let $H_s$ for $s \in [0,1]$ be the cross-sections of $U$.
Thus $H_0 = \bdy^- U$.

\begin{claim*}
Let $\textsc{r} = C \cap \bigonx{H}_0$.
Then $\textsc{r}$ is either a trigon or a quadragon.
\end{claim*}

\begin{proof}
First suppose that $\textsc{r}$ is entirely contained within $H = H_0$.
Appealing to \refrem{BoundaryCrossSectionDraped}, the boundary of $\textsc{r}$ is a piecewise loop in $H$ with either three or four corners.
(Outside of toggle squares, the two train-tracks have opposite signs so there cannot be more than four corners. Inside of toggle squares the combinatorics is standardised.)
These corners are either at transverse intersections between $\tau^H$ and $\tau_H$ or at track-cusps.
If there are four corners then they are all transverse intersections between $\tau^H$ and $\tau_H$ and $\textsc{r}$ is a quadragon.
If there are three corners then two are transverse intersections and one is a track-cusp.
Thus $\textsc{r}$ is a trigon.

Now suppose that $\textsc{r}$ is not entirely contained within $H$.
By \refrem{BoundaryCrossSectionDraped}\refitm{NextToToggleSquare}, the component $\textsc{r}$ meets a crimped bigon $B$ and contains the midpoint of the crimped edge.
The frontier of $\textsc{r}$ in $B$ consists of exactly one arc from each of $\tau^B$ and $\tau_B$, meeting at a point.
The claim now follows in a manner similar to the previous paragraph.
\end{proof}

More generally, suppose that the claim holds with $H_s$ replacing $H_0$.
Let $\tau^s = H_s \cap B^\calV$ (green) and $\tau_s = H_s \cap B_\calV$ (purple).
\refrem{TrackCuspGo} tells us that as $s$ increases, there are only two combinatorial changes:
\begin{enumerate}
\item
Track-cusps of $\tau^s$ move forwards through branches of $\tau_s$.
\item
Track-cusps of $\tau_s$ move backwards through branches of $\tau^s$.
\end{enumerate}
The first move simultaneously creates a new green trigon and converts a green trigon into a quadragon. The second move simultaneously deletes a purple trigon, and converts a quadragon into a purple trigon.  These are both moves between stages in the life of a pinched tetrahedron, as given in \refdef{LifeAndDeath}, as required.  This proves \reflem{PinchedTetrahedraCrossSections}.
\end{proof}

Equipped with this lemma we now prove that $C$ is a pinched tetrahedron.
Let $H$ be a cross-section through a crimped shearing region $U$.
Let \textsc{r} be a region of $\bigonx{H} - (B^\calV \cup B_\calV)$.
Let $S^\textsc{r}$ be the component of $\bigonx{H}-\tau^H$ containing $\textsc{r}$.
Let $S_\textsc{r}$ be the component of $\bigonx{H}-\tau_H$ containing $\textsc{r}$.
Thus $\textsc{r} = S_\textsc{r} \cap S_\textsc{r}$.
Using \refprop{OneCusp} twice, gives a pair of track-cusps $s^\textsc{r} \subset S^\textsc{r}$ and $s_\textsc{r} \subset S_\textsc{r}$.

First suppose that \textsc{r} is a green trigon.
Thus $\textsc{r}$ contains $s^\textsc{r}$.
We must show that this track-cusp eventually crosses a purple arc, turning $\textsc{r}$ into a quadragon.
By \reflem{SplitMeetsToggles}, moving up, (the branch line containing) $s^\textsc{r}$ eventually enters the bottom of a crimped shearing region $V$ through a toggle square.
If the region \textsc{r} persists into $\bdy^- V$, and is still a green trigon, then
moving up through $\Theta_V$, the track-cusp $s^\textsc{r}$ splits forwards and hits the purple arc given by \refrem{BoundaryCrossSectionDraped}\refitm{NextToToggleSquare}.
This turns \textsc{r} into a quadragon.

Moving down instead of up, a similar argument shows that every green trigon is born at some point.
Similar arguments also show that as we move up purple trigons eventually die, and that as we move down, purple trigons eventually turn into quadragons.

Lastly we must show that no quadragon can remain a quadragon forever.
Suppose now that \textsc{r} is a quadragon in a cross-section $H$.
As we move down, (the branch line containing) $s^\textsc{r}$ is eventually inside a toggle square within a cross-section $K = \bdy^- U$.
Using \refrem{BoundaryCrossSectionDraped}\refitm{NextToToggleSquare}, we observe that the component of $\bigonx{K} - B^\calV$ containing $s^\textsc{r}$ has no quadragons.
Therefore for some cross-section below $H$ and above $K$, the region $\textsc{r}$ became a trigon.
A similar argument shows that quadragons must eventually become trigons as we move upwards.

This completes the proof that components of $M - (B^\calV\cup B_\calV)$ are either dynamic shells or pinched tetrahedra.

\subsection{Transience}

Suppose that $F$ is a component of $B_\calV - B^\calV$.
Let $\alpha$ be an upwards ray in $F$, in the sense of \refdef{Upwards}.
Let $x$ be the initial point of $\alpha$.
Let $U$ be a crimped shearing region containing $x$, and let $H$ be the cross-section of $U$ containing $x$.
\refprop{OneCusp} implies that there is one ideal vertex $v$ of $U$ in the component of $H - \tau^H$ containing $x$.
Let $c$ be the cusp of $M$ containing $v$.
By \refsec{Components}, there is a unique dynamic shell $C$ containing $c$.

Let $(\textsc{r}_i)$ be the (necessarily finite) collection of components of $H - (\tau^H \cup \tau_H)$ which separate, in $H - \tau^H$, the region $C \cap H$ from the point $x$.
Let $f \subset F \cap H$ be the component containing $x$.
As we flow upwards through cross-sections $H_s$, even when we move from one crimped shearing region to the next, each of the $\textsc{r}_i$ evolves according to \refdef{LifeAndDeath}.
Thus they all eventually collapse.
Moreover, by \refrem{TrackCuspGo}, no new regions are created between $\alpha \cap H_s$ and $C \cap H_s$.
So $\alpha$ eventually flows into an unstable face of $C$.
The same argument applies to components of $B^\calV - B_\calV$, flowing downwards.

\subsection{Separation}



Recall from \refrem{Dual} that both $B^\calV$ and $B_\calV$ are isotopic (ignoring the branching structure) to the dual two-skeleton of $\calV$.
Suppose that $C$ and $D$ are components of $M - (B^\calV \cup B_\calV)$, each containing a cusp of $M$.
Thus, by \refprop{OneCusp}, each of $C$ and $D$ contains exactly one cusp of $M$.
Suppose that $F$ is a two-cell of the natural cell structure on $B^\calV \cup B_\calV$.
Suppose that $F$ meets $C$ on one side and $D$ on the other.
Then we can find a proper arc dual to $F$, and thus disjoint from one of $B^\calV$ or $B_\calV$.
This is a contradiction.

\subsection{Canonicity and complexity}

The construction does not make any arbitrary choices so draped position is canonical.
In particular, if we reverse the orientation of the manifold or the coorientation of the faces of the triangulation (or both) then only labels change;
the underlying combinatorics of the dynamic pair remains the same.

Now suppose that $\calV$ is a finite transverse veering triangulation.
Let $|\calV|$ denote the number of veering tetrahedra.
In building the shearing decomposition (\refthm{ShearingDecomposition}), we produce $2|\calV|$ half-tetrahedra and perform $2|\calV|$ gluings.
This requires linear time.
In producing the crimped shearing decomposition (\refsec{Crimping}), the work is now proportional to the sum of the edge degrees, which is $6|\calV|$.
This again requires linear time.

To specify the draped positions of $B^\calV$ and $B_\calV$, it suffices to determine the position of every track-cusp $c$ in each horizontal cross-section $H$ appearing in the $\Theta$--decomposition of every crimped shearing region $U$.
The branch intervals of $B^U$ lie in junctions except, possibly, in the lower half of $\Theta_U$.
Taking $H = \bdy^- U$, and supposing that the track-cusp $c$ lies in a toggle square, we find that $c$ splits forward in the (space) splitting described in \refsec{DrapingIsotopyLower}.
The path of $c$ is exactly the train route $\beta(c)$ described in \refsec{DrapingIsotopyMiddleUpper}.
The naive algorithm to produce the route takes time at most quadratic in the heights and widths of various junctions.
Thus the total time required to compute the $\beta$ routes is polynomial.

We now bound the number of edges in the dynamic train-track $B^\calV \cap B_\calV$.
Note that the dynamic train-track is disjoint from the junctions (\refrem{KeepYourHeadDown}).
Suppose that $(U_i)_{i = 1}^{m}$ is a collection of blue crimped shearing regions with the following properties.
\begin{enumerate}
\item
$U = U_1$ has at least one toggle square in $\bdy^- U$.
\item
$V= U_{m}$ has at least one toggle square in $\bdy^+ V$.
\item
For $i = 1, 2, \ldots, m - 1$,
the upper boundary of $U_i$ equals the lower boundary of $U_{i+1}$.
\item There are no toggle squares in this shared cross-section.
\item The length of $U$, and thus of all of the $U_i$, is $n$.
\end{enumerate}
We allow $m$ to be one (and thus $U = V$).
We also allow $n$ to be one.

Let $H = \bdy^+ \Theta_U$ be the middle cross-section of $U$.
The track $\tau_H$ has $2n$ branches (outside of the junctions).
Each of these branches is a line segment in $H$.
Since the draping isotopy is fixed on toggle squares, consulting \reffig{MagnifyPartedBottomToggle} and recalling \refrem{TrackCuspNoGo}, each branch of $\tau_H$ above a toggle square has projection to $\bdy^- U$ contained within that toggle square.
The remaining branches of $\tau_H$ have projections that avoid the toggle squares.
Thus no branch of $\tau_H$ wraps all the way around $H$.
Thus their slopes are more negative than $-1/n$ (times a universal factor of $\sqrt{3}/2$, which we ignore).

Let $K = \bdy^- \Theta^V$ be the middle cross-section of $V$.
By a similar argument, $\tau^K$ has $2n$ branches (outside of the junctions).
Again, each is a line segment in $K$.
Furthermore, all of these are either below toggle squares or have slope greater than $1/n$.
Since there are no toggle squares between $U$ and $V$ the track $\tau^H$ is obtained from $\tau^K$ by shearing one unit, $m$ times.
Thus the branches of $\tau^H$ have slope greater than $1/(m + n)$.
Therefore any branch of $\tau^H$ wraps at most $(m+n)/n$ times around $H$.
It follows that each branch of $\tau^H$ meets each branch of $\tau_H$ at most $((m + n)/n) + 1$ times.
There are $(2n)^2$ such pairs, for a total of at most $4n(m+2n)$ intersections.
This counts all edges of the dynamic train-track above $H$ and below $K$.
Edges of the dynamic train-track either continue or merge in pairs as we descend from
$H$ to $\bdy^- U$.
Thus there are at most an additional $4n(m+2n)$ edges in $\Theta_U$.
Likewise there are at most an additional $4n(m+2n)$ edges in $\Theta^V$.

There are now two cases.
If $m \geq n$ then the size of the dynamic train-track in $\cup_i U_i$ is $O(nm)$;
this is proportional to the number of half-tetrahedra in $\cup_i U_i$.
If $m \leq n$ then the size is instead $O(n^2)$;
this is bounded above by the square of the number of half-tetrahedra in $\cup_i U_i$.
Summing, we deduce that the size of the dynamic train-track is at most quadratic in $|\calV|$.

This completes the proof of \refthm{DynamicPair}. \qed

\subsection{Dual train tracks}

We now give a consequence of \refthm{DynamicPair}.

\begin{corollary}
\label{Cor:Dual}
There is an algorithm that, given a surface $S$ and a pseudo-Anosov homeomorphism $f \from S \to S$, produces a (canonical) splitting/folding sequence of dual train tracks in $S$ that realise $f$.
\end{corollary}

\begin{proof}
Suppose that $f \from S \to S$ is the given pseudo-Anosov homeomorphism.
Using Flipper~\cite{flipper} we obtain Agol's splitting sequence.
We puncture each complementary region exactly once to obtain $f^\circ \from S^\circ \to S^\circ$.
Let $\calV$ be the resulting layered veering triangulation of $M^\circ$, the mapping torus of $f^\circ$.
Applying \refthm{DynamicPair} we obtain the (canonical) dynamic pair of branched surfaces $B^\calV$ and $B_\calV$.
Note that these are also transverse to the \emph{crimped shearing decomposition} (\refdef{CrimpedShearingDecomposition}) and thus are transverse to the faces of $\calV$.

We now give a concrete realisation of the surface bundle structure on $M^\circ$, as follows.
Agol's splitting sequence gives us integer weights on the two-skeleton of $\calV$.
After crimping, the horizontal branched surface is a union of \emph{crimped bigons} and \emph{crimped triangles} -- see Definitions~\ref{Def:CrimpedBigon} and~\ref{Def:CrimpedTriangle}.
The integer weights give a finite collection of copies of $S^\circ$;
we extend these over the veering tetrahedra to obtained the desired foliation $\calF^\circ$.

Finally, the branch lines of $B^\calV$ and $B_\calV$ may have local maxima and minima with respect to $\calF^\circ$.
We remove the local maxima of the branch lines of $B^\calV$ by splitting them downwards;
we remove the local minima of the branch lines of $B_\calV$ by splitting them upwards.
(Note that there are no combinatorial choices to make here.)
We now fill $M^\circ$ along the longitudes of $\calF^\circ$.
The new dynamic pair in $M$ gives the desired pair of splitting sequences in $S$.
\end{proof}

\subsection{A final question}

The construction of \refthm{DynamicPair} is canonical in that we make no choices along the way.  
However, the dynamic pairs produced may have quadratic size.
That is, there is a sequence $(\calV_k)_{k = 2}^\infty$ of veering triangulations with the following properties.
\begin{itemize}
\item
$\calV_k$ has $k$ tetrahedra.
\item
$\calV_{k+1}$ is obtained from $\calV_k$ by \emph{horizontal veering Dehn surgery} (along a M\"obius band)~\cite{veering_dehn_surgery}.
\item
The size of the dynamic train-track of $\calV_k$ grows quadratically with $k$.
\end{itemize}

\begin{question}
Is there some other canonical (independent of orientation and coorientations) construction of a dynamic pair which yields a dynamic train-track of linear size?
\end{question}

\appendix
\section{From equatorial squares to maximal rectangles}
\label{App:Rectangles}

For our future work, we require an analysis of maximal rectangles in the leaf space for the ``flow'' associated to a given veering triangulation.
We proceed as follows.

Suppose that $M$ is a three-manifold.
Suppose that $\calV$ is a veering triangulation of $M$.
Let $\calU$ be the associated crimped shearing decomposition of $M$, as defined in \refsec{Crimping}.
As usual, we now work in the universal cover.

\begin{definition}
\label{Def:Cross}
Suppose that $t$ is a veering tetrahedron of $\calV$.
Let $E = E(t)$ be its equatorial square.
Let $e_0$, $e_1$, $e_2$, and $e_3$ be the veering edges of $E$.
Recall that $E_\crimp(\calV)$ is the crimped equatorial branched surface (\refdef{CrimpedEquatorial}).
Let $n_i$ be a small regular neighbourhood of $e_i$ taken in $E_\crimp(\calV)$.
Let $s_i = n_i - E$.

Let $U$ and $V$ be the crimped shearing regions above and below $s_i$ respectively.
Let $H_i$ be the component of $\bdy^- U \cap \bdy^+ V$ containing $s_i$.
We define $X = X(t) = E \cup (\cup_i H_i)$ to be the \emph{cross} associated to the tetrahedron $t$.
\end{definition}

\begin{figure}[htbp]
\centering
\includegraphics[width =0.97\textwidth]{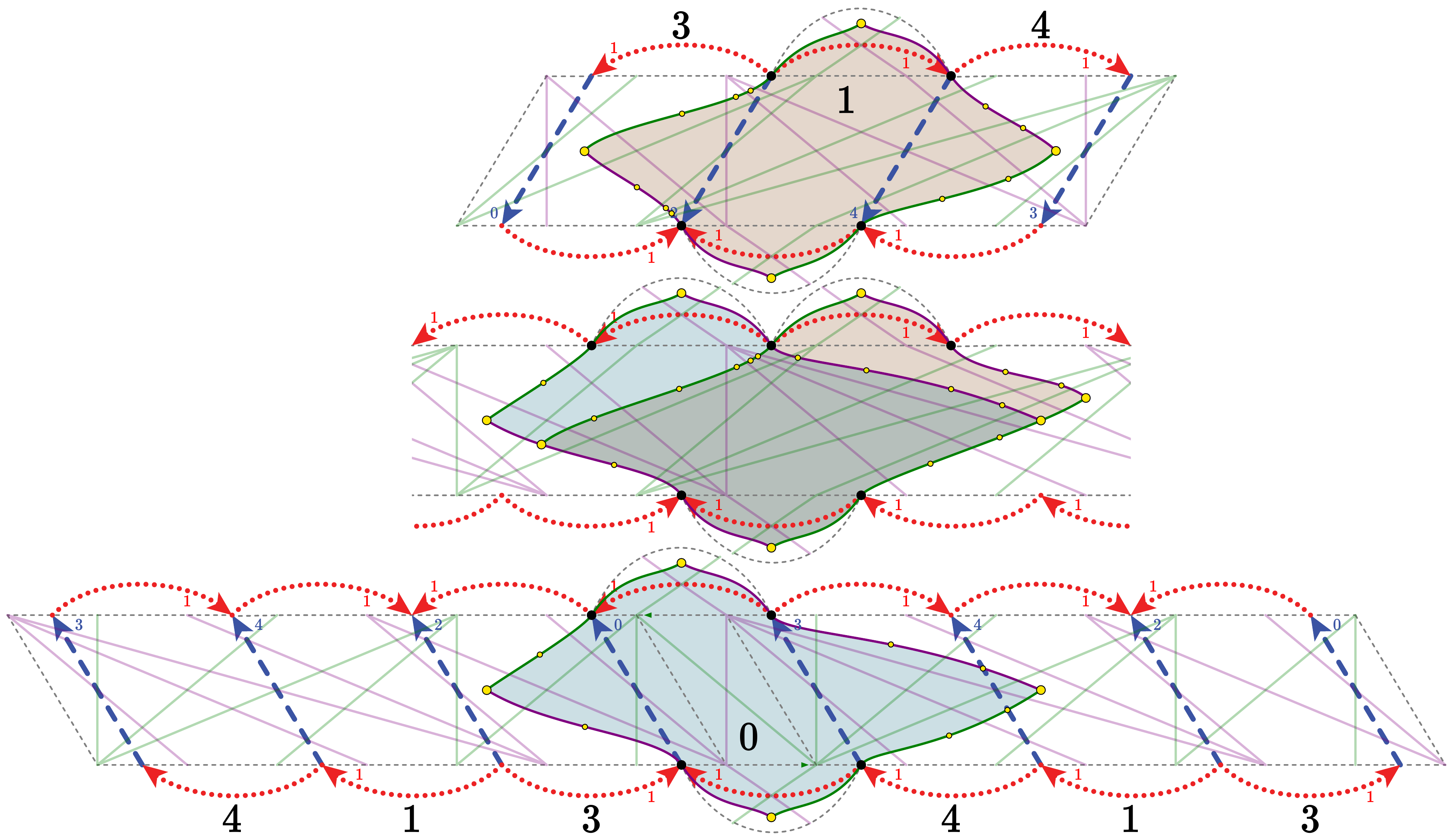}
\caption{
The first row shows the cross for the equatorial square for tetrahedron 1 in \usebox{\BigExVeer}.
The second row shows the T-shape for the unique (in the universal cover) face shared by tetrahedra 1 and 0.
The third row shows the cross for the equatorial square for tetrahedron 0.
The corresponding tetrahedron and face rectangles are shaded.
The vertices and edges of the dual graph are shown only on the boundary of the rectangles.
The cusps are shown with black dots while other regions are indicated with yellow dots.
Corners of the rectangles are drawn with larger yellow dots.
}
\label{Fig:Crosses}
\end{figure}

As usual, we define $\tau^X = X \cap B^\calV$, and similarly define $\tau_X$.
These are train-tracks properly embedded in $X$.
Let $\tau(X) \subset X$ be the graph dual to the union $\tau^X \cup \tau_X$.
In a small abuse, we place vertices of $\tau(X)$, if dual to a cusp region, at the associated cusp.
We call an edge $e'$ of $\tau(X)$ \emph{upper} or \emph{lower} as its dual edge $e$ lies in $\tau_X$ or $\tau^X$ respectively.
A \emph{rectangle} in $X$ is an embedded disk in $X$ whose sides in $\tau(X)$ alternate between upper and lower exactly four times.

\begin{lemma}
\label{Lem:TetRectangle}
There is a unique rectangle $R = R(t)$ in $X = X(t)$ so that $\bdy R$ meets the vertices of $t$ and is disjoint from all junctions.
\end{lemma}

\begin{proof}
Fix an edge $e$ of the equatorial square $E = E(t)$.
Let $c$ and $d$ be the cusps at the two ends of $e$.
Let $Y$ be the component of $X - e$ not containing $E$.
Suppose that the junction immediately adjacent to $c$, in $Y$, intersects $\tau^Y$.
Thus the junction immediately adjacent to $d$, in $Y$, intersects $\tau_Y$.
Let $F$ be the connected component of $Y - \tau_Y$ containing $c$.
Similarly, let $G$ be the connected component of $Y - \tau^Y$ containing $d$.
By \refrem{BoundaryCrossSectionDraped}\refitm{Straight} and~\refitm{Slopes}, the regions $F$ and $G$ intersect in a quadragon.
We deduce that there is a unique path in the dual graph (to $\tau^Y \cup \tau_Y$) from $c$ to $d$ that changes from lower to upper, exactly once, and which avoids junctions.
See \reffig{Crosses}.

Suppose that the junction immediately adjacent to $c$, in $Y$, instead intersects $\tau_Y$.
Then a similar argument finds a path in the dual graph from $c$ to $d$ that changes from upper to lower, exactly once, and which avoids junctions.

Doing the above for all four edges of $E$ gives the boundary of a rectangle $R = R(t)$.
As required, $\bdy R$ meets the vertices of $t$ and avoids junctions.
The uniqueness of $R$ follows from the slope conditions on the branches of $\tau^X$ and $\tau_X$ and the requirement that $\bdy R$ avoids junctions.
\end{proof}

Note that $R(t)$ receives a cellulation from its intersection with $\tau^X$ and $\tau_X$.
We use $R^{(1)}(t)$ to denote the edges of $R(t)$ belonging to $\tau^X$.
Similarly, $R_{(1)}(t)$ denotes the edges of $R(t)$ belonging to $\tau_X$.
We now turn to constructing rectangles for the faces of $\calV$.

\begin{definition}
\label{Def:TShape}
Suppose that $f$ is a veering face of $\calV$.
Let $e_0$, $e_1$, and $e_2$ be its veering edges.
Two of these, say $e_1$ and $e_2$ are the same colour.
Let $c_i$ be the vertex of $f$ opposite $e_i$.
Let $W'$ be the shearing region (in the shearing decomposition), containing $f$.
Let $W$ be the corresponding crimped shearing region.
The edges $e_1$ and $e_2$ are helical in $\bdy W$;
also there is a longitudinal crimped edge $e'_0$ in $\bdy W$ that cobounds a crimped bigon $B$ with $e_0$.
Let $n_0$ be a small regular neighbourhood of $e_0$ taken in $E_\crimp(\calV)$.
Let $s_0 = n_0 - B$.

Let $U$ and $V$ be the crimped shearing regions above and below $s_0$ respectively.
Let $H_0$ be the component of $\bdy^- U \cap \bdy^+ V$ containing $s_0$.
We take $H$ to be the central cross-section of $W$.
We define $T = T(f) = H \cup H_0$ to be the \emph{T-shape} associated to $f$.
\end{definition}

The proof of the following is similar to that of \reflem{TetRectangle}, replacing
\refrem{BoundaryCrossSectionDraped} by \refcor{MidCrossSectionDraped}.

\begin{lemma}
\label{Lem:FaceRectangle}
There is a unique rectangle $R = R(f)$ in $T = T(f)$ so that $\bdy R$ meets the vertices of $f$ and is disjoint from all junctions.  \qed
\end{lemma}

Again, $R(f)$ receives a cellulation from the tracks $\tau^T$ and $\tau_T$.

\begin{proposition}
\label{Prop:Flow}
Suppose that $f$ is a face of $\calV$.
Suppose that $t$ and $t'$ are the tetrahedra of $\calV$ below and above $f$, respectively.
Let $T = T(f)$, let $X = X(t)$, and let $X' = X(t')$.
The upwards combinatorial flow from $R(t) \subset X$ to $T$ takes
\begin{itemize}
\item
distinct cusps to distinct cusps;
\item
vertices to vertices,
\item
edges of $R^{(1)}(t)$ to edges of $T^{(1)}$,
\item
edges of $R_{(1)}(t)$ to edge-paths of $T_{(1)}$, and
\item
two-cells of $R(t)$ to unions of two-cells of $T$.
\end{itemize}
There is a similar statement for the downwards combinatorial flow from $R(t') \subset X'$ to $T$.
The images of $R(t)$ and $R(t')$ in $T$ have intersection exactly $R(f)$.
\qed
\end{proposition}


One example of \refprop{Flow} is shown in \reffig{Crosses}.



\renewcommand{\UrlFont}{\tiny\ttfamily}
\renewcommand\hrefdefaultfont{\tiny\ttfamily}
\bibliographystyle{../../plainurl}
\bibliography{../../bibfile.bib}
\end{document}